\documentclass[a4paper, 12pt]{article}

\usepackage[sort&compress]{natbib}
\bibpunct{(}{)}{;}{a}{}{,} 

\usepackage{amsthm, amsmath, amssymb, mathrsfs, multirow, url, subfigure}
\usepackage{graphicx} 
\usepackage{ifthen} 
\usepackage{amsfonts}
\usepackage[usenames]{color}
\usepackage{fullpage}
\usepackage{tikz}

\theoremstyle{plain} 
\newtheorem{thm}{Theorem}
\newtheorem{cor}{Corollary}
\newtheorem{prop}{Proposition}

\newtheorem*{lem0}{Lemma}
\newtheorem*{thm0}{Theorem}

\theoremstyle{definition}
\newtheorem{defn}{Definition}

\theoremstyle{remark}

\newtheorem*{remark0}{Remark}

\newcommand{\prob}{\mathsf{P}} 
\newcommand{\E}{\mathsf{E}}

\newcommand{\bin}{{\sf Bin}}
\newcommand{\unif}{{\sf Unif}}
\newcommand{\nm}{{\sf N}}

\newcommand{\chisq}{{\sf ChiSq}}

\newcommand{\A}{\mathcal{A}} 
 
\newcommand{\RR}{\mathbb{R}}
 
\newcommand{\U}{\mathcal{U}}

\newcommand{\YY}{\mathbb{Y}}
\newcommand{\UU}{\mathbb{U}}

\newcommand{\TT}{\mathbb{T}}
\newcommand{\T}{\mathcal{T}}

\newcommand{\eps}{\varepsilon}

\newcommand{\prior}{\mathsf{Q}}
\newcommand{\credal}{\mathscr{Q}}
\newcommand{\cred}{\mathscr{C}}

\newcommand{\lPi}{\underline{\Pi}}
\newcommand{\uPi}{\overline{\Pi}}

\newcommand{\pid}{\pi^{\text{\sc d}}}
\newcommand{\pic}{\pi^{\text{\sc c}}}

\newcommand{\lprior}{\underline{\mathsf{Q}}}
\newcommand{\uprior}{\overline{\mathsf{Q}}}

\newcommand{\lprob}{\underline{\mathsf{P}}}
\newcommand{\uprob}{\overline{\mathsf{P}}}

\title{Valid and efficient imprecise-probabilistic inference with partial priors, I. First results}
\author{Ryan Martin\footnote{Department of Statistics, North Carolina State University; {\tt rgmarti3@ncsu.edu}}}
\date{\today}

\begin{document}

\maketitle 

\begin{abstract}  
Between Bayesian and frequentist inference, it's commonly believed that the former is for cases where one has a prior and the latter is for cases where one has no prior.  But the prior/no-prior classification isn't exhaustive, and most real-world applications fit somewhere in between these two extremes.  That neither of the two dominant schools of thought are suited for these applications creates confusion and slows progress.  A key observation here is that ``no prior information'' actually means no prior distribution can be ruled out, so the classically-frequentist context is best characterized as {\em every prior}.  From this perspective, it's now clear that there's an entire spectrum of contexts depending on what, if any, partial prior information is available, with Bayesian (one prior) and frequentist (every prior) on opposite extremes.  This paper ties the two frameworks together by formally treating those cases where only partial prior information is available using the theory of imprecise probability. The end result is a unified framework of (imprecise-probabilistic) statistical inference with a new validity condition that implies both frequentist-style error rate control for derived procedures and Bayesian-style coherence properties, relative to the given partial prior information.  This new theory contains both the  Bayesian and frequentist frameworks as special cases, since they're both valid in this new sense relative to their respective partial priors.  Different constructions of these valid inferential models are considered, and compared based on their efficiency.  

\smallskip

\end{abstract}

\vfill

\pagebreak

\tableofcontents 

\pagebreak

\section{Introduction}
\label{S:intro}

Statistical inference has two dominant schools of thought: {\em Bayesian} and {\em frequentist}.  The fundamental difference between the two is that the former quantifies uncertainty about unknowns in a formal way, using classical/ordinary/precise probability theory, while the latter does so in a less formal way, focusing on procedures that suitably control error rates.  While the disputes between the two sides have now more-or-less ceased, the fact that there are still two distinct theories of statistics remains problematic for the field of statistics and for science more generally.  \citet{fraser2011, fraser2011.rejoinder, fraser.copss} commented on this a number of times (see Section~\ref{SS:big.picture} below), as did \citet{efron2013.ams}:
\begin{quote}
{\em Two contending philosophical parties, the Bayesians and the frequentists, have been vying for supremacy over the past two-and-a-half centuries...
Unlike most philosophical arguments, this one has important practical consequences. The two philosophies represent competing visions of how science progresses and how mathematical thinking assists in that progress.
}
\end{quote}
Rather than choosing a side, Fisher's fiducial argument \citep[e.g.,][]{fisher1933, fisher1935a, zabell1992, seidenfeld1992} aimed to navigate between the two, to offer a different solution to the two-theory problem. 
The consensus, however, is that Fisher's fiducial argument was a bust; but the problem he aimed to solve remains ``the most important unresolved problem in statistical inference'' \citep{efron.cd.discuss}.  Of course, there have been many new attempts, 
including generalized fiducial distributions \citep[e.g.,][]{hannig.review}, confidence distributions \citep[e.g.,][]{schweder.hjort.book}, and other ``data-dependent measures'' \citep[e.g.,][]{belitser.ddm, belitser.nurushev.uq, martin.walker.deb}, but these modern efforts focus mainly on the frequentist properties of their derived procedures, not on probabilistic reasoning and inference.  Procedures having good error rate control properties are important but, at the end of the day, having new ways to construct such procedures isn't going to resolve the two-theory problem.  As I see it, this resolution can only come from expressing what frequentists do in terms compatible with Bayesians' probabilistic reasoning.  This is what the present paper aims to do.  

Along these lines, I've recently been focusing on the construction of data-dependent, imprecise probability distributions with two key features:
\begin{itemize}
\item like a Bayesian posterior, it assigns (lower and upper) probabilities to any assertion about the unknown quantity of interest, but does so without a prior, and 
\vspace{-2mm}
\item the data-dependent (lower and upper) probabilities that users reason with to make inference are calibrated so that their inferences are reliable in a specific sense.  
\end{itemize} 
These data-dependent imprecise probabilities are referred to as {\em inferential models} (IMs), and the property in the second bullet is called {\em validity}; see \citet{imbasics,imbook}, \citet{martin.nonadditive}, and the discussion/references in Section~\ref{S:background} below. 
More recently, I've shown that statistical procedures with control on frequentist error rates correspond to valid IMs \citep{imchar}.  In other words, behind every provably reliable frequentist procedure is a valid IM, a suitably calibrated imprecise probability, based on which probabilistic reasoning can be carried out, just as I think Fisher intended.  

This aligns with what's been known for a long time: frequentists and Bayesians sit on opposite ends of a spectrum.  But this begs the question: {\em what's the spectrum on which these two are the extremes?}  An answer to this question would go a long way towards resolving the two-theory problem, as it would genuinely unify the two extremes under one single framework or perspective.  The key insight is that the frequentist context, which is often understood as one where there is ``no prior,'' is better interpreted as {\em every prior}.  That is, the lack of prior information implies that no prior distribution can be ruled out, so one is effectively considering all possible priors.  That the frequentist theory focuses on cases with a fixed unknown parameter is a consequence of those point-mass priors being the extreme points in the class of all prior distributions.  

If {\em frequentist = every prior} and {\em Bayesian = one prior}, then the aforementioned spectrum must consist of cases where genuine (partial) prior information about the unknowns is available but falls short of a complete prior distribution that the analyst is willing to believe in; see Figure~\ref{fig:diagram}.  This situation must be common in applications---it's just not discussed in the statistics literature because no one knows how to handle it.  There's one important special case of this partial prior information problem that is common in the statistics literature, namely, low-dimensional structural assumptions in high-dimensional inference problems, and I'll discuss this specifically below.  In any case, when only partial prior information is available, what should the data analyst do?  Currently, the two options are to ignore the prior information and carry out a ``frequentist'' analysis, or fill in what's missing to make a full prior distribution and carry out a ``Bayesian'' analysis?  Clearly, neither of these options is fully satisfactory, so something new is needed that apparently neither of the two existing theories can accommodate.  

\begin{figure}[t]
\begin{center}
\fbox{
\begin{tikzpicture}
\draw[red, thick] (-5,0) -- (5,0);
\filldraw[black] (-5, 0) circle (2pt) node[anchor=south]{frequentist};
\filldraw[black] (5, 0) circle (2pt) node[anchor=south]{Bayesian};
\filldraw[black] (-5, 0) circle (0pt) node[anchor=north]{every prior};
\filldraw[black] (5, 0) circle (0pt) node[anchor=north]{one prior};
\filldraw[red] (0, 0) circle (0pt) node[anchor=north]{$\gets$ some priors $\to$};
\end{tikzpicture}
}
\end{center}
\caption{Diagram showing the spectrum with the traditional frequentist and Bayesian frameworks, corresponding to {\em every prior} and {\em one prior}, respectively, at opposite ends, with ``partial priors'' (sets of prior distributions) in the middle.}
\label{fig:diagram}
\end{figure}
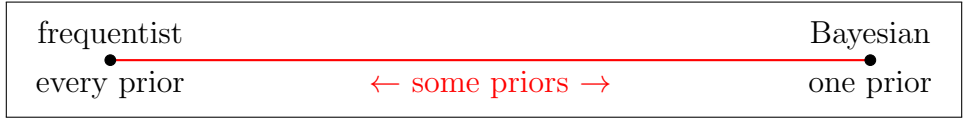

The overarching goal of this series is to develop a new and unified framework in which valid and efficient statistical inference, prediction, decision-making, model assessment, etc., can be carried out while simultaneously accommodating partial prior information.  This optimistic goal will be achieved by leveraging the previously-untapped power of {\em imprecise probability}.  The idea is that, when genuine partial prior information is available---too much to justify ignoring and carrying out a frequentist analysis but too little to justify a single prior distribution and corresponding Bayesian analysis---the new framework should be able to accommodate exactly the information available in the form of an imprecise prior distribution.  This creates opportunities for efficiency gain compared to a frequentist solution while simultaneously avoiding the risk of bias resulting from the strong assumptions that a Bayesian solution requires.  

Making the decision to handle this challenging case of partial prior information using imprecise probability is the easy part.  Valid IMs already work with imprecise probability so this is a pretty obvious idea to consider.  The next step, however, requires answers to two non-trivial questions.  In particular, the primary goal of the present paper is to answer the following two questions. 
\begin{enumerate}
\item {\em What kind of properties would one want the IM that incorporates both data and partial prior information to satisfy?} The original, vacuous-prior IM construction achieves validity and, just like how Bayes estimators are necessarily biased, the incorporation of prior information is sure to ruin validity.  So a new notion of validity is needed that acknowledges the available partial prior information. 
\vspace{-2mm}
\item {\em How to incorporate the prior information in such a way that the aforementioned validity property is satisfied?} There are so many ways this might be achieved, some starting from the basic ingredients (data, prior, etc.)~and others that aim to directly combine the prior information with a suitable vacuous-prior IM.  Which ones ensure validity and, among those, which are most efficient in some sense?  
\end{enumerate}

After a review of the basic construction, interpretation, and properties of vacuous-prior IM in Section~\ref{S:background}, I proceed to address the first question above in Section~\ref{S:partial}.  There I present the new partial-prior-dependent notion of validity, a generalization of vacuous-prior validity, and investigate its consequences.  In particular, I explore the statistical and behavioral consequences of this new notion of validity.  First, on the statistical side, the relevant consequences are sort of obvious: if the IM is valid for inference proper, then procedures derived from it ought to have the desired frequentist error rate control properties, but now with respect to the assumed partial prior information.  Next, on the behavioral side, I draw new connections between validity and the subjectivist notions of coherence/no-sure-loss, \`a la \citet{definetti1937}, that are important to Bayesians and the logic of probabilistic reasoning more generally.  That is, I show that my new validity property implies, among other things, that the IM's output avoids sure-loss in the sense of \citet{walley1991}, \citet{lower.previsions.book}, and \citet{gong.meng.update}. 

Section~\ref{S:achieving} addresses the second question above.  It turns out that there are lots of ways that the new validity property can be achieved and, since this is new territory, I take the liberty of presenting several different options, related to ideas already existing in the imprecise probability literature, including the generalized Bayes approach advocated for in \citet{walley1991} and elsewhere.  A somewhat surprising but important take-away message from Section~\ref{S:achieving} is that an IM that's valid in the original, vacuous-prior sense is also valid in this new, prior-dependent sense.  So, if the previous IM developments already cover the case with partial prior information, then what's the contribution here?  This boils down to a question about {\em why} one might want to incorporate this partial prior information.  The answer to this latter question is clear: the additional structure introduced by the partial prior ought to result in an efficiency gain.  

Section~\ref{S:efficiency} discusses a notion of efficiency and demonstrates empirically how the methods of incorporating prior information presented in Section~\ref{S:achieving} actually make the IM {\em less efficient} than the original vacuous-prior IM.  This implies alternative IM constructions are needed, ones that aim to achieve validity but with an eye towards efficiency gain.  In Section~\ref{S:other}, I consider two very natural alternative constructions---one based on Dempster's rule of combination and another based on a combination rule that I don't believe has a name---and show that these do tend to be more efficient than the other strategies discussed above.  Unfortunately, despite their strong empirical performance in my running example, I'm only able to establish a limited validity-related property for these constructions.  So it remains an open question if these IMs are viable in the sense of being both valid and more efficient than the constructions in Section~\ref{S:achieving}.  

Of course, the strategies above are not the only options and, in Section~\ref{S:validify}, I consider a different strategy: one that starts with an IM that may or may not be valid, and applies a certain transformation to {\em validify} it.  There I prove that a version of the so-called imprecise probability-to-possibility transform \citep{hose2022thesis, dubois.etal.2004} can successfully take an IM that may not be valid and transform it into one that's valid.  Given the desirable properties and empirical performance of the strategies presented in Section~\ref{S:other}, and my inability to prove that they're valid on their own, these make for natural candidates to process through this validification transformation.  Indeed, this makes my favorite of the aforementioned not-yet-provably-valid IMs provably valid and I demonstrate empirically how the validification preserves the efficiency gain.  

My primary motivation for considering the incorporation of partial prior information was the high-dimensional problems mentioned briefly above.  A classical example is inference on the mean vector of a multivariate normal distribution that's assumed to have a certain  low-dimensional structure, e.g., {\em sparsity}.  Structural assumptions like this are frequently encountered in the statistical literature and the go-to strategy is to introduce a penalty that shrinks a fully data-driven estimator towards parameter values that meet the structural constraint \citep[e.g.,][]{hastie.tibshirani.friedman.2009}.  It's often pointed out that the penalty resembles or has an effect similar to that of a prior distribution.  Alternatively, the penalty provides a data-free ranking of how plausible different values of the feature parameter are relative to the posited low-dimensional structure.  There might be lots of prior distributions that provide the same ``plausibility ranking'' and, therefore, it's arguably more natural to describe this as partial prior information.  From this perspective, the developments in the present paper shed some light on the probabilistic reasoning-based principles behind the introduction of penalties to impose structural constraints in high-dimensional inference problems.  A thorough investigation into these high-dimensional problems is beyond the scope of the present paper but, in Section~\ref{S:sparsity}, I explain the fundamental need for a new notion of validity in these problems, I briefly sketch a possible formulation of the partial prior information in this sparsity case, and I show a simple numerical illustration to highlight how one of the proposed solutions here might also be suitable for valid and efficient inference in high-dimensional problems.  

Finally, in Section~\ref{S:discuss}, I give some general concluding remarks, some open questions, and some specific recommendations for next steps.  Part~II of the series will develop the {\em validification} strategy presented in Section~\ref{S:validify} below more formally into a general framework for valid and efficient imprecise-probabilistic inference that accommodates partial prior information when available.

\section{Background}
\label{S:background}

\subsection{Setup and notation}

Let $Y$ denote observable data taking values in a sample space $\YY$; note that the sample space is general, so the data could be a vector, a matrix, etc.  There is an unknown quantity or parameter, denoted by $\theta$, taking values in a general parameter space $\TT$, about which the data $Y$ is believed to be informative.  This relationship is described in a probabilistic manner, so that the distribution of $Y$ depends on $\theta$.  With one exception (see Section~\ref{S:partial} below), all the distributions in this paper will be denoted by $\prob$, with subscripts to indicate which quantity is random.  As a first instance, let $\prob_{Y|\theta}$ denote the conditional distribution of $Y$, given $\theta$; that is, $B \mapsto \prob_{Y|\theta}(B)$ is a probability measure on $\YY$ for each $\theta$ and $\theta \mapsto \prob_{Y|\theta}(B)$ is measurable for each $B$.  Examples include: $Y$ is a scalar count and $\prob_{Y|\theta}$ is a binomial distribution with a fixed number of trials $n$ and unknown success probability $\theta \in [0,1]$; $Y$ is a vector of $n$ independent and identically distributed (iid) samples from a normal distribution, $\prob_{Y|\theta}$, where $\theta=(\mu,\sigma^2)$ denotes the unknown mean and variance parameters; $Y$ is a collection of $n$ iid $d$-vectors with a Gaussian graphical model $\prob_{Y|\theta}$, where $\theta$ is the unknown positive definite precision matrix; or $Y$ is a vector of $n$ iid samples from a distribution $\prob_{Y|\theta}$ where $\theta$ denotes its density function. In any case, the goal is to quantify uncertainty about $\theta$ based on the observed data $Y=y$.  For the moment, I'm assuming that nothing about $\theta$ is known {\em a priori}, but see Section~\ref{S:partial}.  

By uncertainty quantification, here I mean having a data-dependent, probability-like structure defined on a collection $\A$ of subsets of $\TT$.  In other words, the goal is to use the data, statistical model, etc.~to construct something like a ``posterior probability distribution'' for $\theta$, that assigns (lower and upper) probabilities, or degrees of belief, to various assertions/hypotheses about the unknown $\theta$.  Further details are presented in the following subsections.  To be able to manipulate the resulting  probability-like structure akin to the way probabilities are manipulated, the collection $\A$ of relevant assertions/hypotheses needs to be sufficiently rich, so here I'll take $\A$ to be the Borel $\sigma$-algebra on $\TT$; note, in particular, that $\A$ is closed under complementation and contains all the singletons.  To each $A \in \A$, I will associate an assertion about the unknown $\theta$, i.e., both $A$ and ``$\theta \in A$'' will henceforth be referred to as {\em assertions}.   

Especially in the present case where I'm {\em a priori} ignorant about $\theta$, it's not clear how this uncertainty quantification can proceed.  In the statistics literature, the most common strategy is to proceed in a Bayesian way but with a default/non-informative prior for $\theta$ \citep[e.g.,][]{berger2006}.  Other less familiar approaches include Fisher's fiducial inference, Dempster's generalization \citep{dempster1968a, dempster1968b}, Fraser's structural inference \citep{fraser1968}, Hannig's generalized fiducial inference, and the confidence distributions championed by \citet{xie.singh.2012} and \citet{schweder.hjort.book}.  With the exception of a few cases having especially nice structure, no finite-sample guarantees concerning the operating characteristics of these methods are available, so the modern focus is on large-sample properties.  Their reliance on precise probabilities subjects them to {\em false confidence} \citep{balch.martin.ferson.2017}, which implies that inferences are unreliable---or at least at risk of being unreliable---in a certain sense at each fixed sample size.  The point is that reliable inference can't be guaranteed in general with artificial probabilities.  This risk can be avoided, however, using suitable imprecise probabilities, as I describe below.  


\subsection{Inferential models}

Following \citet{martin.nonadditive}, I define an {\em inferential model} (IM) as a mapping from data $y$, depending on statistical model and potentially other optional inputs (e.g., prior information), to a pair of lower and upper probabilities $(\lPi_y, \uPi_y)$ defined on $\A \subseteq 2^\TT$. The IM's output is the aforementioned probability-like structure to be used for uncertainty quantification and inference.  Mathematically, $\uPi_y$ is required to be a sub-additive capacity, i.e., a monotone set function, normalized to satisfy $\uPi_y(\varnothing) = 0$ and $\uPi_y(\Theta) = 1$, where sub-additivity means $\uPi_y(A \cup A') \leq \uPi_y(A) + \uPi_y(A')$ for all disjoint $A$ and $A'$.  Define the corresponding dual/conjugate lower probability 
\[ \lPi_y(A) = 1-\uPi_y(A^c), \quad A \in \A, \]
which, by sub-additivity, satisfies 
\[ \lPi_y(A) \leq \uPi_y(A), \quad A \in \A. \]
This partially explains the names ``lower'' and ``upper probabilities.'' To ensure that the IM output indeed corresponds to tight lower and upper bounds on a collection of probabilities, I'll assume that the IM output is {\em coherent} in the sense that
\[ \lPi_y(A) = \inf_{\Pi \in \cred(\uPi_y)} \Pi(A) \quad \text{and} \quad \uPi_y(A) = \sup_{\Pi \in \cred(\uPi_y)} \Pi(A), \quad A \subseteq \TT, \quad y \in \YY, \]
where $\cred(\uPi_y) = \{\Pi: \Pi(A) \leq \uPi_y(A) \text{ for all $A$}\}$ is the credal set of (data-dependent) probability measures on $\TT$ dominated by $\uPi_y$.  Finally, since I'll be interested in the distribution of the random variable $\uPi_Y(A)$, for fixed $A$, I'll assume throughout that $y \mapsto \uPi_y(A)$ is Borel measurable for all $A \in \A$.  

There's an obvious question: {\em what do these lower/upper probabilities mean, how should they be interpreted?}  In the imprecise probability literature, a behavioral interpretation is common.  Imagine a situation where, after data $Y=y$ has been observed, the value of $\theta$ will be revealed and any gambles made on the truthfulness/falsity of assertions could be settled.  Then \citet{walley1991} and others suggest the following (subjective/personal) behavioral interpretations of my (data-dependent) lower and upper probabilities:
\begin{align*}
\lPi_y(A) & = \text{my maximum buying price for \$$1(\theta \in A)$, given $y$} \\
\uPi_y(A) & = \text{my minimum selling price for \$$1(\theta \in A)$, given $y$}.
\end{align*}
Here and in what follows, $1(E)$ denotes the indicator function of the event $E$.  The idea is that, if the ``real price'' were smaller than $\lPi_y(A)$ or larger than $\uPi_y(A)$, then I'd buy or sell the gamble, respectively; otherwise, the gamble is too risky, so I'd neither buy nor sell. When $\lPi_y \equiv \uPi_y$, the buying and selling prices coincide and the above description reduces to de~Finetti's classical subjective interpretation of probability. 

In the context of statistical inference, however, the gambling terminology generally isn't used; but see \citet{crane.fpp} and \citet{shafer.betting}.  So, instead of treating the lower and upper probabilities as bounds on prices for bets, I can treat them as my own data-dependent degrees of belief about $\theta$.  That is, small $\uPi_y(A)$ suggests to me that the data $y$ doesn't support the truthfulness of the assertion $A$; similarly, large $\lPi_y(A)$ suggests to me that data $y$ does support the truthfulness of $A$; otherwise, if $\lPi_y(A)$ is small and $\uPi_y(A)$ is large, then data $y$ apparently isn't sufficiently informative to make a definitive judgment about $A$.  Like the gambler who refrains from betting in certain situations, in the latter case it's best/safest if I refrain from making inference on $A$ and consider a less complex assertion and/or collect more informative data.  

But there's still something missing here in terms of interpretation.  Clearly, the lower/upper probabilities can't be treated as ``objective'' or frequency-based in any real sense.  But they're also not ``subjective'' in the usual sense because I don't actually have beliefs about $\theta$ that are being processed and somehow quantified in the IM's output.  The objective versus subjective categorization isn't important: what's essential is that the IM output has real-world meaning.  That is, the numerical values assigned by the IM to assertions ought to affect both my opinion and the opinions of like-minded others about what's probably true and not true.  More precisely, ``$\uPi_y(A)$ is small'' certainly makes me doubt the truthfulness of $A$, but it should also have the same effect on others of like mind, i.e., those who accept the statistical model assumptions, etc.  To achieve this, a connection between the IM's lower/upper probabilities and the real world is needed, and the validity condition described below---a statistical constraint---does just that.  Circling back: what does this say about the interpretation of my IM output?  It's essential that my IM be valid, so that it's real-world relevant, but among those valid IMs, I'll pick the one that's ``best,'' striking a balance between {\em efficiency} (see Section~\ref{S:efficiency}), computational simplicity, etc.  So the numerical values my IM assigns to assertions about $\theta$ are more-or-less determined by the IM's statistical properties, since this is what will make my inferences real-world relevant.  In other words, my beliefs are what they are because my (valid) IM is ``best,''\footnote{This isn't really different from modern Bayesian statistics.  Priors are typically chosen in a ``default'' way and the posterior probabilities are never interpreted directly.  The posterior is used to construct procedures (e.g., estimators, confidence regions, etc.), and these are assessed by proving (frequentist) asymptotic consistency results and/or comparing with existing methods based on (frequentist) simulation studies.  In other words, the Bayesian's posterior probabilities are what they are because the various inputs they selected leads to procedures with good statistical properties.} akin to Lewis's {\em best-system} interpretation \citep{lewis1980, lewis1994}. 


At this point, I should emphasize that I have no formal rules in mind for the construction of the IM.  That is, it need not be based on Bayes or generalized Bayes rules, Dempster's rule of combination, etc.  These formal updating rules are appealing in many respects, but they're not without their difficulties,\footnote{The formal updating rules all involve conditional probability and, while it generally doesn't affect things in practice, there are issues with conditional probabilities not being well-defined.  If conditioning wasn't necessary, then conditional distributions not being well-defined wouldn't be an issue.} so I'll not impose any restrictions at the outset.  My focus is on the basic statistical properties we want the IM output to satisfy, so that its inferences are reliable, and then I'll assess what particular kinds of constructions can achieve it.  Interestingly, it'll be shown in Section~\ref{S:achieving} that the generalized Bayes rule does achieve my proposed notion of reliability, but it does so in a way that's not fully satisfactory.  Then I'll consider some alternatives in Section~\ref{S:other}.  


\subsection{Vacuous-prior validity}
\label{SS:no.prior}

Motivated by the behavioral interpretation, the imprecise probability literature mainly focuses on coherence, or internal rationality.  For data analysts on the front lines, the ones crunching the numbers behind real-world decisions, this internal rationality is crucial.  From the perspective of a statistician who is not on the front lines himself, but is developing methods to be used off-the-shelf by front-line data analysts, there are other considerations.  A front-line data analyst's only reason to use a method developed by someone off the front lines is that they believe it's reliable, that it's ``likely to work'' in some specific sense.  This goes beyond the internal rationality of coherence---lots of things that are coherent won't ``work''---and it's what the off-the-front-lines statistician is concerned with.  This external reliability notion is what I call {\em validity}.  


Here I'll review the notion of validity in the case where there's no prior information available; this will be generalized in Section~\ref{S:partial} below.  Recall that, for each $A \in A$, $\lPi_Y(A)$ and $\uPi_Y(A)$ are random variables as functions of $Y \sim \prob_{Y|\theta}$.  Then the IM's reliability is determined by certain features of the distribution of these random variables.  In particular, based on the interpretation of the IM's output, reliability would be questionable if 
\begin{itemize}
\item $\uPi_Y(A)$ doesn't tend to be large when $A$ is true, or
\vspace{-2mm}
\item $\lPi_Y(A)$ doesn't tend to be small when $A$ is false.
\end{itemize}
And since the IM is intended to provide ``uncertainty quantification'' in a broad sense, this notion of reliability should hold for all $A \in \A$.  So the formal property of validity, stated below, requires exactly this.

\begin{defn}
\label{def:no.prior.valid}
An IM with output $(\lPi_Y, \uPi_Y)$ is {\em valid} (with respect to the vacuous-prior model) if either---and hence both---of the following equivalent conditions hold:
\begin{align}
\sup_{\theta \not\in A} \prob_{Y|\theta}\{\lPi_Y(A) > 1-\alpha\} & \leq \alpha, \quad \text{for all $(\alpha,A) \in [0,1] \times \A$}, \notag \\
\sup_{\theta \in A} \prob_{Y|\theta}\{ \uPi_Y(A) \leq \alpha \} & \leq \alpha, \quad \text{for all $(\alpha,A) \in [0,1] \times \A$}. \label{eq:old.valid}
\end{align}
(The equivalence follows from the duality between $\lPi_Y$ and $\uPi_Y$.)
\end{defn}

To put it differently, recall that the IM's lower/upper probabilities are inherently subjective, but the goal is for them to be meaningful to others of like mind in the real world.  That is, an observation ``$\uPi_Y(A)$ is small'' should make others who are willing to accept the model assumptions, etc., inclined to think that assertion $A$ is probably false.  This can only be achieved, as Fisher explained, through an indirect argument.  According to \eqref{eq:old.valid}, ``$\uPi_Y(A)$ is small'' entails a logical disjunction: either $A$ is false or a small-probability event took place.  Since others of like mind can't be inclined to believe a small-probability event took place, they must conclude that $A$ is probably false.  This goes beyond what \citet[][p.~42]{fisher1973} describes in the context of a simple significance test---I'm focused on probabilistic uncertainty quantification, so Definition~\ref{def:no.prior.valid} concerns properties of the IM's entire lower/upper probability output.  See, also, \citet{balch2012} and, for an investigation into these and other efforts to construct calibrated belief functions for statistical inference and prediction, see \citet{denoeux.li.2018}. 

A sufficient condition for vacuous-prior validity is
\begin{equation}
\label{eq:uniform}
\sup_{\theta \in \TT} \prob_{Y|\theta}\bigl\{ \uPi_Y(\{\theta\}) \leq \alpha \bigr\} \leq \alpha, \quad \alpha \in [0,1].
\end{equation}
That is, the random variable $\uPi_Y(\{\theta\})$ is stochastically no smaller than $\unif(0,1)$ as a function of $Y \sim \prob_{Y|\theta}$, for all $\theta \in \TT$.  That this is sufficient for vacuous-prior validity is an immediate consequence of monotonicity of the upper probability, 
\[ \text{if $A \ni \theta$, then $\uPi_Y(\{\theta\}) \leq \alpha$ implies $\uPi_Y(A) \leq \alpha$}. \]
Property \eqref{eq:uniform} is roughly what \citet{walley2002} calls the {\em fundamental frequentist principle}. The only difference is that Walley's version says ``$\alpha \in [0,\bar\alpha]$,'' for some $\bar\alpha \leq 1$.  

One important consequence of  validity is that statistical procedures---e.g., hypothesis tests and confidence regions---derived from a valid IM have frequentist error rate control guarantees.  This explains Walley's choice of name for this property.  For example, if the goal is to test ``$\theta \in A$,'' then the IM-based test that rejects the hypothesis if and only if $\uPi_Y(A) \leq \alpha$ would control the Type~I error at level $\alpha$.  Furthermore, define the $y$-dependent collection $\{C_\alpha(y): \alpha \in [0,1]\}$ of subsets of $\TT$, 
\begin{equation}
\label{eq:pl.region}
C_\alpha(y) = \bigcap \{A \in \A: \lPi_y(A) > 1-\alpha\}.
\end{equation}
I'll refer to $C_\alpha(y)$ as a $100(1-\alpha)$\% {\em plausibility region}. Then validity of the IM implies that its plausibility regions are also confidence regions, i.e., 
\[ \inf_{\theta \in \TT} \prob_{Y|\theta}\{ C_\alpha(Y) \ni \theta \} \geq 1-\alpha, \quad \alpha \in [0,1]. \]
These desirable statistical properties are exact for the valid IM, while most other frameworks only guarantee such properties approximately, e.g., asymptotically.  

It's important to emphasize that validity is about more than controlling frequentist error rates.  Indeed, fiducial or default-prior Bayes solutions often have the property that credible regions derived from the posterior are (approximate) confidence regions.  However, according to the {\em false confidence theorem} in \citet{balch.martin.ferson.2017}, no IM whose output is countably additive can be valid; see, also, \citet{martin.nonadditive} and \citet{prsa.response}.  So validity requires something beyond what these more familiar statistical frameworks can provide.  A brief description of how a valid IM can be constructed is given below.

\subsection{Achieving vacuous-prior validity}
\label{SS:original.im}

One of the only IM constructions that I'm aware of that achieves vacuous-prior validity is that described first in \citet{imbasics, imcond, immarg} and later summarized in the monograph \citet{imbook}, with some new developments more recently that I'll mention later.  This construction has a familiar starting point, like that of Fisher and Dempster, but departs quickly through a novel incorporation of an imprecise probability that's designed specifically to achieve the vacuous-prior validity property above.  

Start by expressing an observation $Y \sim \prob_{Y|\theta}$ using a ``data-generating equation'' 
\begin{equation}
\label{eq:assoc.new}
Y = a(\theta, U), 
\end{equation}
where $U$ has a known distribution $\prob_U$.  This form is familiar in the context of simulating data from a known distribution.  That is, we can generate data $Y$ from $\prob_{Y|\theta}$, with known $\theta$, by first generating $U$ from distribution $\prob_U$ and then plugging the pair $(\theta,U)$ into the expression \eqref{eq:assoc.new} to get $Y$.  Any statistical model $\prob_{Y|\theta}$ can be represented by an expression like \eqref{eq:assoc.new}, but not uniquely.  This representation is also useful for the inference problem, because it's clear that uncertainty quantification about $\theta$, given an observed $Y=y$, can be recast as uncertainty quantification about the unobserved value, say, $u^\star$, of the auxiliary variable $U$.  Fisher, Dempster, and Fraser would quantify uncertainty about $u^\star$, given $Y=y$, with the same probability distribution $\prob_U$.  However, as Martin and Liu argue, the uncertainty about $u^\star$, after $Y=y$ is observed, is fundamentally different than that for the variable $U$ in \eqref{eq:assoc.new} before $Y$ is observed.  As such, they suggest a different quantification of their uncertainty about $u^\star$, an imprecise probability determined by a suitable random set $\U$, with distribution $\prob_\U$, taking values in the power set of $\UU$. 

The choice of random set $\U$ requires some care but, once chosen, the IM is readily constructed as follows.  For observed data $Y=y$ and the specified $\U$, there's a corresponding random set on the parameter space given by 
\[ \TT_y(\U) = \bigcup_{u \in \U} \{\vartheta \in \TT: y = a(\vartheta, u)\}, \quad \U \sim \prob_\U. \]
If I assume, as Martin and Liu do, that $\TT_y(\U)$ is non-empty with $\prob_\U$-probability~1 for (almost) all $y$,\footnote{When $\TT_y(\U)$ is empty with positive probability, this can often be resolved by modifying the initial formulation in \eqref{eq:assoc.new}.  This is described in Chapters~6--7 in \citet{imbook}.  In other cases, one can condition the conflict away in the spirit of Dempster, or can suitably stretch $\U$ to avoid conflict altogether \citep{leafliu2012}.} then the IM's upper probability is given by 
\[ \uPi_y(A) = \prob_\U\{ \TT_y(\U) \cap A \neq \varnothing\}, \quad A \subseteq \TT. \]
The intuition is that, if $\U$ contains ``plausible'' values of $u^\star$, then $\TT_y(\U)$ contains those parameter values that are equally ``plausible,'' given $Y=y$.  The upper probability is a special type of plausibility function, namely, the type corresponding to the distribution of a random set, which is both continuous and infinitely-alternating \citep[e.g.,][]{molchanov2005}, and the associated lower probability $\lPi_y$ is a belief function \citep[e.g.,][]{shafer1976, kohlas.monney.hints, cuzzolin.book}.  It's clear that this IM is largely determined by the user's choice of $\U \sim \prob_\U$, so that same choice must also determine whether validity holds or not.  For the relatively mild sufficient conditions on $\U \sim \prob_\U$ under which validity holds, see \citet{imbasics, imbook}.  The latter reference also contains lots of examples to illustrate this construction; see, also, \citet{martin.nonadditive}. 

While not necessary for validity, \citet{imbook} recommend the use of a random set $\U$ that's {\em nested}, i.e., it's support is nested in the sense that for any pair of sets in the support, one is a subset of the other.  This recommendation, made formal in Theorem~4.3 of \citet{imbook}, was based largely on considerations of efficiency as discussed in Section~\ref{S:efficiency} below.  If, as is often the case, nested $\U$ implies nested $\TT_y(\U)$, then the corresponding IM is a {\em consonant} belief/plausibility function or, equivalently, a {\em possibility measure} \citep[e.g.,][]{dubois2006, dubois.prade.book, destercke.dubois.2014}.  The key feature of a possibility measure is that the IM is completely determined by its associated {\em plausibility contour} function 
\[ \pi_y(\vartheta) = \prob_\U\{ \TT_y(\U) \ni \vartheta \}, \quad \vartheta \in \TT, \]
which is an ordinary point function, much simpler than a general set function.  The characterization of $\uPi_y$ in terms of the plausibility contour is given by
\[ \uPi_y(A) = \sup_{\vartheta \in A} \pi_y(\vartheta), \quad A \subseteq \TT. \]
As an analogy, in classical probability theory, a density or mass function is integrated over the set $A$ to calculate probabilities; here, upper probabilities $\uPi_y(A)$ are calculated by maximizing the contour function $\pi_y$ over $A$. 

\citet{imchar} recently recognized the importance of consonance for both the IM construction and, more generally, for valid statistical inference.  A slightly more direct construction of a valid (and consonant) IM can be achieved through the arguments in \citet{imposs}; see, also, \citet{hose2022thesis}. This possibility-theoretic perspective on the IM construction fits nicely with the ``generalized'' IM framework first presented in \citet{plausfn, gim} and more recently applied in \citet{immeta, imcens} and \citet{imconformal, impred.isipta, cella.martin.imrisk}. 

\section{Validity with partial priors}
\label{S:partial}

\subsection{Definitions}
\label{SS:def}

The main objective of this paper is to extend the IM construction and the vacuous-prior validity property to the case where partial prior information about $\theta$ is available.  First, I need to describe what I mean by ``partial prior information'' and how it's encoded. 

Suppose prior information about $\theta$ is available in the form of a (closed and convex) credal set $\credal$ of prior distributions $\prior$ on $\TT$.  As is customary, when I'm referring to a $\theta$ that's to be interpreted as a random variable having distribution, say, $\prior$, I'll use the upper-case font $\Theta$. Roughly speaking, the ``size'' of $\credal$ controls the prior information's precision, with $\credal = \{\prior\}$ being the most precise and $\credal = \{\text{all probability distributions}\}$ being the least.  These two extreme credal sets are special: the former is classical Bayes while the latter aligns with the frequentist setup.  

I'm particularly interested here in cases between the two extremes, where the credal set is neither a singleton nor the set of all distribution.  There are lots of examples like this, two of these are presented in Section~\ref{SS:running} and \ref{S:sparsity}.  
The relevant question to be considered in this section is how vacuous-prior validity as described above should be modified to account for the availability of this partial prior information.  A formal definition of validity, more general than that in Definition~\ref{def:no.prior.valid}, is given below, along with its  consequences. 

Before getting into these details, some additional notation is needed.  If $\prob_{Y|\theta}$ is the conditional distribution of $Y$, given $\Theta=\theta$, and $\prior$ is a prior distribution for $\Theta$, let $\prob_{Y|\prior}(\cdot) = \int \prob_{Y|\theta}(\cdot) \, \prior(d\theta)$ denote the corresponding marginal distribution of $Y$.  Moreover, let $\prior(\cdot \mid y)$ denote the conditional distribution of $\Theta$, given $Y=y$, obtained by Bayes's rule applied to prior $\Theta \sim \prior$ and likelihood $\prob_{Y|\theta}$.  This setup defines a collection of joint distributions, indexed by the credal set $\credal$, so let $\uprob_\credal$ denote the upper envelope.  That is, if $E$ is any (appropriately measurable) joint event about $(Y,\Theta)$, then the upper probability $\uprob_\credal(E)$ can be expressed more concretely as 
\begin{align}
\uprob_\credal(E) & = \sup_{\prior \in \credal} \iint 1\{(y,\theta) \in E\} \, \prob_{Y|\theta}(dy) \, \prior(d\theta) \notag \\ 
& = \sup_{\prior \in \credal} \iint 1\{(y,\theta) \in E\} \, \prior(d\theta \mid y) \, \prob_{Y|\prior}(dy). \label{eq:total.prob} 
\end{align}
Similarly, there is a corresponding lower probability, $\lprob_\credal$, with the supremum above replaced with an infimum, but this will rarely be needed in what follows.

When partial prior information is available, it makes sense for this to be incorporated into the IM construction, so that $(\lPi_y, \uPi_y)$ depend on $\credal$ in some way.  But remember that I'm not requiring the IM to be obtained based on any formal rule, e.g., conditioning, for updating prior information about $\theta$ in light of the observation $Y=y$.  These updating rules are certainly candidates for the IM construction, and I'll consider these in more detail below.  Again, my primary objective is to achieve a suitable notion of validity, or reliability, so I don't limit my options concerning the IM construction until after the precise goal has been identified.  So, $(\lPi_{y,\credal},\uPi_{y,\credal})$ is just a suitable pair of capacities depending on data $y$ and partial prior $\credal$.  

As discussed above, the data analyst will use his data-dependent lower and upper probabilities to make judgments about various relevant assertions concerning the unknown $\theta$.  Still, large values of $\lPi_{Y,\credal}(A)$ support the truthfulness of $A$ and small values $\uPi_{Y,\credal}(A)$ support the truthfulness of $A^c$. Consequently, the events 
\[ \{(y,\theta): \uPi_{y,\credal}(A) \leq \alpha, \, \theta \in A\} \quad \text{and} \quad \{(y,\theta): \lPi_{y,\credal}(A) > 1-\alpha, \, \theta \not\in A\} \]
in the joint $(Y,\Theta)$-space are cases when incorrect or erroneous conclusions might be made, so the new definition of validity is designed to ensure that these undesirable events are suitably rare, thus making the IM's uncertainty quantification reliable.  

\begin{defn}
\label{def:valid}
Let $\A' \subseteq \A$ be a sub-collection of assertions.  An IM with output $(\lPi_{Y,\credal}, \uPi_{Y,\credal})$ is {\em $\A'$-valid}, relative to $\credal$, if either---and hence both---of the following equivalent conditions holds:
\begin{align}
\uprob_\credal\bigl\{ \uPi_{Y,\credal}(A) \leq \alpha, \, \Theta \in A \bigr\} & \leq \alpha, \quad \text{for all $(\alpha,A) \in [0,1] \times \A'$}, \label{eq:valid.up} \\
\uprob_\credal\bigl\{ \lPi_{Y,\credal}(A) \geq 1-\alpha, \, \Theta \not\in A \bigr\} & \leq \alpha, \quad \text{for all $(\alpha,A) \in [0,1] \times \A'$}. \label{eq:valid.lo}
\end{align}
If an IM is $\A'$-valid with respect to $\credal$ for $\A' = \A$, then I'll simply say it's {\em valid}. 
\end{defn}

This definition is clearly in the same spirit as Definition~\ref{def:no.prior.valid}, the main difference here being that the ``probability'' is based on considering all possible joint distributions of $(Y,\Theta)$ consistent with the statistical model $\prob_{Y|\theta}$ and the (genuine) partial prior information encoded in $\credal$.  The other difference is that I'm relaxing the requirement that the bounds hold uniformly in $\A$, allowing them to hold only on a sub-collection $\A' \subseteq \A$.  The reason for this relaxation is a technical one---for certain IMs I'm currently only able to establish the bounds above for a proper sub-collection $\A'$.  This relaxation, however, should not be seen as license to tailor the IM construction to a very narrow, specific sub-collection $\A'$; at least not to the extent that the IM is effectively useless for assertions $A \not\in \A'$.  Remember, the goal still is valid uncertainty quantification in a broad sense, so I want $\A'$ to be as close to $\A$ as possible. 

That Definition~\ref{def:valid} is a generalization of Definition~\ref{def:no.prior.valid} can be seen by considering the case where $\credal$ is the set of all probability distributions on $\TT$.  In that case, validity in the sense of Definition~\ref{def:valid} reduces to Definition~\ref{def:no.prior.valid}.  

While the validity condition in Definition~\ref{def:valid} seems strong in the sense that it requires the $\uprob_\credal$-probability bounds (\ref{eq:valid.up}--\ref{eq:valid.lo}) to hold for all assertions $A$, there is another sense in which it is too weak.  In a gambling scenario, the agent will advertise his buying and selling prices based on his specified IM $(\lPi_{Y,\credal}, \uPi_{Y,\credal})$, depending on data $Y$, and his opponents can decide what, if any, transactions they'd like to make.  If the opponents also have access to data $Y$, then of course they will use that information to make a strategic choice of $A$ in order to beat the agent.  If the opponents can use data-dependent assertions, then it's not enough to consider the assertion-wise guarantees provided by Definition~\ref{def:valid}, some kind of uniformity in $A$ is required. This scenario is not so far-fetched.  Imagine a statistician who's developing a method for the applied data analyst to use.  If the statistician can prove that his method satisfies (\ref{eq:valid.up}--\ref{eq:valid.lo}), then his method is reliable for any fixed $A$.  But what if the data analyst peeks at the data for guidance about relevant assertions?  Without some uniformity, validity cannot be ensured in such cases.  With this in mind, consider the following stronger notion of validity.

\begin{defn}
\label{def:strong}
An IM with output $(\lPi_{Y,\credal}, \uPi_{Y,\credal})$ is {\em strongly valid}, relative to $\credal$, if one---and hence both---of the following equivalent conditions holds:
\begin{align}
\uprob_\credal\{ \uPi_{Y,\credal}(A) \leq \alpha \text{ for some $A \ni \Theta$} \} & \leq \alpha, \quad \text{for all $\alpha \in [0,1]$} \label{eq:strong.up} \\
\uprob_\credal\{ \lPi_Y(A) \geq 1-\alpha \text{ for some $A \not\ni \Theta$} \} & \leq \alpha, \quad \text{for all $\alpha \in [0,1]$}. \label{eq:strong.lo}
\end{align}
\end{defn}

Strong validity is mathematically stronger than validity in the sense of Definition~\ref{def:valid} because the ``for some $A \ni \Theta$'' amounts to taking a union over all $A$ that contain $\Theta$.  Therefore, the event on the left-hand side of \eqref{eq:strong.up} is much larger than the corresponding event in \eqref{eq:valid.up}, so the bound of the former by $\alpha$ implies the same of the latter.  Since the aforementioned union would generally be uncountable, there is no obvious guarantee that it would be measurable; if it's not measurable, then obviously the statements in \eqref{eq:strong.up} and \eqref{eq:strong.lo} are meaningless.  However, the following lemma gives a mathematically equivalent---and measurable---version of the events in \eqref{eq:strong.up} and \eqref{eq:strong.lo}, which eliminates the aforementioned measurability concerns.  It also provides a simpler condition to check.   

\begin{lem0}
For an IM with output $(\lPi_{Y,\credal}, \uPi_{Y,\credal})$, define its plausibility contour as 
\begin{equation}
\label{eq:contour}
\pi_{y,\credal}(\vartheta) = \uPi_{y,\credal}(\{\vartheta\}), \quad \vartheta \in \TT. 
\end{equation}
Then strong validity in the sense of Definition~\ref{def:strong} is equivalent to 
\begin{equation}
\label{eq:strong.alt}
\uprob_\credal\{\pi_{Y,\credal}(\Theta) \leq \alpha\} \leq \alpha, \quad \alpha \in [0,1]. 
\end{equation}
\end{lem0}

\begin{proof}
That the conditions \eqref{eq:strong.up} and \eqref{eq:strong.lo} are equivalent is a consequence of the duality between lower and upper probabilities.  So here I'll show that the two events 
\begin{align*}
E_1 & = \{(y,\theta): \text{$\uPi_{y,\credal}(A) \leq \alpha$ and $\theta \in A$ for some $A$}\} \\
E_2 & = \{(y,\theta): \pi_{y,\credal}(\theta) \leq \alpha\}
\end{align*}
are the same, by proving $E_1 \supseteq E_2$ and $E_1 \subseteq E_2$. First it is easy to see that $E_1 \supseteq E_2$ since, if $(y,\theta) \in E_2$, then $A$ can be taken as $A=\{\theta\}$.  Next, to show that $E_1 \subseteq E_2$, recall that the upper probability is monotone: if $A \subseteq B$, then $\uPi_{y,\credal}(A) \leq \uPi_{y,\credal}(B)$ for all $y$.  If $(y,\theta) \in E_1$, then there is a set $A$ such that $\uPi_{y,\credal}(A) \leq \alpha$ and $\theta \in A$.  By monotonicity, it follows that $\pi_{y,\credal}(\theta) := \uPi_{y,\credal}(\{\theta\}) \leq \alpha$; therefore, $(y,\theta) \in E_2$ and, hence, $E_1 \subseteq E_2$.  If the two events are the same, then of course the probabilities in \eqref{eq:strong.alt} and \eqref{eq:strong.up} are the same, which proves the claim. 
\end{proof}

Condition \eqref{eq:strong.alt} also sheds light on what mathematical form the IM output likely needs to have in order to achieve strong validity.  Indeed, since \eqref{eq:strong.alt} says that the random variable $\pi_{Y,\credal}(\Theta)$ is stochastically no smaller than $\unif(0,1)$, as a function of $(Y,\Theta)$, this might not hold if the plausibility contour couldn't reach values arbitrarily close to 1.  While not strictly impossible for other models, those IMs whose upper probabilities take the form of a consonant plausibility function, or possibility measure, are those where, by definition, the contour can take values close to 1.  In what follows, I'll focus on these consonant IMs when considering strong validity.


\subsection{Statistical implications}

Following the logic laid out above and in Section~\ref{SS:no.prior}, a very basic requirement is that validity ought to imply that statistical procedures derived from the IM have certain desirable operating characteristics.  In particular, in order for the validity property to have any ``reliability'' implications, it ought to ensure that the derived procedures have error rate control guarantees.  Theorem~\ref{thm:error} below makes this precise.

\begin{thm}
\label{thm:error}
Let $(\lPi_{Y,\credal}, \uPi_{Y,\credal})$ be an IM depending on the partial prior $\credal$.  Then the following error rate control properties hold.  
\begin{enumerate}
\item If the IM is valid in the sense of Definition~\ref{def:valid}, then the test {\em reject hypothesis ``$\theta \in A$'' if and only if $\uPi_{Y,\credal}(A) \leq \alpha$} has size $\alpha$, i.e., it satisfies  
\begin{equation}
\label{eq:size}
\uprob_\credal\{ \text{\em test rejects and $\Theta \in A$} \} \leq \alpha. 
\end{equation}
\item If the IM is strongly valid in the sense of Definition~\ref{def:strong}, then the set $C_\alpha(y)$ defined in \eqref{eq:pl.region} is a $100(1-\alpha)$\% confidence region, i.e., it satisfies 
\begin{equation}
\label{eq:coverage}
\uprob_\credal\{C_\alpha(Y) \not\ni \Theta\} \leq \alpha. 
\end{equation}
\end{enumerate}
\end{thm}

\begin{proof}
Part~1 follows immediately from \eqref{eq:valid.up}. For Part~2, note that the event where $C_\alpha(Y)$ misses $\Theta$ can be written as 
\begin{align*}
C_\alpha(Y) \not\ni \Theta & \iff \Theta \not\in \textstyle\bigcap \{A: \lPi_{Y,\credal}(A) \geq 1-\alpha\} \\
& \iff \Theta \in \textstyle\bigcup \{A^c: \lPi_{Y,\credal}(A) \geq 1-\alpha\} \\
& \iff \text{$\Theta \in A^c$ and $\lPi_{Y,\credal}(A) \geq 1-\alpha$ for some $A$}.
\end{align*}
The right-most event can be rewritten as 
\[ \text{$\lPi_{Y,\credal}(A) \geq 1-\alpha$ for some $A \not\ni \Theta$}. \]
Since the IM is assumed to be strongly valid, it follows from \eqref{eq:strong.lo} that this event has probability no more than $\alpha$, which proves \eqref{eq:coverage}. 
\end{proof}

For some intuition behind Theorem~\ref{thm:error}, consider two important (extreme) special cases corresponding to the traditional frequentist and Bayes approaches.  For the frequentist case, where $\credal$ is all possible distributions on $\TT$, \eqref{eq:coverage} immediately reduces to the familiar non-coverage probability bound, $\sup_\theta \prob_{Y|\theta}\{C_\alpha(Y) \not\ni \theta\} \leq \alpha$, which is satisfied if $C_\alpha$ is a $100(1-\alpha)$\% confidence region in the traditional sense.  Next, for the purely Bayes case, where $\credal$ is a singleton $\{\prior\}$, $\uprob_\credal$ corresponds to a specific joint distribution of $(Y,\Theta)$ and \eqref{eq:coverage} is the condition automatically satisfied when $C_\alpha$ is the $100(1-\alpha)$\% Bayesian posterior credible region; see Corollary~\ref{cor:suff} below and the related discussion. 

From the lemma in Section~\ref{SS:def}, the $100(1-\alpha)$\% plausibility region $C_\alpha(Y)$ for $\Theta$ in \eqref{eq:pl.region} can be re-expressed in more convenient terms 
\begin{equation}
\label{eq:pl.region.alt}
C_\alpha(y) = \{\vartheta \in \TT: \pi_{y,\credal}(\vartheta) > \alpha\}. 
\end{equation}
Next is an immediate consequence of Theorem~\ref{thm:error}.2 and the lemma. 

\begin{cor}
\label{cor:strong}
If the IM is strongly valid in the sense of Definition~\ref{def:strong} or, equivalently, if its plausibility contour satisfies \eqref{eq:strong.alt}, then the set $C_\alpha(Y)$ in \eqref{eq:pl.region.alt} is a $100(1-\alpha)$\% confidence region in the sense that it satisfies \eqref{eq:coverage}.  
\end{cor}

\subsection{Behavioral implications}

Validity not only has implications for the operating characteristics of procedures derived from the IM, it also has behavioral implications.  Towards this, below I show that if an IM experiences (a weak form of) {\em contraction}, i.e., $\lPi_{y,\credal}(A)$ or $\uPi_{y,\credal}(A)$ uniformly larger or smaller, respectively, than all of the $\prior(A)$ values, then it can't be valid in the sense of Definition~\ref{def:valid}.  An important consequence of this result is that {\em validity implies no sure-loss}.  Avoiding sure-loss is related to the aforementioned coherence properties \citep[e.g.,][Ch.~6--7]{walley1991}, as I describe below.  This establishes a new perspective on validity compared to what had been discussed in previous works.  This helps solidify the intuition that a procedure which is externally reliable shouldn't be internally irrational.  

First a bit of additional notation and terminology.  For the (closed and convex) set $\credal$ of prior distributions $\prior$, define 
\[ \lprior(A) = \inf_{\prior \in \credal} \prior(A) \quad \text{and} \quad \uprior(A) = \sup_{\prior \in \credal} \prior(A), \quad A \in \A. \]
These are {\em a priori} bounds on the buying and selling prices for gambles of the form \$$1(\theta \in A)$, or on degrees of belief in the assertion ``$\theta \in A$.''  If I consider the IM with $y$-dependent output $(\lPi_y, \uPi_y)$ as an ``update''\footnote{The quotation marks around ``update'' are meant to make clear that there is no single, formal updating rule being considered here for the IM construction.  So, in principle, I can't rule out the case that the IM construction ignores the prior entirely, in which it wouldn't make sense to describe the IM construction as an ``update.'' The results here, however, establish that IMs that ignore the prior information may not be valid, which is fully expected.} of the prior beliefs based on the observation $Y=y$, then I'll say that the IM {\em contracts} at $A$ if 
\begin{equation}
\label{eq:contraction}
\lprior(A) < \inf_{y \in \YY} \lPi_{\credal}(A) \leq \sup_{y \in \YY} \uPi_{y,\credal}(A) < \uprior(A). 
\end{equation}
In words, contraction implies that, regardless of what data is observed, the posterior beliefs become more precise.  It's the ``regardless of what's observed'' part of this that's problematic, since it suggests a logical inconsistency, i.e., one should've started with more precise prior beliefs if the updates become more precise no matter what is observed.  This contraction property and its consequences were discussed recently in \citet{gong.meng.update}.  Of course, the same problems are encountered even if the contraction property only happens on one side.  I'll say that the IM has one-sided contraction if
\begin{equation}
\label{eq:contraction.onesided}
\inf_{y \in \YY} \lPi_{y,\credal}(A) > \lprior(A) \quad \text{or} \quad \sup_{y \in \YY} \uPi_{y,\credal}(A) < \uprior(A), \quad \text{for some $A \in \A$}. 
\end{equation}
The intuition behind this is as follows.  Suppose, for example, that the second inequality holds for some $A$, so that $\uPi_y(A)$ is strictly less than $\uprior(A)$ for all observations $y$.  In this case, you know that the price at which I'm willing to sell gambles on the event ``$\Theta \in A$'' will go down as soon as $Y=y$ is observed, no matter what $y$ is, so you'll surely be in a better position if you wait until $y$ is revealed to make your purchase.  This doesn't guarantee that I lose money, but the amount I win is strictly less if you purchase the gamble after $Y$ is observed instead of before. It'd be silly for me to give you a risk-free strategy to improve your circumstances, hence I should avoid the case in \eqref{eq:contraction.onesided}. 

Extreme forms of the one-sided contraction are 
\begin{equation}
\label{eq:sure.loss}
\sup_{y \in \YY} \uPi_{y,\credal}(A) < \lprior(A) \quad \text{and} \quad \inf_{y \in \YY} \lPi_{y,\credal}(A) > \uprior(A). 
\end{equation}
These cases correspond to what's called {\em sure-loss} \citep[e.g.,][Def.~3.3]{gong.meng.update}, since this severe inconsistency between the prior $\credal$ and IM output $(\lPi_{y,\credal}, \uPi_{y,\credal})$ could be exploited by an opponent to force the user who adopts these as bounds on prices for gambles to lose money.  Given that sure-loss is an especially egregious logical violation, it's a pleasure to see that, as I show below, {\em validity implies no sure-loss}. 

It turns out, however, that a notion of contraction that's weaker than those discussed above leads to incompatibility with validity.  I'll say that a partial-prior IM with lower and upper probability output $(\lPi_{y,\credal}, \uPi_{y,\credal})$, depending on the credal set $\credal$ having envelopes $\lprior$ and $\uprior$, {\em almost surely contracts at $A$} if one---and hence both---of the following equivalent conditions holds:
\begin{align}
\prob_{Y|\theta}\{ \uPi_{Y,\credal}(A) \leq \alpha \} & = 1 \quad \text{for some $\alpha < \uprior(A)$ and all $\theta \in A$} \label{eq:pcontract} \\
\prob_{Y|\theta}\{ \lPi_{Y,\credal}(A) \geq 1 - \alpha\} & = 1 \quad \text{for some $\alpha < \uprior(A^c)$ and all $\theta \in A^c$}. \notag
\end{align}
The equivalence of the two statements follows from the duality between $\lPi_{Y,\credal}$ and $\uPi_{Y,\credal}$.  I'll focus on the version \eqref{eq:pcontract} because I prefer to explain things in terms of the upper probability.  To compare with contraction proper, note that \eqref{eq:pcontract} still refers to an unexpected case where the IM's upper probability is ``always'' less than the prior upper probability, but it doesn't require uniformity, no supremum over data sets $y$ is required.  Since validity is a probabilistic notion, it makes intuitive sense that validity can be broken by a high-probability logical violation.  The following result establishes that almost-sure contraction is incompatible with validity in the sense of Definition~\ref{def:valid}.

\begin{thm}
\label{thm:no.sure.loss}
Given $\credal$ with lower and upper envelopes $\lprior$ and $\uprior$, let $(\lPi_{Y,\credal}, \uPi_{Y,\credal})$ denote the partial prior IM's lower and upper probability output.  Suppose that the IM experiences almost sure contraction at $A$ in the sense of \eqref{eq:pcontract}.  Then the IM isn't $\A'$-valid with respect to $\credal$ for any $\A' \ni A$. 
\end{thm}

\begin{proof}
Take the case where $(A,\alpha)$ satisfies \eqref{eq:pcontract}.  For this $(A,\alpha)$ pair, I get 
\[ \uprob_\credal\{\uPi_{Y,\credal}(A) \leq \alpha, \, \Theta \in A\} = \sup_{\prior \in \credal} \int_A \prob_{Y|\theta}\{\uPi_{Y,\credal}(A) \leq \alpha\} \, \prior(d\theta). \]
The integrand on the right-hand side is constant equal to 1 and, therefore, 
\[ \uprob_\credal\{\uPi_{Y,\credal}(A) \leq \alpha, \, \Theta \in A\} = \uprior(A) > \alpha, \]
which implies \eqref{eq:valid.up} fails.  
\end{proof}

It is clear that almost sure contraction---\eqref{eq:pcontract} and its equivalent---is implied by the more familiar notions of contraction discussed above.  In particular, sure-loss in the sense of \eqref{eq:sure.loss} implies \eqref{eq:pcontract} and, Theorem~\ref{thm:no.sure.loss}, the conclusion {\em validity implies no sure-loss}.  But it should be kept in mind that the result is considerably stronger than this.  Validity implies that even a much weaker logical inconsistency, a probabilistic notion of contraction is avoided.  Moreover, that the proof above is basically just one line would suggest that there's a still weaker notion of contraction that is enough to break validity. 

It'll help to draw some connections to the more classical notions of coherence in, e.g., \citet[][Ch.~6--7]{walley1991}.  Walley's Section~6.5.2 gives a pair of (necessary and) sufficient conditions, which he denotes as (C8) and (C9), for a (unconditional, conditional) upper probability pair to be coherent.  In my notation, (C8) is equivalent to the right-most inequality in \eqref{eq:contraction.onesided} {\em failing}, i.e., that $\sup_{y \in \YY}\lPi_{y,\credal}(A) \geq \uprior(A)$ for all $A$.  Condition (C9) is more complicated but, in my notation, reduces to 
\[ \lprior(A) \leq \max\Bigl\{\lPi_{y,\credal}(A), \sup_{x \neq y} \uPi_{x,\credal}(A) \Bigr\}, \quad \text{for all $y \in \YY$ and all $A \subseteq \TT$}. \]
The intuition behind this latter condition is as follows.  If it fails, then there exists a $(y,A)$ pair such that you can sell me a gamble on ``$\Theta \in A$'' for $\lprior(A)$ and then:
\begin{itemize}
\item if $Y \neq y$, then you buy it back from me for $\uPi_Y(A) < \lprior(A)$, so I lose, or
\vspace{-2mm}
\item if $Y=y$, then you do nothing, which effectively forces me to pay more than $\lPi_y(A)$, my advertised largest buying price.
\end{itemize} 
It would often be the case that $y \mapsto \uPi_{y,\credal}(A)$ is continuous, in which case the above condition reduces to $\lprior(A) \leq \sup_{y \in \YY} \uPi_{y,\credal}(A)$, which is implied by the valid IM's no-sure-loss.  Therefore, in many cases, the valid IM would also be coherent in Walley's sense, if only trivially.  Since there are cases \citep[e.g.,][Sec.~6.5.4]{walley1991} in which generalized Bayes (see Section~\ref{SS:gbayes}) is the only IM construction that is coherent, one can't expect that Walley's coherence property can be established for a general IM construction. 

An important example where contraction would be a concern is in the vacuous-prior setting where $\credal$ consists of all priors, so that $\lprior(A) \equiv 0$ and $\uprior(A) \equiv 1$ for all $A$.  In this case, at least intuitively, any IM whose output, say $\Pi_y$, is an ordinary probability distribution experiences contraction---because the vacuous prior is converted into a precise ``posterior'' with values $\Pi_y(A)$ between 0 and 1---and, therefore, is at risk of not being valid relative to the the vacuous prior according to Theorem~\ref{thm:no.sure.loss}.  This suggests a new perspective on the false confidence theorem in \citet{balch.martin.ferson.2017}.  Briefly, the false confidence theorem states that IMs whose output is a probability distribution can't be valid in the vacuous-prior sense of Definition~\ref{def:no.prior.valid}. 

\begin{cor}[A weak false confidence theorem]
\label{cor:false.confidence}
Consider the case where the prior is vacuous, so that $\lprior(A) = 0$ and $\uprior(A) = 1$ for all $A$, and suppose that the IM's output is a probability measure $\Pi_Y$.  Suppose that there exists $A \subseteq \TT$ and $\alpha < 1$ such that 
\[ \prob_{Y|\theta}\{ \Pi_Y(A) \leq \alpha \} = 1 \quad \text{for all $\theta \in A$}. \]
Then the IM fails to be valid relative to the vacuous $\credal$ in the sense of Definition~\ref{def:valid} and, hence, also in the sense of Definition~\ref{def:no.prior.valid}.
\end{cor} 

If, e.g., the collection $\{\Pi_y: y \in \YY\}$ of probability distributions on $\TT$ is tight in the sense of, e.g., \citet[][p.~380]{billingsley}, which automatically holds if $\TT$ is compact, then there exists $A \subset \TT$ with $\alpha := \sup_{y \in \YY} \Pi_y(A) < 1$.  Then the conditions of Corollary~2 are satisfied with this pair $(A,\alpha)$ and, therefore, the IM $y \mapsto \Pi_y$ is not valid. 

It's important to emphasize that the version of the false confidence theorem in \citet{balch.martin.ferson.2017} is stronger than that above.  The principle reason is that the original false confidence theorem doesn't require any contraction-related assumptions, they use the structure of the countably-additive IM to show that the assertions having a problematic contraction-like property always exist.  So the value I see to the above discussion is that it successfully connects the false confidence phenomenon, which is very much frequentist in spirit, to the behavioral considerations which are almost exclusively Bayesian.  Perhaps that ``still weaker notion of contraction'' I referred to above would make the connection to the false confidence phenomenon even tighter.




\subsection{Big-picture implications}
\label{SS:big.picture}

P-values, statistical significance, and other aspects of modern statistical practice have recently come under fire for their alleged contribution to the replication crisis in science \citep[e.g.,][]{nuzzo2014, wasserstein.lazar.asa, wasserstein.schirm.lazar.2019}. My opinion is that the replication crisis has very little to do with statistical practice, but, at least superficially, these are attacks on the foundations of statistics so the statistical community must respond.  And, inevitably, the {\em frequentist vs.~Bayesian} debates re-emerge.  

All sorts of recommendations have been put forth,\footnote{The specific details of these recommendations aren't relevant to the discussion here, so I'll just refer the reader to, e.g., the 2019 special issue of {\em The American Statistician} (\url{https://www.tandfonline.com/toc/utas20/73/sup1}) on ``Statistical Inference in the 21st Century: A World Beyond $p < 0.05$.''} including suggestions to ban p-values and statistical significance from the scientific literature, replacing frequentist with Bayesian reasoning.  While no consensus has been reached, there is common ground: a general lack of understanding or appreciation of uncertainty, variability, etc.~is a major contributor to the replication crisis.  But what can/should be done about it?  The American Statistical Association's President recently charged a task force to prepare an official statement on statistical significance and replicability, and on this point of ``understanding uncertainty,'' the statement\footnote{\url{https://magazine.amstat.org/blog/2021/08/01/task-force-statement-p-value/}} reads
\begin{quote}
{\em Much of the controversy surrounding statistical significance can be dispelled through a better appreciation of uncertainty, variability, ...} 
\end{quote}
Unfortunately, they give no explanation of what ``a better appreciation of uncertainty'' means, so it's not at all clear how their target audience of non-statisticians can make this actionable.  Furthermore, in their section that lists strategies for dealing with uncertainty, they trip up when they lump together frequentist and Bayesian approaches for dealing with uncertainty as if the two are interchangeable.  If there's ever an instance when differences between Bayesian and frequentist can be ignored, it's {\em not} when providing general guidance to non-statisticians on statistical practice and appreciation of uncertainty!  So the two-theory problem plagues the statistical community, and the following (paraphrased) passage from \citet{fraser2011.rejoinder} seems especially relevant:
\begin{quote}
{\em As a modern discipline statistics has inherited two prominent approaches to the analysis of models with data ... How can a discipline, central to science and to critical thinking, have two methodologies, two logics, two approaches that frequently give substantially different answers to the same problems. Any astute person from outside would say, ``Why don’t they put their house in order?'' ...
Of course, the two approaches have been around since 1763 and 1930 with regular disagreement and yet no sense of urgency to clarify the conflicts. And now even a tired discipline can just ask, ``Who wants to face those old questions?'': a fully understandable reaction! But is complacency in the face of contradiction acceptable for a central discipline of science?}
\end{quote} 

Arguably, the statistical community gave up on ``putting their house in order'' because there were no new ideas.  But I believe that the {\em imprecise probability + validity} perspective advocated here is genuinely new and can resolve these important open problems.  Clearly, everyone wants their inferences to be valid relative to what they're willing to assume about the structure of their problem.   To achieve this, the data analyst has only two concrete tasks to complete.  First, he writes down a model for $(Y,\Theta)$ that contains exactly what he's willing to assume about the joint distribution.  If he's willing to assume a single joint distribution for $(Y,\Theta)$, then he's in a situation that's traditionally called ``Bayesian'' whereas, if he's not willing to assume anything about the prior for $\Theta$, then he's in a situation commonly referred to as ``frequentist;'' of course, there are other intermediate cases between these two extremes, such as those in the robust Bayesian literature or when low-dimensional structural assumptions are imposed in high-dimensional problems.  Regardless of where the data analyst falls on this spectrum of starting points, his assumptions about $(Y,\Theta)$ are encoded in $\credal$ (or $\uprob_\credal$).  Second, to ensure the soundness of his inferences---in terms of both the external error rate control and internal rationality properties in Theorems~\ref{thm:error} and \ref{thm:no.sure.loss}---he quantifies his uncertainty in such a way that it's valid with respect to his $\credal$ (or $\uprob_\credal$), in the sense of Definition~\ref{def:valid}.  

The point is that it's not being ``Bayesian,'' ``frequentist,'' or whatever that makes a statistical analysis sound, it's that the data analyst's choice of IM is valid with respect to the specific assumptions that he is willing to make.  This is fully compatible with both the Bayesian and frequentist frameworks, and more.  For example, suppose one is interested in a hypothesis ``$\theta \in A$.'' If the data analyst is willing to assume a specific joint distribution for $(Y,\Theta)$, then the IM that quantifies uncertainty about $A$ using the Bayesian posterior probability for $A$ would be valid with respect to $\uprob_\credal$ the assumed joint distribution; see Section~\ref{SS:gbayes}.  Similarly, if the data analyst has no genuine prior information about $\theta$ and, accordingly, adopts a vacuous prior, then an IM for which $\uPi_y(A)$ is a p-value for the test of $H_0: \theta \in A$ would be valid with respect to $\uprob_\credal$ for any $\credal$; see Section~\ref{SS:revisited}.  A p-value is meant to be interpreted as a {\em plausibility}, and the results here show this interpretation is also mathematically sound; see, also, \citet{impval} and \citet{imchar}.  But despite being compatible with the existing Bayesian and frequentist frameworks, it's also entirely different.  That is, a Bayesian solution is valid with respect to one set of model assumptions and a frequentist solution is valid with respect to another, so there's no chance of mistakenly concluding, as the ASA task force does, that the two frameworks are interchangeable.  For example, the false confidence theorem ensures that no Bayes solution can be valid with respect to the frequentist's vacuous-prior $\uprob_\credal$.    

Of course, efficiency considerations are important too, and I discuss this in Section~\ref{S:efficiency} below.  However, I consider these to be secondary compared to the validity notion emphasized here, in part because overcoming the replication crisis requires control over false discoveries, and also because validity is what establishes a baseline with respect to which the efficiency of different IMs can be compared.   

To summarize, my claim is that by focusing on valid uncertainty quantification relative to the assumptions one is willing to make about $(Y,\Theta)$, the two-theory problem is settled. Instead of worrying about Bayesian-or-frequentist, one looks at what information is available in a given problem and formulates their IM accordingly.  Imprecise probability is crucial for this because the data analyst must be able to tailor the degree of precision in his IM to the strength of the assumptions about $(Y,\Theta)$ he's willing to make.

\section{Achieving validity}
\label{S:achieving}

\subsection{Generalized Bayes}
\label{SS:gbayes}

The goal of this section is to understand what kinds of IM constructions, especially those that incorporate partial prior information in the form of $\credal$, can achieve the validity property in Definition~\ref{def:valid}.  As I show below, the validity property is implied by a particular relationship between the IM's upper probability $\uPi_{y,\credal}(\cdot)$ and the ordinary Bayesian posterior probability $\prior(\cdot \mid y)$.  From this, it'll be clear that the {\em generalized Bayes rule}, described below, provides a construction of an IM that's valid with respect to a given $\credal$. Of course, if $\credal$ consists of a single prior distribution, then the classical Bayesian posterior based on this particular prior is valid with respect to the joint distribution $\uprob_\credal$.  

Towards this, I'll take a closer look at the validity property \eqref{eq:valid.up}.  Using the iterated expectations formula, 
\begin{equation}
\label{eq:iterated}
\uprob_\credal\bigl\{ \uPi_{Y,\credal}(A) \leq \alpha, \, \Theta \in A \bigr\} = \sup_{\prior \in \credal} \int 1\{\uPi_{y,\credal}(A) \leq \alpha\} \, \prior(A \mid y) \, \prob_{Y|\prior}(dy). 
\end{equation}
The integral on the right-hand side can be rewritten as a conditional expectation, 
\[  \E_{Y|\prior} \{ \prior(A \mid Y) \mid \uPi_{Y,\credal}(A) \leq \alpha \} \, \prob_{Y|\prior}\{ \uPi_{Y,\credal}(A) \leq \alpha\}, \]
where $\E_{Y|\prior}$ is expectation with respect to the marginal distribution, $\prob_{Y|\prior}$, of $Y$ depending on prior $\prior$.  Then it's easy to see that validity in the sense of \eqref{eq:valid.up} is implied by 
\begin{equation}
\label{eq:valid.alt}
\sup_{\prior \in \credal} \E_{Y|\prior} \{ \prior(A \mid Y) \mid \uPi_{Y,\credal}(A) \leq \alpha \} \leq \alpha, \quad \text{for all $(\alpha,A) \in [0,1] \times \A$}. 
\end{equation}
There are two points I want to draw the reader's attention to.  First, note that the above condition implies that the magnitude of $\uPi_{Y,\credal}(A)$ impacts the magnitude of $\prior(A \mid Y)$ for every $\prior \in \credal$.  That is, on the set of $Y$'s such that $\uPi_{Y,\credal}(A) \leq \alpha$, the average value of $\prior(A \mid Y)$ is also $\leq \alpha$.  Compare this to the {\em calibration safety} property in Definition~5 of \citet{grunwald.safe}, which asks that inferences about $A$ drawn based on the true conditional distribution of $\Theta$, given $Y$, won't differ on average from those based on the IM output $\uPi_{Y,\credal}(A)$.  Second, it's clear from \eqref{eq:valid.alt}---or directly from \eqref{eq:iterated}---that validity follows if $\uPi_{y,\credal}$ dominates all the conditional probabilities, i.e., if 
\[ \uPi_{y,\credal}(A) \geq \sup_{\prior \in \credal} \prior(A \mid y), \quad \text{for all $(A,y) \in \A \times \YY$}. \]
But it's not necessary that this {\em dominance} holds uniformly in $y$, it's enough if dominance holds only in an average sense.  The following proposition makes this precise. 

\begin{prop}
\label{prop:suff}
If, for some $\A' \subseteq \A$, the IM $(\lPi_{Y,\credal},\uPi_{Y,\credal})$ satisfies the following dominance property, 
\begin{equation}
\label{eq:dominance}
\sup_{\prior \in \credal} \int \frac{\prior(A \mid y)}{\uPi_{y,\credal}(A)} \, \prob_{Y|\prior}(dy) \leq 1, \quad \text{for all $A \in \A'$},
\end{equation}
then it's $\A'$-valid, relative to $\credal$, in the sense of Definition~\ref{def:valid}. 
\end{prop}

\begin{proof}
Take any $A \in \A'$.  By \eqref{eq:iterated}, validity is equivalent to 
\[ \sup_{\prior \in \credal} \int 1\{\uPi_{y,\credal}(A) \leq \alpha\} \, \prior(A \mid y) \, \prob_{Y|\prior}(dy) \leq \alpha. \]
Divide both sides by $\alpha$ and use the fact that 
\[ \alpha^{-1} 1\{\uPi_{y,\credal}(A) \leq \alpha\} \leq \uPi_{y,\credal}(A)^{-1}, \]
to see that condition \eqref{eq:dominance} implies $\A'$-validity. 
\end{proof}

The following is an important and immediate consequence of Proposition~\ref{prop:suff}.  Indeed, the conservatism built in to the generalized Bayes rule, which is motivated by subjective coherence properties, is also sufficient to achieve validity.  

\begin{cor}
\label{cor:suff}
For a given $\credal$, if the IM has output upper probability $\uPi_{Y,\credal}$ defined as the upper envelope in the generalized Bayes rule \citep[e.g.,][Sec.~6.4]{walley1991}, that is, if
\begin{equation}
\label{eq:gbayes}
\uPi_{y,\credal}(A) = \sup_{\prior \in \credal} \prior(A \mid y), \quad A \in \A, 
\end{equation}
then it's valid relative to $\credal$ in the sense of Definition~\ref{def:valid}.  In particular, in the classical single-prior Bayesian formulation, the IM that's equal to the usual Bayesian posterior distribution is valid in the sense of Definition~\ref{def:valid} with respect to the singleton $\credal$. 
\end{cor} 

Theorem~\ref{thm:error} suggested that strong validity in the sense of Definition~\ref{def:strong} was required for the IM's plausibility region to have the nominal coverage probability.  However, that suggestion was being made for a general IM construction.  In the present context, since the generalized Bayes IM $\uPi_{y,\credal}$ involves the same ingredients used to calculate $\uprob_\credal$-probabilities, a coverage probability property can be verified directly without appealing to the general result in Theorem~\ref{thm:error}.  Indeed, using the second expression for evaluating $\uprob_\credal$-probabilities given in \eqref{eq:total.prob}, 
\[ \uprob_\credal\{C_\alpha(Y) \not\ni \Theta\} = \sup_{\prior \in \credal} \int \prior\{ C_\alpha(y)^c \mid y\} \, \prob_{Y|\prior}(dy). \]
So, if $C_\alpha(y)$ is such that 
\begin{equation}
\label{eq:bayes.credible}
\uprior\{ C_\alpha(y)^c \mid y\} \leq \alpha, \quad \text{for (almost) all $y$}, 
\end{equation}
then $\uprob_\credal\{ C_\alpha(Y) \not\ni \Theta \} \leq \alpha$ as claimed.  However, the set $C_\alpha(y)$ in \eqref{eq:pl.region}, i.e., 
\[ C_\alpha(y) = \bigcap \{A: \lprior(A \mid y) \geq 1-\alpha\}, \]
generally doesn't satisfy \eqref{eq:bayes.credible}, so some effort is needed to identify a suitable $C_\alpha(y)$.  The idea of \eqref{eq:bayes.credible} is correct, i.e., to find the smallest $A$ with $\lprior(A \mid y) \geq 1-\alpha$, but the intersection of all $A$'s with this property generally won't satisfy this property. 

The above connection between validity and dominance also sheds light on what constructions likely don't lead to valid IMs.  For example, consider an approach like that described in \citet{dempster2008}, where independent belief functions for $\theta$---one based on prior information, the other based on data---are combined, via Dempster's rule, to produce an IM with output $(\lPi_{Y,\credal}, \uPi_{Y,\credal})$.  The probability intervals $[\lPi_{Y,\credal}(A), \uPi_{Y,\credal}(A)]$ obtained by Dempster's rule tend to be narrower than those corresponding to the generalized Bayes lower and upper envelopes; see \citet[][Theorems~A.3 and A.6]{kyburg1987} and \citet[][Lemma~5.3]{gong.meng.update}.  Moreover, Example~5 in \citet{gong.meng.update}, shows that Dempster's rule can lead to contraction like in \eqref{eq:contraction}, so  Theorem~\ref{thm:no.sure.loss} above suggests that IMs constructed via Dempster's rule can't generally be valid.  

It's important to emphasize that dominance in the sense of \eqref{eq:dominance} above is a sufficient but not necessary condition for validity.  Indeed, there are other IM constructions besides the generalized Bayes lower/upper envelopes above that achieve the validity property, as I show below. It's also worth emphasizing that validity, on its own, doesn't make the IM ``good''---it may happen that \eqref{eq:dominance} is achieved in a trivial way, which is not practically useful.  For example, if the credal set $\credal$ is large, then the upper envelope \eqref{eq:gbayes} in the generalized Bayes rule could be close to 1, for all/many $A$'s, and then the inference would not be informative.  So, beyond validity, it is necessary to consider the IM's {\em efficiency}, which is the focus of Section~\ref{S:efficiency}.

\subsection{Assertion-wise dominance}
\label{SS:dumb}

In the previous section, it was shown that if the IM's upper probability dominates the Bayesian posterior probability uniformly over priors in $\credal$, then validity follows.  It turns out, however, that other kinds of dominance also imply validity.  

\begin{prop}
\label{prop:dumb}
Let $\uPi_{Y,\credal}$ be any IM that incorporates both data $Y$ and the partial prior information encoded in $\credal$, and let $\A' \subseteq \A$.  If the IM satisfies 
\begin{equation}
\label{eq:dumb}
\uPi_{Y,\credal}(A) \geq \uPi_Y(A) \, \uprior(A), \quad \text{$A \in \A'$ with $\uprior(A) > 0$}, 
\end{equation}
where $\uPi_Y$ is the vacuous-prior IM and $\uprior$ is the prior upper probability, then it's $\A'$-valid in the sense of Definition~\ref{def:valid}. 
\end{prop}

\begin{proof}
For any $A \in \A'$ and any given $\theta \in A$, 
\begin{align*}
\prob_{Y|\theta}\{\uPi_{Y,\credal}(A) \leq \alpha\} & \leq \prob_{Y|\theta}\{\uPi_Y(A) \, \uprior(A) \leq \alpha\} \\
& = \prob_{Y|\theta}\{ \uPi_Y(A) \leq \alpha \uprior(A)^{-1}\} \\
& \leq 1 \wedge \alpha \uprior(A)^{-1},
\end{align*}
where the last line follows by validity of the vacuous-prior IM, and $\wedge$ denotes the minimum operator.  Then 
\begin{align*}
\uprob_\credal\bigl\{ \uPi_Y(A) \leq \alpha, \, \Theta \in A \bigr\} & = \sup_{\prior \in \credal} \int_A \prob_{Y|\theta}\{\uPi_{Y,\credal}(A) \leq \alpha\} \, \prior(d\theta) \\
& \leq \sup_{\prior \in \credal} \frac{\alpha \prior(A)}{\uprior(A)} \\
& = \alpha,
\end{align*}
which establishes $\A'$-validity as in Definition~\ref{def:valid}. 
\end{proof}

An important question is if there are any reasonable IM constructions that satisfy the dominance property \eqref{eq:dumb}.  Unfortunately, I'm not aware of any constructions that do so with $\A' = \A$.  To understand this issue, consider just directly using the condition \eqref{eq:dumb} to define the IM.  For example, suppose $\uPi_{y,\credal}(A) := \uPi_Y(A) \star \uprior(A)$, for all $A \in \A$, where $\star$ is a t-norm that's no smaller than the product, i.e., $a \star b \geq ab$. Then \eqref{eq:dumb} holds by definition and, hence, so does validity.  The problem is that, while the so-defined $\uPi_{y,\credal}$ is a capacity, it's not guaranteed to be sub-additive and, hence, isn't an IM.  In particular, if $A$ and $B$ are two disjoint assertions and $\star$ is the product t-norm, then 
\begin{align*}
\uPi_y(A \cup B) \cdot \uprior(A \cup B) & \leq \{\uPi_y(A) + \uPi_y(B)\} \cdot \{\uprior(A) + \uprior(B)\} \\
& \geq \uPi_y(A) \cdot \uprior(A) + \uPi_y(B) \cdot \uprior(B).
\end{align*}
Of course, the oppositely-oriented inequalities don't rule out the possibility of sub-additivity, but more structure would be needed to prove it.  Fortunately, as I'll show in Section~\ref{S:other}, there are natural IM constructions, e.g., using Dempster's rule, that satisfy \eqref{eq:dumb} for all $A$ in a suitable, proper sub-collection $\A'$ of $\A$.  Then the result in Proposition~\ref{prop:dumb} can be used to prove that these IMs are $\A'$-valid.  

From a pragmatic point of view, there may be situations in which a full-blown probabilistic uncertainty quantification about $\Theta$ isn't necessary.  Sometimes being able to test a hypothesis ``$\Theta \in A$'' is enough.  In such cases, the product $\uPi_Y(A) \, \uprior(A)$ could be treated like a p-value, which defines a decision rule 
\begin{equation}
\label{eq:dumb.test}
\text{reject ``$\Theta \in A$'' if $\uPi_Y(A) \, \uprior(A) \leq \alpha$}. 
\end{equation}
Then Proposition~\ref{prop:dumb} above implies an error rate control for the test.

\begin{cor}
The test defined in \eqref{eq:dumb.test} satisfies \eqref{eq:size} for all $A \in \A$.  
\end{cor}

By the validity of the vacuous-prior IM, the corresponding test based on the ``p-value'' $\uPi_Y(A)$ alone would also control the Type~I error.  Since $\uprior(A) \leq 1$, it must be that the product ``p-value'' $\uPi_Y(A) \, \uprior(A)$ is no larger than $\uPi_Y(A)$ and, therefore, generally leads to a more discerning, more powerful test. 


\subsection{Information aggregation}
\label{SS:fusion}

If one thinks of the partial prior encoded in the credal set $\credal$ and the vacuous-prior IM $\uPi_y$ as separate pieces of information concerning the unknown $\theta$, then present problem can be viewed as one of {\em information aggregation} or {\em information fusion}. Naturally, there are so many ways that this fusion step can be carried out, and there's an extensive literature on the subject; see, e.g., the extensive review in \citet{dubois.fusion.2016}. 

Here I want to focus on a simple strategy presented recently in \citet{hose.hanss.2021}.  Their perspective is to combine pieces of information at the credal set level, which is certainly reasonable.  I've described the prior credal set already, but I've not mentioned the credal set corresponding to the vacuous-prior IM treated as a possibility measure.  In \citet{imdec}, I showed that the credal set corresponding to the vacuous-prior IM consists of ``confidence distributions'' in a sense that's slightly broader than that given in the confidence distribution literature \citep[e.g.,][]{xie.singh.2012}.  Since both the prior and the vacuous-prior IM credal sets correspond to meaningful probability distributions for $\theta$, the intersection of these two sets deserves investigation.  

The specific proposal in \citet{hose.hanss.2021}, Proposition~37, is to define an aggregated plausibility contour 
\begin{equation}
\label{eq:fusion}
\pi_{y,\credal}(\vartheta) = 1 \wedge (1-w)^{-1} \, \pi_y(\vartheta) \wedge w^{-1} \, q(\vartheta), \quad \vartheta \in \TT, 
\end{equation}
where $q$ and $\pi_y$ are the plausibility contours for the prior and the vacuous-prior IM, respectively, and $w \in (0,1)$ is a weight parameter that is constant in $(y,\vartheta)$.  The idea is that $w \to 0$ corresponds to putting more weight on the data-driven IM and less on the prior IM.  But since the informativeness of the data is already baked into the vacuous-prior IM's contour $\pi_y$, it's not unreasonable to take $w=\frac12$, which corresponds to equal weight on the two IMs being combined.  This defines a genuine possibility measure if and only if $\sup_\vartheta \pi_{y,\credal}(\vartheta)=1$, which, according to \citet[][Remark~36]{hose.hanss.2021}, is equivalent to the intersection of $\credal$ and the credal set corresponding to $\uPi_y$ being non-empty.


It's often the case that the credal set intersection is non-empty, but this can't be controlled by the user, since it depends on the data.  However, by the connection between data $Y$ and the credal set $\credal$ through the posited statistical model, there is good reason to expect that empty credal set intersection would be a rare event, and Proposition~\ref{prop:fusion} below confirms this hunch.  So I'll simply ignore this potential emptiness and treat $\pi_{y,\credal}$ above as if it defines a genuine IM and proceed to consider if validity is achieved.  Indeed, the resulting IM is strongly valid.  

\begin{prop}
\label{prop:fusion}
The IM with contour $\pi_{y,\credal}$ defined by the aggregation rule above is strongly valid, with respect to $\credal$, in the sense of Definition~\ref{def:strong}. 
\end{prop}

\begin{proof}
As a function of $(Y,\Theta)$ from the posited model, with $\Theta \sim \prior$, the following events are all equivalent:
\begin{align*}
\pi_{Y,\credal}(\Theta) \leq \alpha & \iff (1-w)^{-1}\pi_Y(\Theta) \wedge w^{-1} q(\Theta) \leq \alpha \\
& \iff \pi_Y(\Theta) \leq (1-w)\alpha \text{ or } q(\Theta) \leq w\alpha. 
\end{align*}
Then 
\[ \prob_{Y,\Theta}\{\pi_{Y,\credal}(\Theta) \leq \alpha\} \leq \prob_{Y,\Theta}\{ \pi_Y(\Theta) \leq (1-w)\alpha\} + \prior\{q(\Theta) \leq w\alpha\}. \]
It follows immediately from the strong validity of the vacuous-prior IM that the first term in the above upper bound is $\leq (1-w)\alpha$.  Similarly, a well-known result \citep[e.g.,][]{cuoso.etal.2001} on the probability of alpha-cuts defined by a plausibility contour implies that the second term above is $\leq w\alpha$ as well.  Therefore, the left-hand side is no more than $\alpha$, uniformly in the prior $\prior \in \credal$, which establishes the strong validity claim.  
\end{proof}

This is a striking result because it says there's effectively no reason to be concerned about the aforementioned credal set intersection being empty.  If, e.g., $\sup_\vartheta \pi_{y,\credal}(\vartheta)$ was strictly less than 1 for almost all $y$, then strong validity would not be possible.  So it must be that the supremum is not too small, at least not too frequently.  This could easily be normalized, if so desired, by simply dividing through by $\sup_t \pi_{y,\credal}(t)$.  This normalization step adds to the computational complexity but doesn't affect validity.  Indeed, since the normalizing factor appearing in the denominator is no more than 1, the normalized version of the plausibility contour can't be any smaller than the un-normalized version, and since the latter is strongly valid by Proposition~\ref{prop:fusion}, the former must be too.  

The only downside I see to the above construction is that the resulting IM tends to be inefficient in a statistical sense; see Section~\ref{S:efficiency} below.  Before I can explain what I mean by statistical efficiency, I need to revisit the vacuous-prior IM from Section~\ref{S:background}.

\subsection{Vacuous-prior IMs, revisited}
\label{SS:revisited}

In Section~\ref{S:background}, I describe the basic IM as a ``vacuous-prior'' IM, meaning that there's no genuine prior information available about $\theta$ to incorporate for the purpose of uncertainty quantification.  But since having no prior information requires consideration of all possible states of nature, perhaps a better name for ``vacuous-prior'' is ``every-prior.'' From this perspective, IMs that are valid in the vacuous-prior sense clearly must also be valid with respect to any partial prior $\credal$.  For various reasons, the vacuous-prior construction is not a fully satisfactory solution to the problem of valid inference with partial prior information, but it's technically a strategy that achieves validity, so it deserves mentioning.  Actually, it turns out that the vacuous-prior construction satisfies some other interesting properties when applied in the partial-prior setting, and I present these details below.  

Aside from being trivially valid for any prior credal set $\credal$, the vacuous-prior IMs satisfy some even stronger properties.  The first is the {\em uniform-in-assertions} validity property introduced in Definition~\ref{def:strong}, and the second is a {\em conditional} validity property.  I'll present each of these in turn.  

The vacuous-prior IMs constructed as in Section~\ref{SS:original.im} achieve validity quite easily, arguably too easily.  So perhaps a stronger notion of validity, with uniformity in $A$, is ``just right.'' Indeed, it is not difficult to show that 
\begin{equation}
\label{eq:singleton}
\uPi_y(A) \leq \alpha \text{ for some $A \ni \theta$} \iff \uPi_y(\{\theta\}) \leq \alpha. 
\end{equation}
So, since the vacuous-prior IMs discussed in Section~\ref{SS:original.im} satisfy \eqref{eq:uniform}, they must also be strongly valid in the sense of Definition~\ref{def:strong}.   


\begin{prop}
\label{prop:strong}
If the IM's upper probability $\uPi_Y$ satisfies \eqref{eq:uniform}, then the strong validity property in Definition~\ref{def:strong} holds for any $\credal$. 
\end{prop}

\begin{proof}
The result is an immediate consequence of \eqref{eq:singleton}. To prove \eqref{eq:singleton}, note that, by monotonicity, $\uPi_y(\{\theta\}) > \alpha$ if and only if $\uPi_y(A) > \alpha$ for every $A \ni \theta$.  This latter observation is obviously equivalent to \eqref{eq:singleton}. 
\end{proof}

For the second unique feature of the vacuous-prior IM, I consider a notion of conditional validity, one that conditions on the event ``$\Theta \in A$'' which makes the assertion true instead of simply excluding those cases where the assertion is false, as I do in Definition~\ref{def:valid}.  The following is a precise statement of the property I have in mind.  

\begin{defn}
\label{def:cond.valid}
An IM with output $(\lPi_Y, \uPi_Y)$ is {\em conditionally valid}, relative to $\credal$, if 
\begin{equation}
\label{eq:cond.valid}
\sup_{\prior \in \credal} \frac{\int_A \prob_{Y|\theta}\{ \uPi_Y(A) \leq \alpha\} \, \prior(d\theta)}{\prior(A)} \leq \alpha \quad \text{for all $\alpha \in [0,1]$ and $A \in \A$}. 
\end{equation}
There's an analogous bound in terms of $\lPi_Y$, but I won't consider this here. {\em Note:} To avoid trivialities in \eqref{eq:cond.valid}, one can either define ``$0/0$'' to be 0 or restrict the calculation to assertions $A \in \A$ such that $\lprior(A) > 0$.
\end{defn}

The ratio on the left-hand side of \eqref{eq:cond.valid} is simply the conditional probability that $\uPi_Y(A) \leq \alpha$, given $\Theta \in A$, relative to the prior $\prior \in \credal$, hence the name ``conditional validity.''  This conditioning operation is relevant because, as I mentioned above, while Definition~\ref{def:valid} restricts to those $(Y,\Theta)$ pairs for which the assertion ``$\Theta \in A$'' is true, this new condition restricts and renormalizes to those pairs.  That is, the supremum of the numerator  in \eqref{eq:cond.valid} is the quantity that's bounded in Definition~\ref{def:valid}, so bounding the ratio, with $\prior(A) \leq 1$ in the denominator, is a mathematically stronger condition. 

It turns out that it's straightforward to show that IMs which are valid in the vacuous-prior sense of Definition~\ref{def:no.prior.valid} are also conditionally valid for any $\credal$.  

\begin{prop}
\label{prop:cond.valid}
If the IM's output $(\lPi_Y, \uPi_y)$ satisfies \eqref{eq:uniform}, then it's conditionally valid in the sense of Definition~\ref{def:cond.valid} for any $\credal$. 
\end{prop}

\begin{proof}
The integral in the numerator of \eqref{eq:cond.valid} can be bounded as follows:
\[ \int_A \prob_{Y|\theta}\{ \uPi_Y(A) \leq \alpha \} \, \prior(d\theta) \leq \sup_{\theta \in A} \prob_{Y|\theta}\{ \uPi_Y(A) \leq \alpha\} \cdot \prior(A). \]
It follows from \eqref{eq:uniform} that the supremum in the upper bound is no more than $\alpha$.  Now divide by $\prior(A)$ and take supremum over $\prior \in \credal$ to prove the claim. 
\end{proof}

Apparently conditional validity is too strong for other IM constructions.  For example, even a traditional Bayesian formulation, where $\credal$ consists of a single prior distribution $\prior$, is valid in the sense of Definition~\ref{def:valid} but fails to be conditionally valid.  I'll illustrate this numerically below with a simple simulation study.  

Consider a normal mean problem where, given $\Theta=\theta$, $Y$ is distributed as $\prob_{Y|\theta} = \nm(\theta, n^{-1})$, so that $Y$ behaves like the sample mean based on $n$ independent observations from the normal mean model.  The prior distribution for $\Theta$ is $\prior = \nm(0,1)$.  Let $\Pi_y(\cdot) = \prior(\cdot \mid y)$ denote the Bayesian IM based on this single known prior.  Of interest here is the (conditional) distribution function 
\begin{align*}
H_A(\alpha) & = \text{``$\prob\{\Pi_Y(A) \leq \alpha \mid \Theta \in A\}$''} \\
& = \frac{1}{\prior(A)} \int_A \prob_{Y|\theta}\{ \Pi_Y(A) \leq \alpha \} \, \prior(d\theta). 
\end{align*}
In particular, since it relates to conditional validity, the question is: for what assertions $A \subseteq \TT$ is $H_A(\alpha) \leq \alpha$ for all $\alpha$?  To assess this question, I simulate pairs $(Y,\Theta)$ from the aforementioned joint model and approximate $H_A(\alpha)$ using Monte Carlo; here I take $n=10$ but none of the results are specific to that choice.  Figure~\ref{fig:no.cond.valid} shows these Monte Carlo estimates for two assertions: $A=(1,5)$ and $A=(1,1.5)$.  In the former case, clearly $H_A(\alpha)$ is well below the diagonal line, but not for the latter case.  That there exists an assertion $A$ for which the inequality in \eqref{eq:cond.valid} fails, it follows that the Bayesian IM with a known prior is not conditionally valid in the sense of Definition~\ref{def:cond.valid}.  

\begin{figure}[t]
\begin{center}
\subfigure[Assertion: $A=(1,5)$]{\scalebox{0.6}{\includegraphics{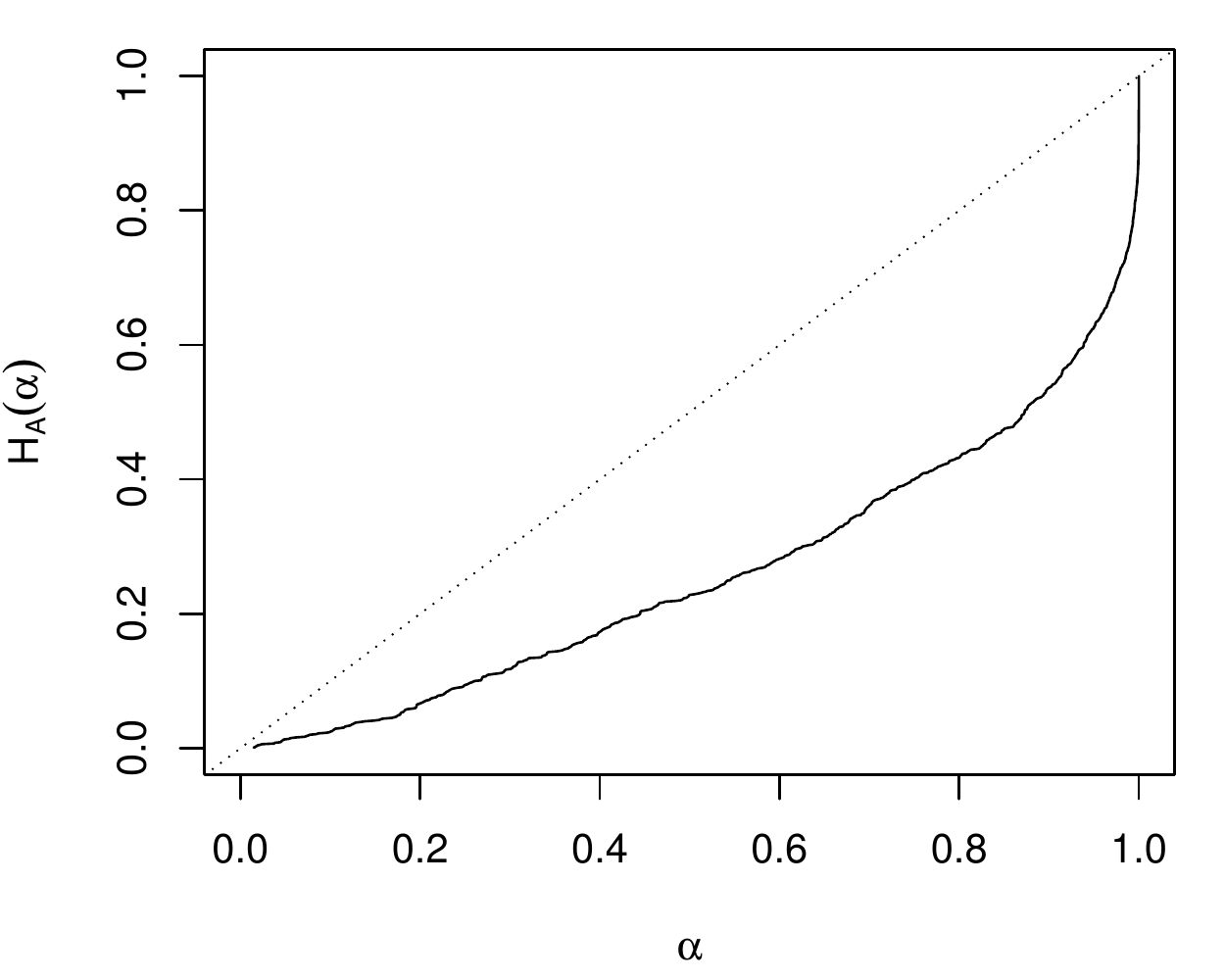}}}
\subfigure[Assertion: $A=(1,1.5)$]{\scalebox{0.6}{\includegraphics{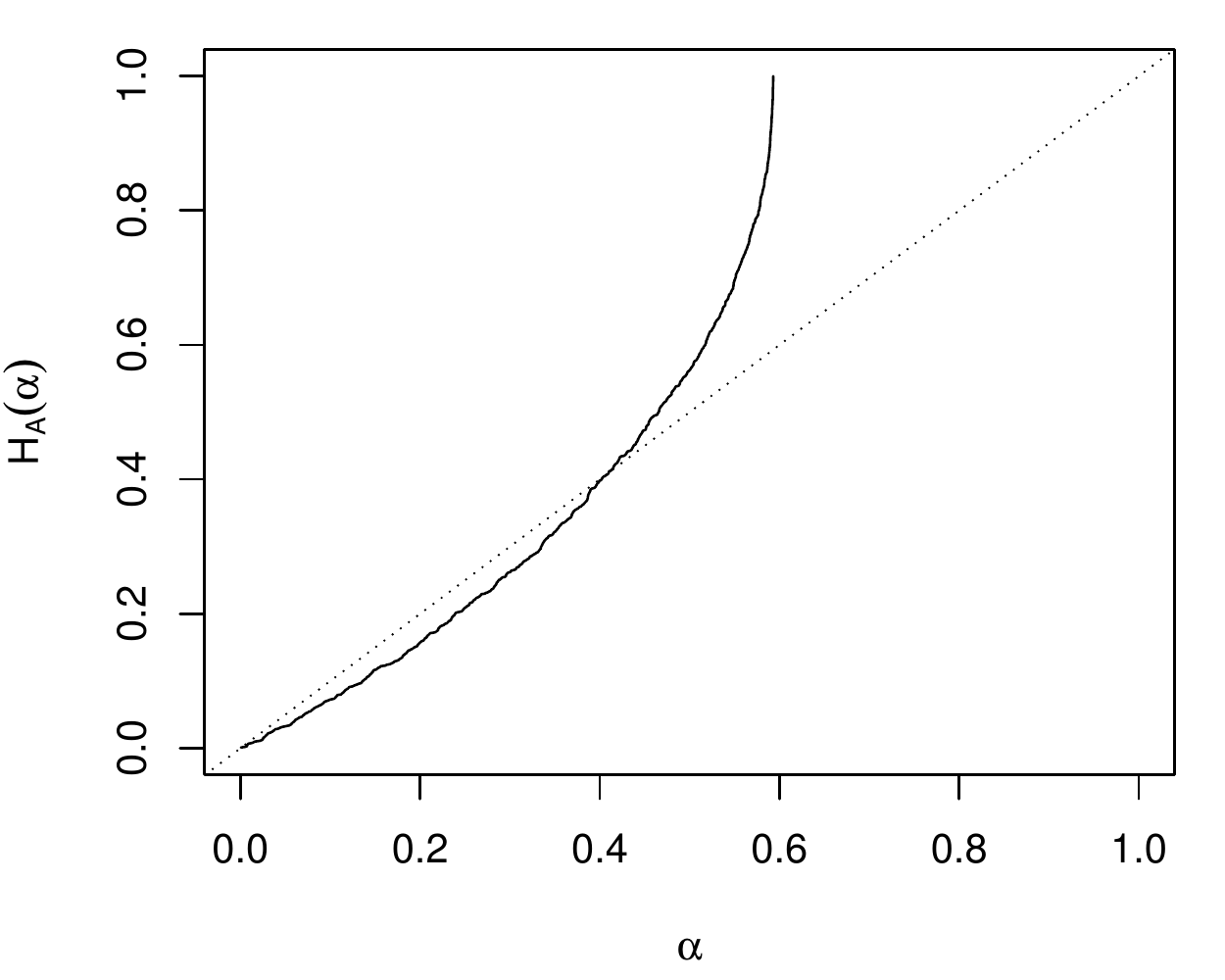}}}
\end{center}
\caption{Monte Carlo estimates of the (conditional) distribution function $H_A(\alpha)$ for the Bayesian IM in the simple normal mean illustration, for two assertions about $\Theta$.}
\label{fig:no.cond.valid}
\end{figure}

That a Bayesian IM generally can't achieve conditional validity even in the known-prior case might be surprising, so it may help for me to describe the related property the Bayesian IM does generally satisfy.  For this explanation, suppose that $Y^n=(Y_1,\ldots,Y_n)$ consists of a collection of $n$ conditionally iid random variables, with $\prob_{Y^n|\theta} = \prob_{Y|\theta}^n$ their joint distribution, given $\Theta=\theta$, and where the marginal distribution of $\Theta$ is $\prior$.  If $\Pi_{Y^n}$ denotes the usual Bayesian posterior distribution for $\Theta$, given $Y^n$, then Doob's theorem \citep[e.g.,][]{doob1949, ghosal.vaart.book, miller.doob} states that $\Pi_{Y^n}(A) \to 1$ with $\prob_{Y^n|\theta}$-probability~1 as $n \to \infty$, for any $\theta$ in a set of $\prior$-probability~1 and for any $A \ni \theta$.  Therefore, 
\[ \prob_{Y^n|\theta}\{\Pi_{Y^n}(A) \leq \alpha\} \to 0, \quad \text{$\prior$-almost all $\theta$, all $\alpha$, and all $A \ni \theta$}, \]
so
\[ \int_A \prob_{Y^n|\theta}\{ \Pi_{Y^n}(A) \leq \alpha \} \, \prior(d\theta) \to 0, \quad \text{as $n \to \infty$, for all $A$}. \]
This implies both ordinary and conditional validity hold trivially, for the Bayesian IM with $\credal = \{\prior\}$, in this large-sample limit.  But I don't think this is satisfactory for the following reason: for any fixed $(\alpha,A)$, validity is achieved in a trivial way as $n \to \infty$, and for any fixed $n$, there are $(\alpha,A)$ pairs for which validity fails, some severely.

\section{Balancing validity and efficiency}
\label{S:efficiency}

\subsection{Objective}

As was shown above, there are a number of different ways that (partial prior) validity can be achieved.  The original vacuous-prior IM, however, achieves both strong and conditional validity for any partial prior.  So if it's already known how to achieve validity (and more) for every prior, then what's the point of considering alternative IM constructions?  The objective is to gain something by incorporating the available partial prior information.  The vacuous-prior IM, by definition, ignores any available partial prior information and, therefore, can't possibly gain.  So in order to compare different IMs that achieve validity, this notion of ``gain'' needs to be explained further.  

The type of gain I have in mind is in terms of what I'll call here {\em efficiency}.  This notion of efficiency is, in my opinion, easy to understand but difficult to formulate in precise mathematical terms.  The upper probability that my IM assigns to assertions needs to be sufficiently large in order to be valid; that is, if my $\uPi_Y$ tends to be too small, then the event ``$\text{$\uPi_Y(A)$ small}$'' may have large probability and then validity fails.  But there is no statistical benefit to the IM assigning upper probabilities any larger than necessary in order to achieve validity.  So, roughly speaking, by efficiency I mean the upper probability is as small as possible.  Like a hypothesis testing rule that always rejects is most efficient---in the sense of having maximal power---this needs to be balanced against something else in order for the test to be practically useful.  So what I'm proposing here is that the upper probabilities should be made as small as possible subject to the validity constraint.  It may help to compare this to some other familiar but related notions.
\begin{itemize}
\item I already mentioned the case of hypothesis testing above, so let's pursue this a bit further.  Of course, a test that assigns p-value 0 to the hypothesis, regardless of the data, will have maximal power, but no one would use it.  By allowing some amount of Type~I error, which amounts to letting the p-value be positive, at least for some data sets, a more useful test can be achieved.  But if I already have a p-value function, say, $p(y)$, that has an acceptable Type~I error, then I'd have no motivation to replace it by something larger, say, $p^{1/2}(y)$. It's in this sense that efficiency seeks upper probabilities no larger than necessary to achieve validity.  
\item Possibility theory has a notion of {\em specificity} \citep[e.g.,][]{dubois.prade.1986} that seems closely related to efficiency.  That is, for two possibility measures having contours $\pi$ and $\pi'$, respectively, the former is more specific than the latter if $\pi \leq \pi'$ pointwise.  This is in line with the notion of efficiency that I'm considering, since pointwise smaller plausibility contours means smaller upper probabilities.  The problem, however, is that I'm not strictly focused on IMs that take the form of possibility measures, so specificity may not be applicable.  More importantly, this notion of specificity doesn't fully capture the characteristic that I'm after.  To see this, consider an example.  Let $y=(y_1,\ldots,y_n)$ denote a sample of size $n$ from $\nm(\theta, 1)$.  Consider two vacuous-prior, consonant IMs with contours 
\begin{align*}
\pi_y(\vartheta) & = 2\{ 1 - \Phi(n^{1/2}|\hat\theta_y - \vartheta|)\}, \\
\pi_y'(\vartheta) & = 2\{ 1 - G_n(|\hat\theta_y' - \vartheta|)\}, 
\end{align*}
where $\hat\theta_y$ is the sample mean, $\hat\theta_y'$ is the sample median, and $G_n$ is the distribution function of the sample median based on $n$ iid samples from $\nm(0,1)$.  It's easy to verify that both of these IMs are valid.  We'd all agree that the former IM is more efficient than the latter, since we know that the sample mean is a better (``more efficient'') estimator of $\theta$ in this example, but the contour functions don't generally satisfy the pointwise ordering, i.e., $\pi_y \not\leq \pi_y'$, at least for some $y$, so I can't say that former is more specific than the latter.  
\end{itemize}
Since what I mean by efficiency here is relatively clear at an intuitive level, I'm satisfied to proceed without a mathematically precise formulation.  

Let's consider the vacuous-prior scenario.  How does the vacuous-prior IM construction, as described in Sections~\ref{SS:original.im} and \ref{SS:revisited} above, stack up to alternative methods in terms of efficiency?  The generalized Bayes formulation described briefly in Section~\ref{SS:gbayes} is attractive for a number of reasons, both statistical and behavioral.  In a vacuous-prior situation, however, it leaves something to be desired.  That is, the natural way to mathematically encode prior ignorance is via a so-called vacuous prior, one where the credal set $\credal=\credal_{\text{all}}$ consists of all possible prior distributions on $\TT$.  It is easy to see, however, that the vacuous prior is supremely informative in the sense that, there is no data set such that the corresponding generalized Bayes posterior is not also vacuous \citep[e.g.,][Theorem~4.8]{gong.meng.update}.  So the naive vacuous-prior version of generalized Bayes is practically useless---learning from the data is impossible.  While generalized Bayes seems hopeless in this context, as mentioned above, \citet{walley2002} remarkably showed that the generalized Bayes rule applied to a variation on the vacuous prior---one with credal set $\credal_W = (1-\eps) \, \prior_0 + \eps \, \credal_{\text{all}}$, where $\prior_0$ is a fixed prior on $\TT$, $\credal_{\text{all}}$ is the aforementioned vacuous prior credal set, and $\eps \in (0,\frac12)$---yields a non-trivial IM that's even effectively valid.  What's especially impressive about Walley's solution is that it achieves virtually all desirable statistical properties at once: (near) validity, coherence, likelihood principle, etc.  However, the cost of this achievement is a general loss of efficiency.  For example, in a normal mean problem, Walley's interval estimates have width of the order $(\log n)^{1/2} n^{-1/2}$, whereas the vacuous-prior IM's intervals \citep[e.g.,][]{imbook} have width of the order $n^{-1/2}$, just like classical confidence intervals.  Therefore, in the vacuous-prior setting, I think it's clear that the vacuous-prior IM 
is most efficient.  

Recall the original question motivating this discussion: {\em if the vacuous-prior IM is (strongly and conditionally) valid relative to any $\credal$, then why bother considering alternative IMs that incorporate the partial prior information?}  How do the alternative IM constructions above, which incorporate the partial prior information, stack up with the vacuous-prior IM that ignores the prior information?  Since the ``construction'' in Section~\ref{SS:dumb} doesn't meet all the conditions needed to call it an IM construction, I'll not consider this here.  That leaves the generalized Bayes and information aggregation constructions.  While both of these constructions incorporate partial prior information in ways that seem intuitively reasonable, plus they achieve the desired validity, my claim is that they do so in a way that's too conservative.  That is, they lead to {\em no efficiency gains compared to the already valid vacuous-prior IM}.  There is some support for the claim concerning the generalized Bayes IM construction in the existing literature.  Specifically, as presented in \citet{kyburg1987} and \citet{gong.meng.update}, the generalized Bayes rule cannot contract---which is why it achieves validity and no sure loss---and, moreover, it tends to dilate.  Therefore, generalized Bayes can't gain in efficiency and, in fact, would tend to be even more conservative than the vacuous-prior IM.  For the information aggregation strategy, that it can't gain in efficiency compared to the vacuous-prior IM should be clear from its definition, since the corresponding plausibility contours are inflated before being combined.  This will be demonstrated numerically in Section~\ref{SS:running} below.  Since these approaches to incorporating the partial prior information do not lead to a gain in efficiency, I'll need to consider some alternative constructions in Section~\ref{S:other}.


\subsection{Example}
\label{SS:running}

Suppose that the conditional distribution of $Y$, given $\Theta=\theta$, is $\prob_{Y|\theta} = \nm(\theta, n^{-1})$ for some known precision index $n \geq 1$ that plays the role of a sample size, and $\theta \in \TT = \RR$.  The situation I'm interested in here is one where a subject-matter expert makes a claim like ``I'm $100(1-\beta)$\% sure that $\Theta$ belongs to $[a,b]$,'' for a given interval $[a,b]$ and given probability $\beta \in [0,1]$, where $\beta$ would typically be relatively small. Such a situation is surely common in real-life applications, but rarely seen in the literature because prior information this vague is either completely ignored or embellished upon with further assumptions.  Following the suggestion in \citet{imexpert}, my proposal is to encode this partial prior information via a random set
\[ \T = \begin{cases} [a,b] & \text{with probability $1-\beta$} \\ \RR & \text{with probability $\beta$}. \end{cases} \]
This (nested) random set has upper probability given by 
\[ \uprior(A) = \prob_\T(\T \cap A \neq \varnothing) = \begin{cases} 0 & \text{if $A = \varnothing$} \\ 1 & \text{if $A \cap [a,b] \neq \varnothing$} \\ \beta & \text{otherwise}. \end{cases} \]
Since $\uprior$ is also a possibility measure, it can be described by the plausibility contour 
\[ q(\vartheta) = \beta + (1-\beta) \, 1\{\vartheta \in [a,b]\}, \quad \vartheta \in \RR. \]
For the data-dependent, vacuous-prior IM, if $Y=y$ is the observed data, then the (nested) random set $\TT_y(\U)$ is  
\begin{align*}
\TT_y(\U) & = y - \U \\
& = [y - |\tilde U|, \, y + |\tilde U|], \quad \text{where $\tilde U \sim \nm(0,n^{-1})$}. 
\end{align*}
This, too, defines a possibility measure, so the corresponding upper probability is 
\[ \uPi_y(A) = \sup_{\vartheta \in A} \pi_y(\vartheta), \]
where the plausibility contour function is
\[ \pi_y(\vartheta) = 2\{1 - \Phi(n^{1/2}|y - \vartheta|)\}, \quad \vartheta \in \RR, \]
with $\Phi$ the $\nm(0,1)$ distribution function.  I didn't clearly explain why the vacuous-prior construction is as presented above, in particular, why I chose the generator random set 
\[ \U = [-|\tilde U|, |\tilde U|], \quad \tilde U \sim \nm(0, n^{-1}). \]
I've commented on this in more detail elsewhere so I'll not repeat myself here.  As justification for this choice here, let me just say that if one extracts the $100(1-\alpha)$\% plausibility region from this vacuous-prior IM, one finds that 
\[ \{\vartheta: \pi_y(\vartheta) > \alpha\} = y \pm \Phi^{-1}(1-\tfrac{\alpha}{2}) \, n^{-1/2}, \]
and the right-hand side is the textbook $100(1-\alpha)$\% z-confidence confidence interval, which is the shortest interval estimator that achieves the nominal coverage probability.  

Since we already have a good/optimal solution in the vacuous-prior IM, it makes sense to consider what, if anything, is gained through the use of one of the aforementioned IMs that incorporates the available partial prior information.  For this first version of the example, I'll just compare the vacuous-prior IM and that based on the simple information aggregation rule in Section~\ref{SS:fusion}, with $\pi_{y,\credal}$ defined as in \eqref{eq:fusion}.  This choice is mainly for visualization purposes, since both are expressed directly in terms of the plausibility contours and, in particular, can be visualized as such.  In this experiment, we take $a=1$, $b=2$, and $\beta=0.1$; so the expert is saying she's ``90\% sure that $\theta$ is in $[1,2]$.''  Figure~\ref{fig:ab.example1} shows plots of the corresponding plausibility contours for four different values of $y$, with $n=10$.  The particular choices of $y$ are based on the interpretations:
\begin{itemize}
\item[(a)] $y=1.5$: data consistent with the prior;
\vspace{-2mm}
\item[(b)] $y=1.1$: data moderately consistent with the prior;
\vspace{-2mm}
\item[(c)] $y=0.9$: data moderately inconsistent with the prior;
\vspace{-2mm}
\item[(d)] $y=0.5$: data inconsistent with the prior.
\end{itemize} 
In two ``moderate'' cases we see some signs of a gain from incorporating the partial prior information, but this wouldn't be reflected in, say, the 90\% or 95\% plausibility intervals.  In fact, in all four cases, the information aggregation would give wider plausibility intervals at those levels of significance.  In Panel~(a), where the data and prior are in perfect agreement, we'd hope to see some gain in efficiency, but we actually see efficiency loss.  Finally, in Panel~(d), where data and prior disagree, the incompatibility introduced in the information aggregation rule can be seen by noticing that now the curve doesn't reach the value 1.  This could be renormalized but then it's clear that the corresponding IM would be very inefficient compared to the vacuous-prior IM.  The take-away message is that, while (strongly) valid, the information aggregation rule does not serve the desired purpose of providing an efficiency gain through the incorporation of partial prior information.  Therefore, there is a need to consider some alternative constructions. 

\begin{figure}[t]
\begin{center}
\subfigure[$y=1.5$]{\scalebox{0.6}{\includegraphics{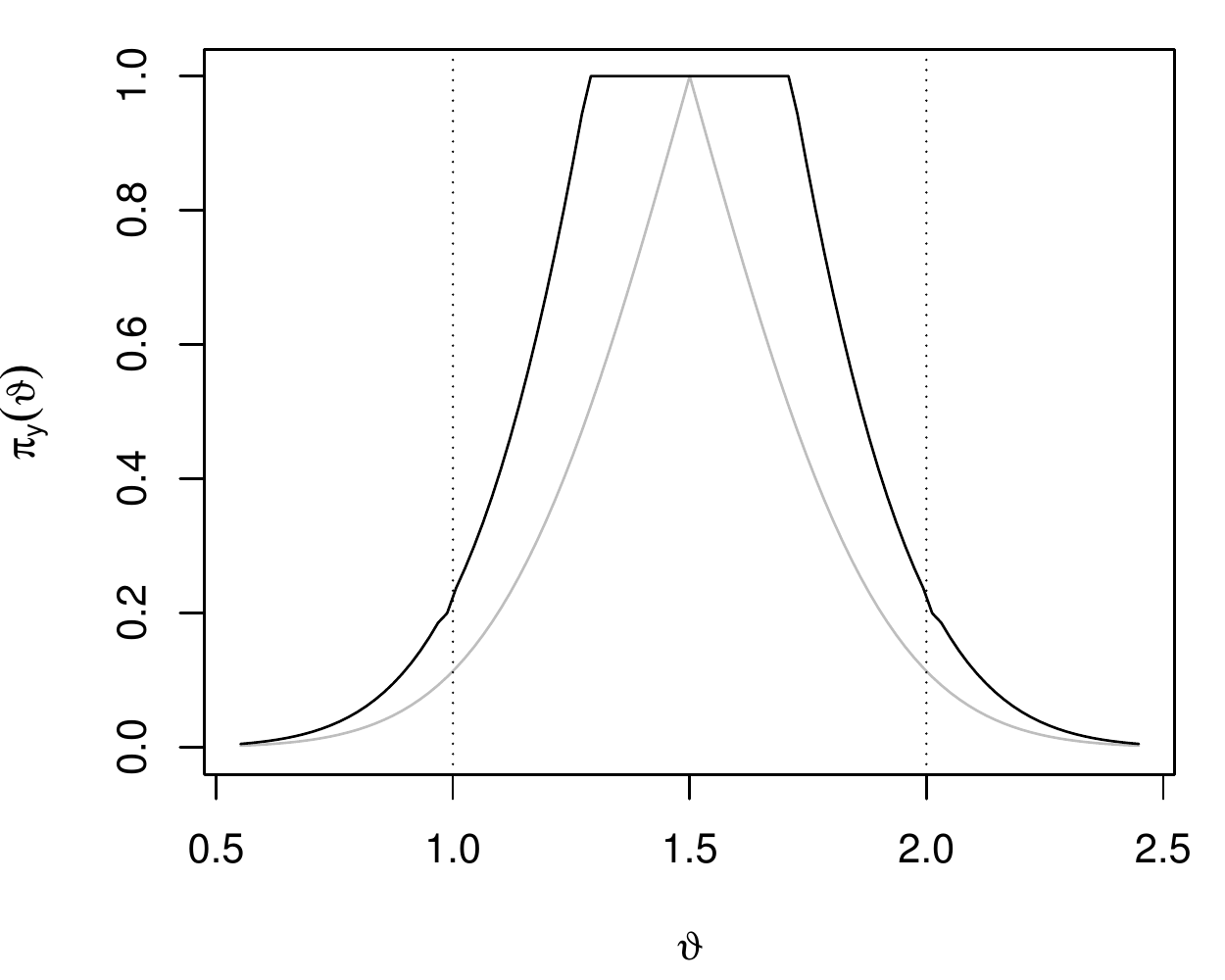}}}
\subfigure[$y=1.1$]{\scalebox{0.6}{\includegraphics{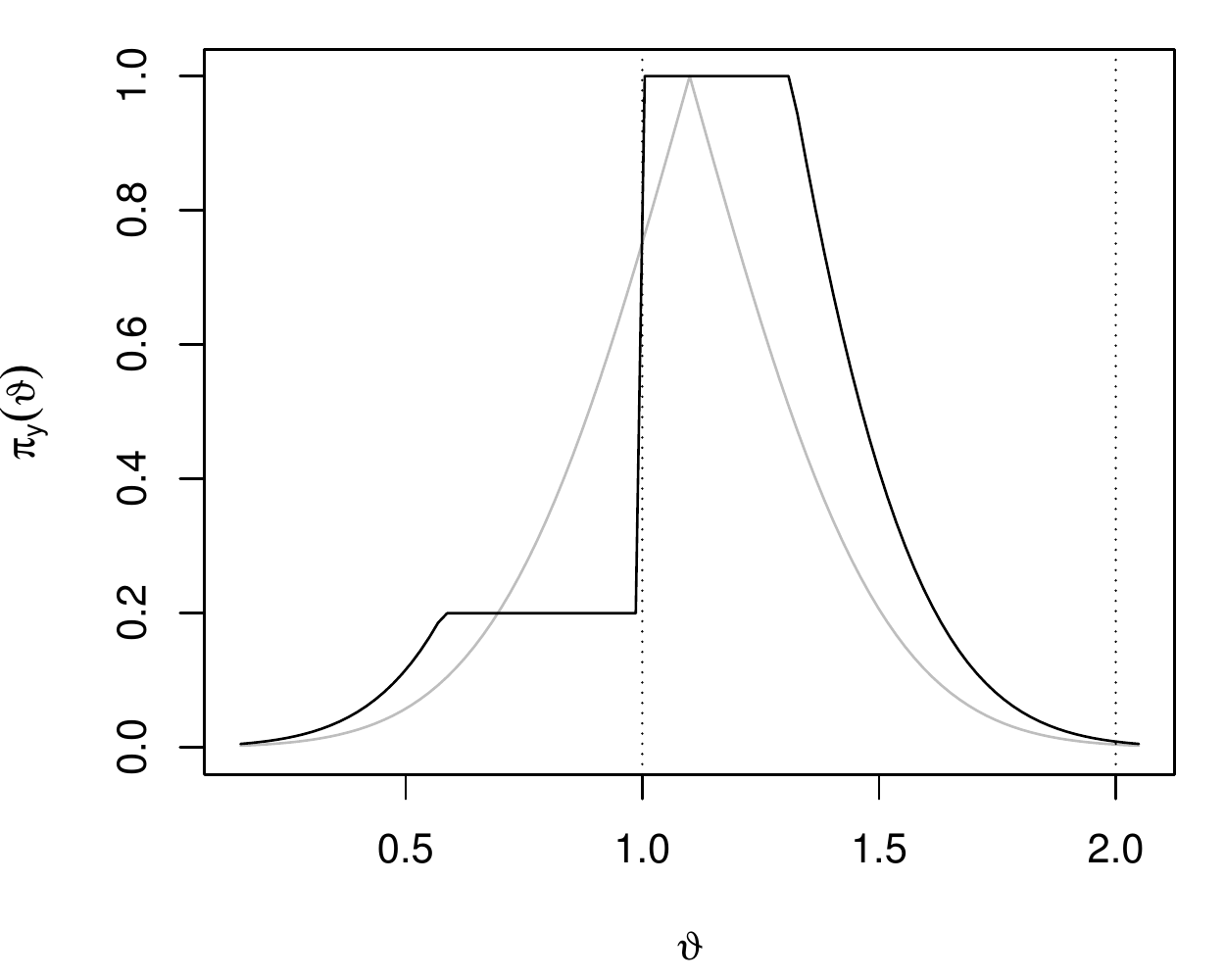}}}
\subfigure[$y=0.9$]{\scalebox{0.6}{\includegraphics{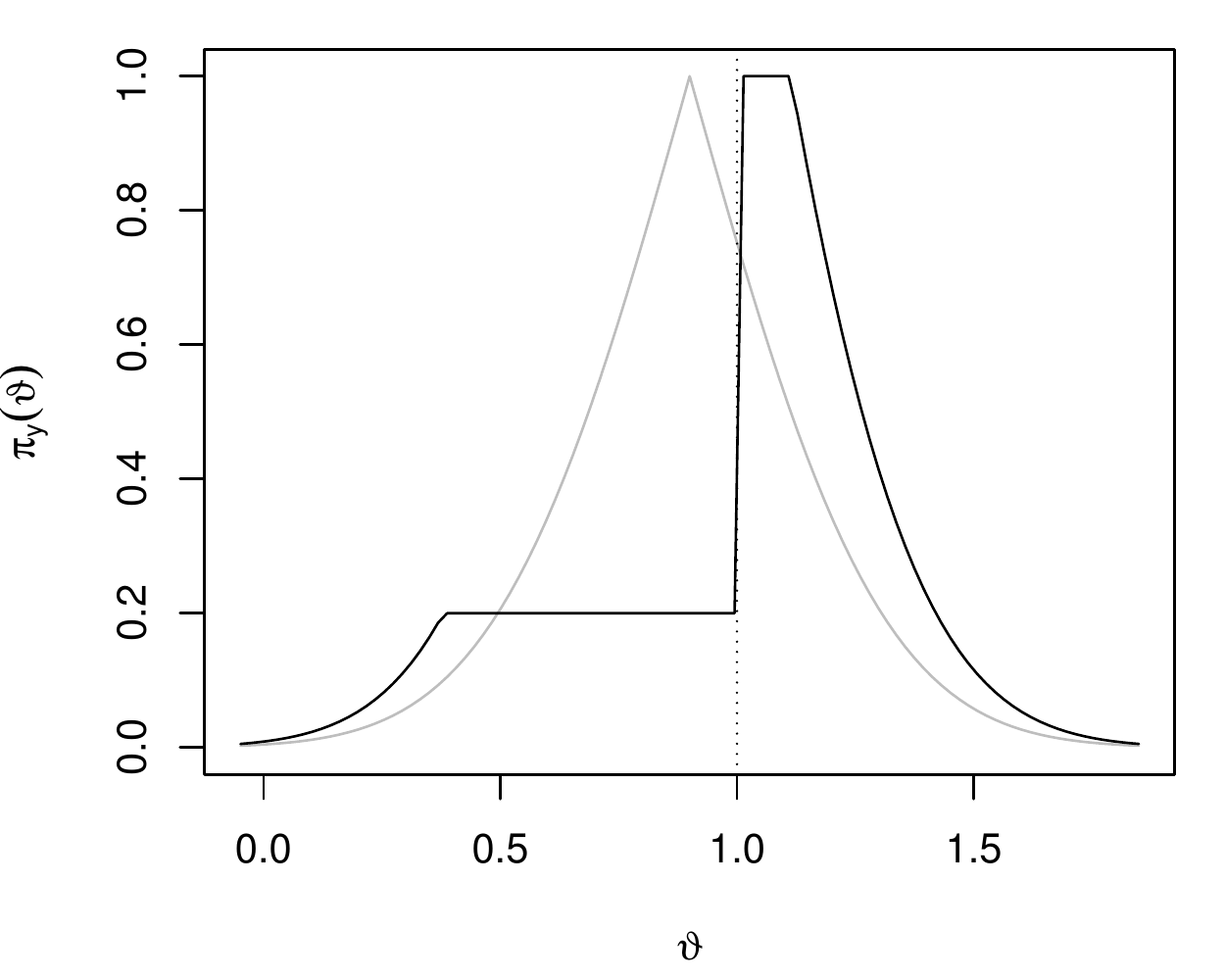}}}
\subfigure[$y=0.5$]{\scalebox{0.6}{\includegraphics{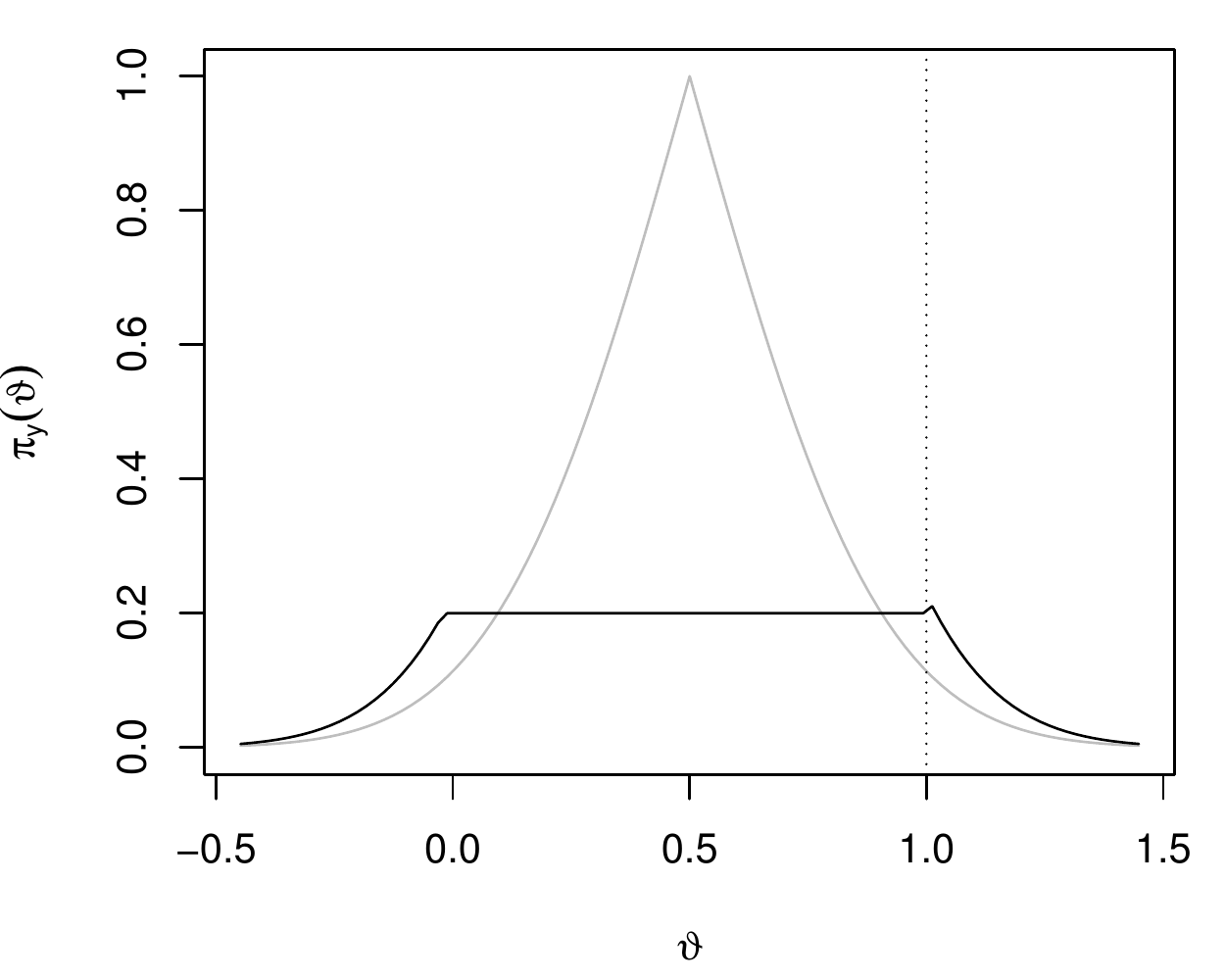}}}
\end{center}
\caption{Plausibility contour plots for the two different IMs discussed in Section~\ref{SS:running}, for four different values of the data $y$. Gray line is based on the vacuous-prior IM; black line is the based on the information aggregation rule in \eqref{eq:fusion}.}
\label{fig:ab.example1}
\end{figure}

\begin{remark0}
After reading a previous draft of this manuscript, Dominik Hose pointed out to me a different information aggregation strategy.  This idea has its roots in \citet{destercke.etal.2009} and elsewhere, and is further developed in \citet{hose2022thesis}.  As Dominik explains, this alternative rule fuses information in the prior contour $q$ with that in the vacuous-prior IM's contour $\pi_y$ according to the rule
\begin{equation}
\label{eq:fusion.2}
\pi_{y,\credal}(\vartheta) = [1 - \{1 - q(\vartheta)\}^2] \wedge [1 - \{1 - \pi_y(\vartheta)\}^2], \quad \vartheta \in \TT. 
\end{equation}
This approach has several advantages compared to that described in Section~\ref{SS:fusion}: first, it doesn't depend on a tuning parameter $w$; second, it tends to be at least slightly more efficient; and, third, it's also strongly valid, at least in some cases.  To see this latter claim, for $\pi_{Y,\credal}$ in \eqref{eq:fusion.2}, 
\begin{align*}
\pi_{Y,\credal}(\Theta) \leq \alpha & \iff 1 - \{1 - q(\Theta)\}^2 \leq \alpha \text{ or } 1 - \{1 - \pi_Y(\Theta)\}^2 \leq \alpha \\
& \iff q(\Theta) \wedge \pi_Y(\Theta) \leq 1-(1-\alpha)^{1/2}.
\end{align*}
Then 
\begin{align*}
\prob_{Y,\Theta}\{\pi_{Y,\credal}(\Theta) \leq \alpha \} & = \prob_{Y,\Theta}\{\pi_Y(\Theta) \leq 1-(1-\alpha)^{1/2}\} + \prob_\Theta\{q(\Theta) \leq 1-(1-\alpha)^{1/2}\} \\
& \qquad - \prob_{Y,\Theta}\{q(\Theta) \vee \pi_Y(\Theta) \leq 1-(1-\alpha)^{1/2}\}.
\end{align*}
It's clear that strong validity of the vacuous-prior IM isn't enough to get strong validity of the fused version \eqref{eq:fusion.2}, but this can work in certain special cases.  For example, if the first two terms, which are generally both $\leq 1-(1-\alpha)^{1/2}$, happen to satisfy 
\[ \prob_{Y,\Theta}\{\pi_Y(\Theta) \leq 1-(1-\alpha)^{1/2}\} = \prob_\Theta\{q(\Theta) \leq 1-(1-\alpha)^{1/2}\} = 1-(1-\alpha)^{1/2}, \]
and if the random variables $\pi_Y(\Theta)$ and $q(\Theta)$ are independent or suitably positively correlated such that 
\begin{align*}
\prob_{Y,\Theta}\{q(\Theta) \vee \pi_Y(\Theta) \leq 1-(1-\alpha)^{1/2}\} & \geq \prob_{Y,\Theta}\{\pi_Y(\Theta) \leq 1-(1-\alpha)^{1/2}\} \\
& \qquad \times \prob_\Theta\{q(\Theta) \leq 1-(1-\alpha)^{1/2}\}, 
\end{align*}
then I get 
\[ \prob_{Y,\Theta}\{\pi_{Y,\credal}(\Theta) \leq \alpha \} \leq 2\{1 - (1-\alpha)^{1/2}\} - \{1-(1-\alpha)^{1/2}\}^2 = \alpha, \]
which implies strong validity.  There are certain cases where these conditions can be verified, for example, in smooth location models where $\pi_Y(\Theta)$ is a function of only the error in ``$Y = \Theta + \text{error}$,'' which is independent of $\Theta$.  For the running example considered here, the independence assumption does hold, and my numerical results (not shown) confirm the claim that this strategy is more efficient.  However, since these conditions can't generally hold, I won't discuss this strategy any further. 
\end{remark0}

\section{Alternative constructions}
\label{S:other}

\subsection{Setup}

As discussed above, the vacuous-prior IM achieves strong validity properties uniformly over all credal sets, so really the only statistical motivation for introducing the partial prior information is if, by doing so, the efficiency improves.  In this section, I present two reasonable strategies for combining the vacuous-prior IM with partial prior information, such that the resulting IM is valid (at least potentially so) and more efficient than the strategies presented above.  The two strategies I employ here are based on different combination rules, namely, Dempster's rule of combination and another that is commonly used in the possibility theory literature but I don't believe has a name.  

Next I make three quick remarks to fix the scope of what's presented below.  
\begin{itemize}
\item First, I'm treating the vacuous-prior IM, which depends on data and statistical model only, and the partial prior information as two independent pieces of information about $\theta$ to be combined.  That is, I'm taking the construction of a vacuous-prior IM as a given first step, and then considering how the partial prior information can be incorporated in such a way that validity is maintained but efficiency is gained.  Instead, one could start at a more preliminary step with the data, statistical model, and partial prior and consider pooling information together from there.  This is also reasonable, but I'll not consider this perspective here.  
\item Second, the partial prior information, encoded in the credal set $\credal$, will be described in terms of a random set $\T \sim \prob_\T$, taking values in the power set of $\TT$, independent of the ingredients used to construct the vacuous-prior IM.  In this case, the prior upper probability is given by 
\[ \uprior(A) = \prob_\T(\T \cap A \neq \varnothing),  \quad A \subseteq \TT. \]
Furthermore, I'll assume that this random set is nested, so that $\uprior$ satisfies the properties of a possibility measure.  Of course, there may be applications where this assumption wouldn't make sense, but there are some cases where it does, so this is a reasonable assumption for this first investigation.  
\item Third, this is not an exhaustive list of combination rules that might work for this purpose, so there may very well be better strategies than those considered below.  Some further remarks on this are presented in Section~\ref{S:discuss}.  
\end{itemize}

\subsection{Dempster's combination rule}
\label{SS:dempster}

In addition to the setup described above, with a (nested) random set $\T \sim \prob_\T$ that encodes the partial prior information about $\theta$, recall also that the vacuous-prior IM for $\theta$ can be described by a (nested) random set $\TT_y(\U)$, a data-dependent mapping of the random set $\U \sim \prob_\U$ in an auxiliary variable space $\UU$.  Since $\T$ and $\U$ are (naturally) treated as independent, we can employ Dempster's rule of combination to construct a new IM that incorporates both data and partial prior information as follows:  
\[ \uPi_{y,\credal}(A) = \frac{\prob_{\T,\U}\{ \TT_y(\U) \cap \T \cap A \neq \varnothing\}}{\prob_{\T, \U}\{\TT_y(\U) \cap \T \neq \varnothing\}}, \quad A \subseteq \TT. \]
This defines a perfectly reasonable IM, a data- and prior-dependent quantification of uncertainty about the unknown $\theta$.  In fact, the original, vacuous-prior IM is a special case, one where the prior random set $\T$ is almost surely equal to $\TT$. It's common to assume that the support of $\T$ is finite; see the example in Section~\ref{SS:running}, discussed further below.  Then the prior uncertainty quantification is determined by a probability mass function, $m$, defined on the finite support of $\T$, so that 
\[ \uprior(A) = \prob_\T(\T \cap A \neq \varnothing) = \sum_{T: T \cap A \neq \varnothing} m(T). \]
In this case, the Dempster's rule-based IM can be re-expressed as 
\begin{equation}
\label{eq:dempster}
\uPi_{y,\credal}(A) = \frac{\sum_{T: A \cap T \neq \varnothing} m(T) \, \uPi_y(A \cap T)}{\sum_T m(T) \, \uPi_y(T)}, \quad A \subseteq \TT. 
\end{equation}
The approach that combines the vacuous-prior IM and the partial prior via Dempster's rule is very natural in this case.  The questions are (a)~is it valid and, (b)~if so, is it also efficient?  In \citet{imexpert}, it was shown empirically that $\uPi_{Y,\credal}$ in \eqref{eq:dempster} doesn't generally satisfy the vacuous-prior validity property in Definition~\ref{def:no.prior.valid}.  This is not especially surprising, it's arguably an unfair comparison in the present context,\footnote{In \citet{imexpert}, the focus was on situations where the (partial) prior information was to be used but not fully trusted, so the vacuous-prior validity property was a more natural benchmark than here where I'm taking the prior information is genuine, albeit vague/incomplete.} so, on it's own, this is not a criticism of Dempster's rule.  However, it does suggest that establishing validity for Dempster's rule \eqref{eq:dempster} might be delicate. 

Unfortunately, a general proof of validity presently escapes me.  I can, however, prove $\A'$-validity for a sub-collection $\A'$ of all possible assertions.  I give some further remarks on this below. To define this sub-collection of assertions, suppose without loss of generality that the focal elements of the prior random set $\T$ are listed in increasing order, i.e., 
\[ T_1 \subset T_2 \subset \cdots \subseteq \TT. \]
Let the assertion $A$ be such that
\begin{equation}
\label{eq:dempster.restrict}
\text{there exists $k \geq 1$ with $A \cap T_j = \varnothing$ for $j < k$ and $A \subseteq T_j$ for $j \geq k$}. 
\end{equation}
Note that $k$ can vary with $A$.  Denote the subclass of all assertions for which \eqref{eq:dempster.restrict} holds as $\A_\credal'$, with the subscript indicating the dependence of this sub-collection on the focal elements that contribute to the definition of $\credal$.  In words, $\A_\credal'$ consists of assertions that are completely contained in $T_k \setminus T_{k-1}$ for some $k \geq 1$.  In the example presented in Section~\ref{SS:running}, the two focal elements are $T_1 = [a,b]$ and $T_2 = \RR$, so \eqref{eq:dempster.restrict} is satisfied for assertions $A$ such that $A \subseteq [a,b]$ or $A \subseteq [a,b]^c$, that is, $A$ can't have non-empty intersection with both $[a,b]$ and its complement.  This collection is generally non-empty, as it definitely includes all the singletons.  It's precisely the sub-collection $\A_\credal'$ of assertions for which I can verify that Dempster's rule as in \eqref{eq:dempster} combines the prior information in $\credal$ with the vacuous-prior IM in a way that achieves validity.

\begin{prop}
\label{prop:dempster}
Let $\credal$ be a partial prior encoded in a random set $\T$ with focal elements as described above.  Then the IM in \eqref{eq:dempster} that combines prior and data according to Dempster's rule is $\A_\credal'$-valid in the sense of Definition~\ref{def:valid} for $\A_\credal'$ defined by \eqref{eq:dempster.restrict}. 
\end{prop}

\begin{proof}
Let $A \in \A_\credal'$, and let $k$ be as specified in \eqref{eq:dempster.restrict}. A first immediate observation is that the vacuous-prior IM satisfies
\[ \uPi_y(A \cap T_j) = \begin{cases} 0 & \text{if $j < k$} \\ \uPi_y(A) & \text{if $j \geq k$}. \end{cases} \]
A less immediate observation is that, for the same $k$, 
\[ \sum_{j \geq k} m(T_j) = \uprior(T_k) \geq \uprior(A). \]
Putting these two together gives the following bound
\begin{align*}
\uPi_{y,\credal}(A) & \geq \sum_j m(T_j) \, \uPi_y(A \cap T_j) \\
& = \sum_{j \geq k} m(T_j) \, \uPi_y(A) \\
& = \uPi_y(A) \, \uprior(T_k) \\
& \geq \uPi_y(A) \, \uprior(A),
\end{align*}
where the first inequality follows because the denominator in \eqref{eq:dempster} is no more than 1; see Section~\ref{SS:remarks}. This is precisely the dominance property in \eqref{eq:dumb}.  Since this holds for all $A \in \A_\credal'$, the claim here follows by Proposition~\ref{prop:dumb}.
\end{proof}

The need to restrict this result to $A \in \A_\credal'$ is unexpected, but it's not clear to me at present how to get around this; see Section~\ref{SS:remarks} for explanation of the difficulties.  However, this restriction aligns with that implicitly introduced in \citet{leafliu2012}.  There, they considered the case where $\T \equiv \TT_0$, with $\TT_0$ a fixed proper subset of $\TT$; in other words, the partial prior information said only that $\theta$ is sure to be in $\TT_0$. In that case, $\A_\credal'$ contains all assertions that are subsets of $\TT_0$, and it's for those assertions that \citet{leafliu2012} were able to prove that the vacuous-prior IM, combined with this specific form of prior information via Dempster's rule, was valid.  So, in effect, Proposition~\ref{prop:dempster} generalizes their result.  The difference is that, for the extreme form of prior information considered in \citet{leafliu2012}, this restriction is entirely intuitive; in my more general case, I don't presently have any clear intuition for this restriction.


\subsection{Consonance-preserving combination rule}
\label{SS:consonant}

In the setup described above, both the prior and the vacuous-prior IM are possibility measures, or consonant plausibility functions, that can be described by the corresponding plausibility contour functions, $q$ and $\pi_y$, respectively.  In such cases, one might want the resulting combined IM to also be consonant, but this can't be guaranteed when, e.g., the combination is carried out using Dempster's rule as described above.  This requires a specific consonance-preserving combination strategy.  Let $\star$ denote a suitable t-norm and define the combined IM directly via its plausibility contour
\begin{equation}
\label{eq:im.tnorm}
\pi_{y,\credal}(\vartheta) = \frac{\pi_y(\vartheta) \star q(\vartheta)}{\sup_{t \in \TT} \{ \pi_y(t) \star q(t)\}}, \quad \vartheta \in \TT. 
\end{equation}
This construction can also be motivated from the perspective of fuzzy set intersections, but I'll not dig into this interpretation here; see, e.g., \citet{dubois.prade.1988}. It's clear that this function is bounded between 0 and 1, and takes the value 1 at where the supremum in the denominator is attained.  This ensures $\pi_{y,\credal}$ is a plausibility contour function and, therefore, defines a consonant IM.  That is, 
\begin{equation}
\label{eq:consonant}
\uPi_{y,\credal}(A) = \sup_{\vartheta \in A} \pi_{y,\credal}(\vartheta) = \frac{\sup_{\vartheta \in A} \{ \pi_y(\vartheta) \star q(\vartheta)\}}{\sup_{\vartheta \in \TT} \{ \pi_y(\vartheta) \star q(\vartheta)\}}, \quad A \subseteq \TT. 
\end{equation}
Common choices of the t-norm $\star$ are product and minimum \citep[e.g.,][]{zadeh1965, zadeh1978}, but I have preference for the product, as I explain in Section~\ref{SS:running2} below.   

Unfortunately, like with the construction above based on Dempster's rule, I've not yet been able to establish theoretically that this consonance-preserving construction achieves validity in the broad sense.  The technical challenge is the same as that for the Dempster's rule-based construction, i.e., analysis of the normalizing constant in the denominator.  Some further remarks on the technical challenges are given in Section~\ref{SS:remarks}, and I show empirically in Sections~\ref{SS:running2} and \ref{S:sparsity} that this construction, with product t-norm, performs quite well in terms of both validity and efficiency.  I can, however, prove an analogous result to that in Proposition~\ref{prop:dempster}, i.e., that the consonance-preserving IM construction is $\A_\credal'$-valid for the collection of assertions $\A_\credal'$ defined in \eqref{eq:dempster.restrict}. 

\begin{prop}
\label{prop:consonant}
Let $\credal$ be a partial prior encoded in a plausibility contour $q$ as described above.  Then the IM in \eqref{eq:consonant} that combines prior and data in a way that preserves consonance is $\A_\credal'$-valid in the sense of Definition~\ref{def:valid} for $\A_\credal'$ defined by \eqref{eq:dempster.restrict}. 
\end{prop}

\begin{proof}
Without loss of generality, I'll work with the minimum t-norm, i.e., $a \star b = a \wedge b$, in \eqref{eq:im.tnorm}.  Since this is the largest of the t-norms, the argument below implies the same conclusion for any other t-norm.  

The first key observation, which was used in the proof of Proposition~\ref{prop:dempster}, is that the normalizing constant in \eqref{eq:im.tnorm} is no more than 1; see Section~\ref{SS:remarks}.  This implies that 
\[ \pi_{y,\credal}(\vartheta) \geq \pi_y(\vartheta) \wedge q(\vartheta). \]
The second key observation is that the prior plausibility contour $q$ is constant on any $A \in \A_\credal'$.  In fact, it's easy to see that $q(\vartheta) \equiv \uprior(A)$ for $A \in \A_\credal'$.  Then I can complete the proof by considering two separate cases based on the magnitude of that constant.  First, if $\uprior(A) \leq \alpha$, then it follows automatically that  
\begin{align*}
\prob_{Y,\Theta}\{\uPi_{Y,\credal}(A) \leq \alpha, \Theta \in A\} & \leq \prob_{Y,\Theta}\{\pi_Y(\Theta) \wedge q(\Theta) \leq \alpha, \Theta \in A\} \\
& \leq \uprior(A) \\
& \leq \alpha. 
\end{align*}
Second, if $\uprior(A) > \alpha$, then 
\[ \pi_y(\vartheta) \wedge q(\vartheta) \leq \alpha, \; \vartheta \in A \iff \pi_y(\vartheta) \leq \alpha, \quad \vartheta \in A. \]
Therefore, 
\begin{align*}
\prob_{Y,\Theta}\{\uPi_{Y,\credal}(A) \leq \alpha, \Theta \in A\} & \leq \prob_{Y,\Theta}\{\pi_Y(\Theta) \wedge q(\Theta) \leq \alpha, \Theta \in A\} \\
& \leq \prob_{Y,\Theta}\{\pi_Y(\Theta) \leq \alpha, \, \Theta \in A\} \\
& \leq \alpha,
\end{align*}
where the last inequality follows from the validity of the vacuous-prior IM. The above two cases exhaust all the possibilities, so the claim is proved. 
\end{proof}

The same remarks following the proof of Proposition~\ref{prop:dempster} can be applied here as well.  In the next subsection, I'll explain the challenge in extending the above validity result to a broader class of assertions.

\subsection{Remarks on the technical challenges}
\label{SS:remarks}

The strategy used in the proof of Proposition~\ref{prop:dempster} suggests a seemingly obvious way to establish the result for general assertions $A$.  That is, decompose the general $A$ based on its intersections with the focal elements of $\T$.  More specifically, write 
\[ A = \bigcup_k \{ A \cap (T_k \setminus T_{k-1}) \}, \quad \text{general $A \in \A$}. \]
Then each piece, $A_k := A \cap (T_k \setminus T_{k-1})$, satisfies the condition \eqref{eq:dempster.restrict}, i.e., $A_k \in \A_\credal'$, and so the Dempster's rule-based IM, with upper probability $\uPi_{y,\credal}$ in \eqref{eq:dempster}, satisfies 
\[ \uPi_{y, \credal}(A_k) \geq \uPi_y(A_k) \, \uprior(A_k), \quad \text{for each $k$}. \]
Then 
\begin{align*}
\int_A \prob_{Y|\theta}\{ \uPi_{Y,\credal}(A) \leq \alpha\} \, \prior(d\theta) & = \sum_k \int_{A_k} \prob_{Y|\theta}\{ \uPi_{Y,\credal}(A) \leq \alpha\} \, \prior(d\theta) \\
& \leq \sum_k \int_{A_k} \prob_{Y|\theta}\{ \uPi_{Y,\credal}(A_k) \leq \alpha\} \, \prior(d\theta).
\end{align*}
From the proof of Proposition~\ref{prop:dumb}, on which the proof of Proposition~\ref{prop:dempster} was based, it's clear that the integrand in the last expression above can be bounded as 
\[ \prob_{Y|\theta}\{ \uPi_{Y,\credal}(A_k) \leq \alpha\} \leq \alpha \uprior(A_k)^{-1}. \]
Plugging in this bound into the integral and then taking the supremum gives 
\begin{align*}
\uprob_\credal\{ \uPi_{Y,\credal}(A) \leq \alpha, \, \Theta \in A\} & = \sup_{\prior \in \credal}\int_A \prob_{Y|\theta}\{ \uPi_{Y,\credal}(A) \leq \alpha\} \, \prior(d\theta) \\
& \leq \sup_{\prior \in \credal} \sum_k \int_{A_k} \prob_{Y|\theta}\{ \uPi_{Y,\credal}(A_k) \leq \alpha\} \, \prior(d\theta) \\
& \leq \sup_{\prior \in \credal} \sum_k \frac{\alpha \, \prior(A_k)}{\uprior(A_k)} \\
& = \alpha \times \sup_{\prior \in \credal} \sum_k \frac{\prior(A_k)}{\uprior(A_k)}.
\end{align*}
Each individual summand is upper-bounded by $1$, but the supremum would be greater than 1 and so this argument fails to establish validity of the Dempster's rule-based IM at a general assertion $A$.  Of course, the failure of one particular proof strategy does not make the claim false, but it does give me some reasons to be concerned.  At the very least, it suggests a better understanding is needed.  




Next, consider the consonance-preserving IM defined above, with t-norm $a \star b = a \wedge b$; this is the largest of the t-norms and, therefore, would be the easiest with which to prove the desired result.  Since the consonant $\uPi_{y,\credal}$ is defined via a supremum, a proof of validity would surely require showing 
\[ \prob_{Y,\Theta}\{\pi_{Y,\credal}(\Theta) \leq \alpha, \, \Theta \in A\} \leq \alpha, \quad \text{for all $A$ and all priors $\prior \in \credal$ for $\Theta$}. \]
Moreover, since the denominator of $\pi_{Y,\credal}$ is complicated, involves a supremum, it's most convenient to replace it with it's upper bound 1, i.e., 
\[ \pi_{Y,\credal}(\Theta) \leq \alpha \implies \pi_Y(\Theta) \wedge q(\Theta) \leq \alpha. \]
This leads to the upper bound
\begin{align*}
\prob_{Y,\Theta}\{ \pi_{Y,\credal}(\Theta) \leq \alpha, \, \Theta \in A\} & \leq \prob_{Y,\Theta}\{ \pi_Y(\Theta) \wedge q(\Theta) \leq \alpha, \, \Theta \in A\} \\
& = \prob_{Y,\Theta}\{ \pi_Y(\Theta) \leq \alpha, \, \Theta \in A\} + \prob_\Theta\{ q(\Theta) \leq \alpha, \, \Theta \in A\} \\
& \qquad - \prob_{Y,\Theta}\{ \pi_Y(\Theta) \vee q(\Theta) \leq \alpha, \, \Theta \in A\}.
\end{align*}
The first probability on the right-hand side is upper-bounded by $\alpha \uprior(A)$, and this bound is sharp; the second can upper-bounded by $\alpha \wedge \uprior(A)$, but this is probably not sharp.  The third can be trivially lower-bounded by 0, but this doesn't help, since $\alpha \uprior(A) + \alpha \wedge \uprior(A)$ exceeds $\alpha$, at least when $\uprior(A)$ is larger than $\alpha$.  Presently, I don't see how get a suitable, non-trivial lower bound on the third probability on the last line of the above display.  

Incidentally, the numerical results that follow indicate that the consonance-preserving construction might lead to strong validity.  However, I have serious doubts that this can be true in general.  Indeed, to prove strong validity, I'd need to show
\[ \prob_{Y,\Theta}\{\pi_{Y,\credal}(\Theta) \leq \alpha\} \leq \alpha \quad \text{for all priors $\prior \in \credal$ for $\Theta$}. \]
Using the same upper-bound-the-denominator-by-1 strategy described above, I get 
\begin{align*}
\prob_{Y,\Theta}\{ \pi_{Y,\credal}(\Theta) \leq \alpha\} & \leq \prob_{Y,\Theta}\{ \pi_Y(\Theta) \wedge q(\Theta) \leq \alpha\} \\
& = \prob_{Y,\Theta}\{ \pi_Y(\Theta) \leq \alpha\} + \prob_\Theta\{ q(\Theta) \leq \alpha\} - \prob_{Y,\Theta}\{ \pi_Y(\Theta) \vee q(\Theta) \leq \alpha\}.
\end{align*}
Like above, the first term can be (sharply) bounded by $\alpha$ and the second term can be (maybe-not-sharply) bounded by $\alpha$.  But the third probability is no bigger than the minimum of the first two, so this upper bound is actually {\em no smaller than $\alpha$}.  Again, an upper bound for $\prob_{Y,\Theta}\{ \pi_{Y,\credal}(\Theta) \leq \alpha\}$ that's bigger than $\alpha$ doesn't disprove the claim, but it means that the upper-bound-the-denominator-by-1 strategy isn't going to work.  And based on my experience investigating properties of the denominator, it seems unlikely that there's a useful upper bound that's less than 1.  

Although the validity results I've been able to prove for the two IM constructions described above are not fully satisfactory, I do believe better results are possible; see Section~\ref{SS:running2} below.  I'll continue to work on this and I hope to report on these efforts elsewhere, perhaps in a subsequent revision of the present paper.

\subsection{Example, continued}
\label{SS:running2}

Two alternative partial prior-based IM constructions were proposed above, one based on Dempster's rule and another based on a consonance-preserving combination rule.  Both of these IMs, as well as the vacuous-prior IM, have plausibility contour functions.  The two consonant IMs would be completely determined by their plausibility contours, while the Dempster's rule-based IM is not; however, there's no reason the latter's contour function can't also be used to visualize the its output for the purpose of making a direct comparison with the former's.  Recall that both the prior and the vacuous-prior IM are determined by their plausibility contours:
\begin{align*}
q(\vartheta) & = \beta + (1-\beta) \, 1(\vartheta \in [a,b]) \\
\pi_y(\vartheta) & = 2\{ 1 - \Phi\bigl(n^{1/2}|y-\vartheta|\bigr)\}. 
\end{align*}
Then the Dempster's rule-based and consonance-preserving rule-based combined IMs have contour functions given by, respectively,  
\begin{equation}
\label{eq:pid.example}
\pid_{y,\credal}(\vartheta) = \frac{\pi_y(\vartheta) \{\beta + (1-\beta) \, 1(\vartheta \in [a,b])\}}{\beta + (1-\beta) \{1(y \in [a,b]) + \pi_y(a) \, 1(y < a) + \pi_y(b) \, 1(y > b)\}}
\end{equation}
and
\begin{equation}
\label{eq:pic.example}
\pic_{y,\credal}(\vartheta) = \frac{\pi_y(\vartheta) \{\beta + (1-\beta) \, 1(\vartheta \in [a,b])\}}{1(y \in [a,b]) + \{\beta \vee \pi_y(a)\} \, 1(y < a) + \{(1-\beta) \vee \pi_y(b)\} \, 1(y > b)}
\end{equation}
Clearly, the two expressions have the same numerators, the difference is only in the denominator.  This means the shape of the two curves will be the same, only the magnitudes will differ.  The denominator in \eqref{eq:pic.example} is no larger, and generally strictly smaller, than that in \eqref{eq:pid.example}, so a plot of $\pic_{y,\credal}$ will tend to be above that of $\pid_{y,\credal}$.  Of course, when $y \in [a,b]$, both contours are the same and equal to the product $\pi_y \, q$.  Compared to the vacuous-prior IM, both of \eqref{eq:pid.example} and \eqref{eq:pic.example} will be no larger and will tend to be strictly smaller as a result of being combined with partial prior information.  To visualize this, the same numerical example is carried out as in Section~\ref{SS:running}, and the plots in Figure~\ref{fig:ab.example2} can be compared to those in Figure~\ref{fig:ab.example1} above.  In Panels~(a) and (b), there's no difference between the two prior-dependent contours because data closely agrees with the prior, i.e., since $y \in [a,b]$.  When data and prior agree, a gain in efficiency is expected in the prior versus vacuous-prior IMs and the plots confirm this---note that both $\pic_{y,\credal}$ and $\pid_{y,\credal}$ are more narrowly spread than $\pi_y$.  In Panels~(c) and (d), the data and prior are not in perfect agreement, and the differences between the contour functions are more apparent.  In Panel~(c), the data are apparently sufficient close to agreeing with the prior that values $\vartheta$ close and to the right of $a$ are assigned higher plausibility than those near $y$.  In Panel~(d), where there is the most disagreement between data and prior, the two prior-based contours differ most strikingly.  Indeed, in cases like this with prior--data conflict, the denominator in \eqref{eq:pid.example} is considerably larger than that in \eqref{eq:pic.example}, so $\pid_{y,\credal} \ll \pic_{y,\credal}$.  In this case, perhaps as expected, the prior-based IMs can't fully decide between the data and the prior, so high plausibility is assigned to a neighborhood around $y$ and to a neighborhood around the closest boundary point $a$ of the region supported by the prior.  

\begin{figure}
\begin{center}
\subfigure[$y=1.5$]{\scalebox{0.6}{\includegraphics{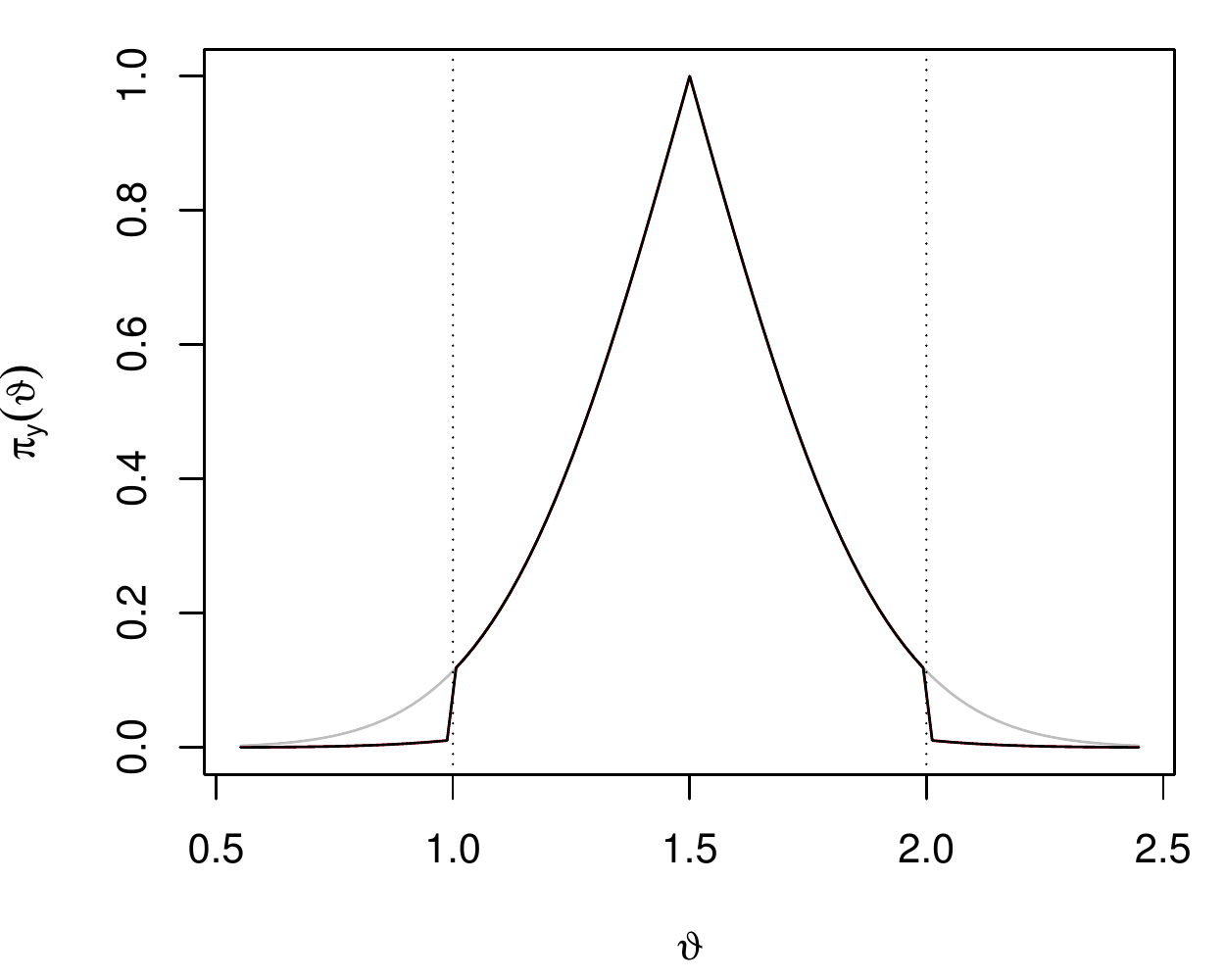}}}
\subfigure[$y=1.1$]{\scalebox{0.6}{\includegraphics{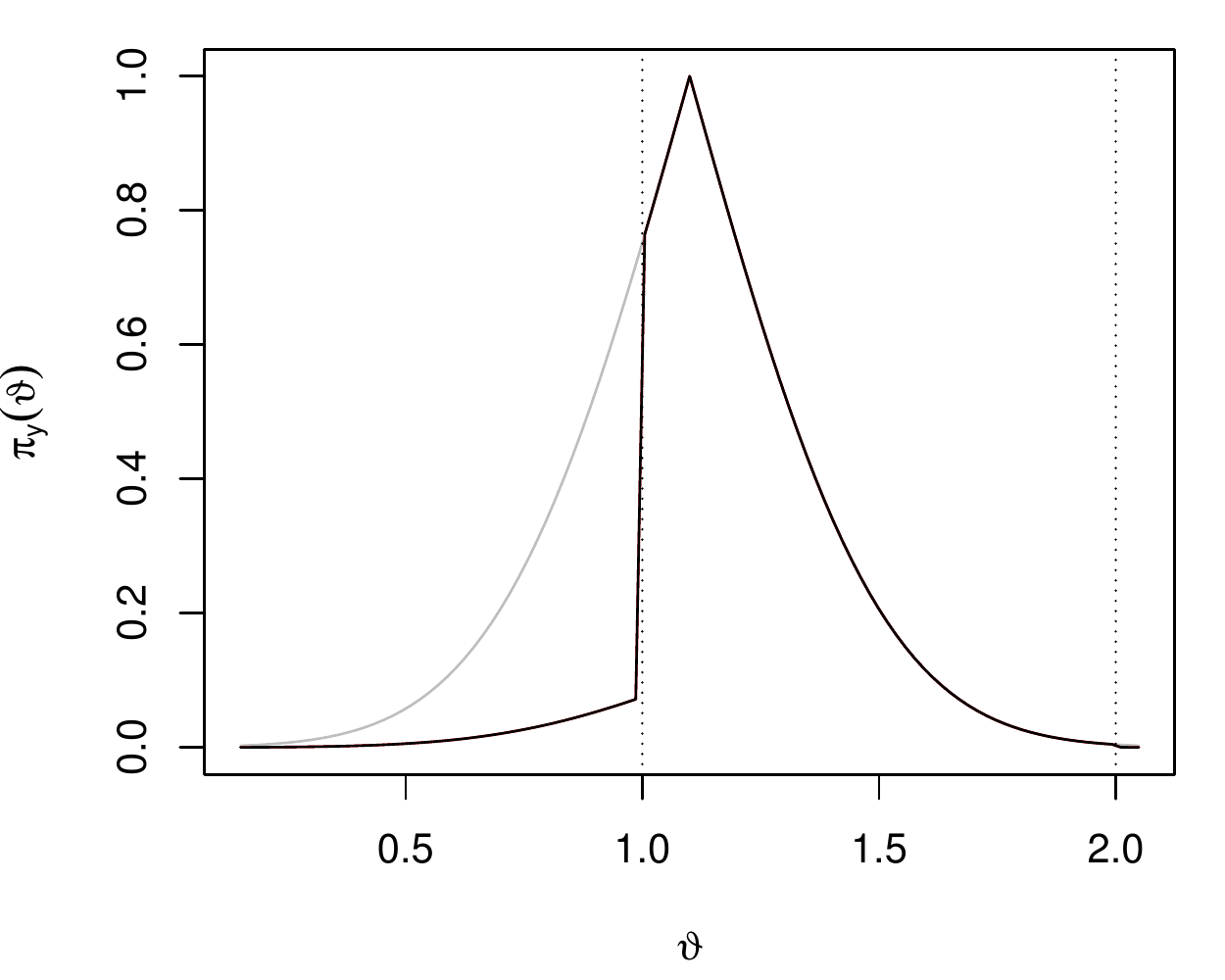}}}
\subfigure[$y=0.9$]{\scalebox{0.6}{\includegraphics{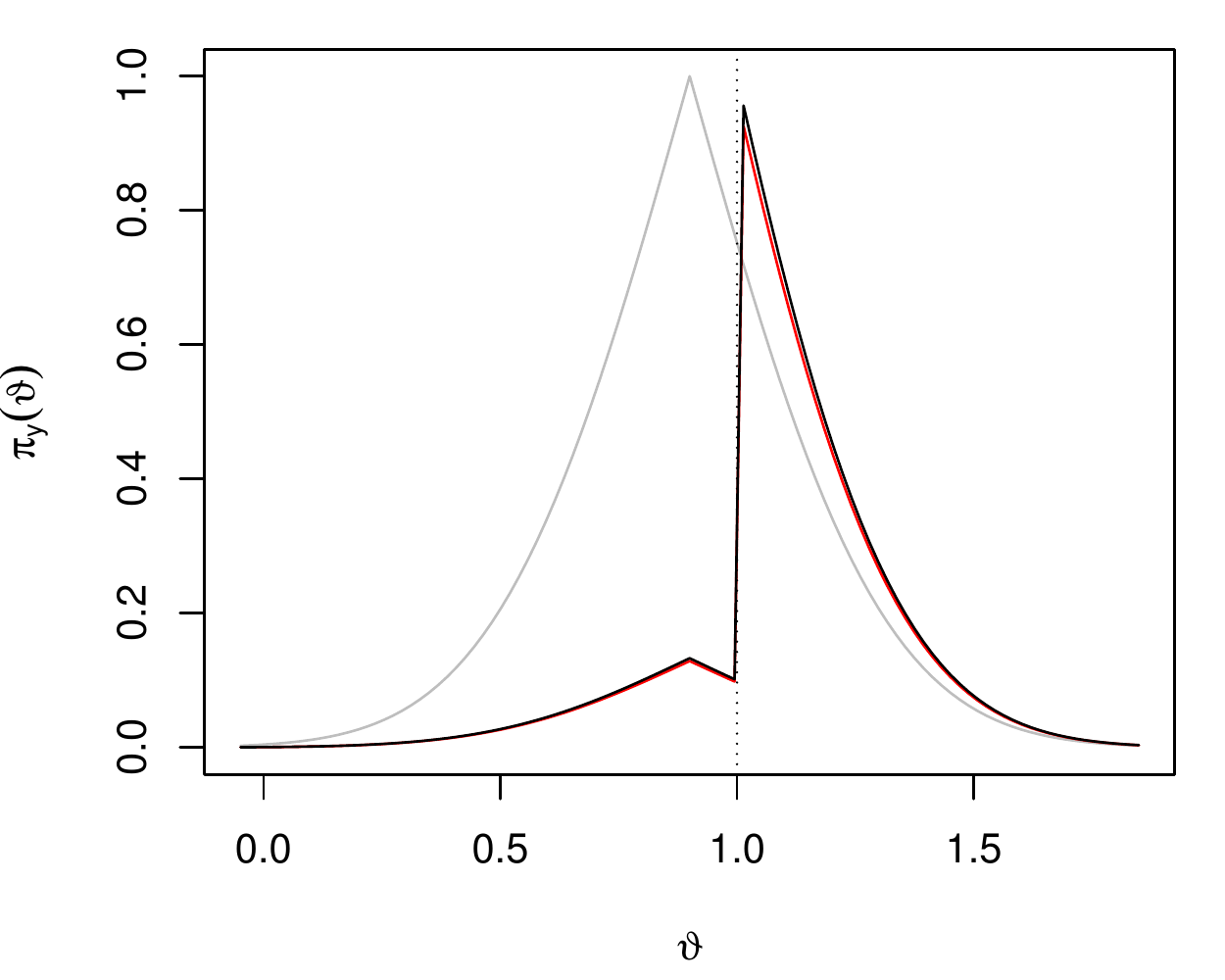}}}
\subfigure[$y=0.5$]{\scalebox{0.6}{\includegraphics{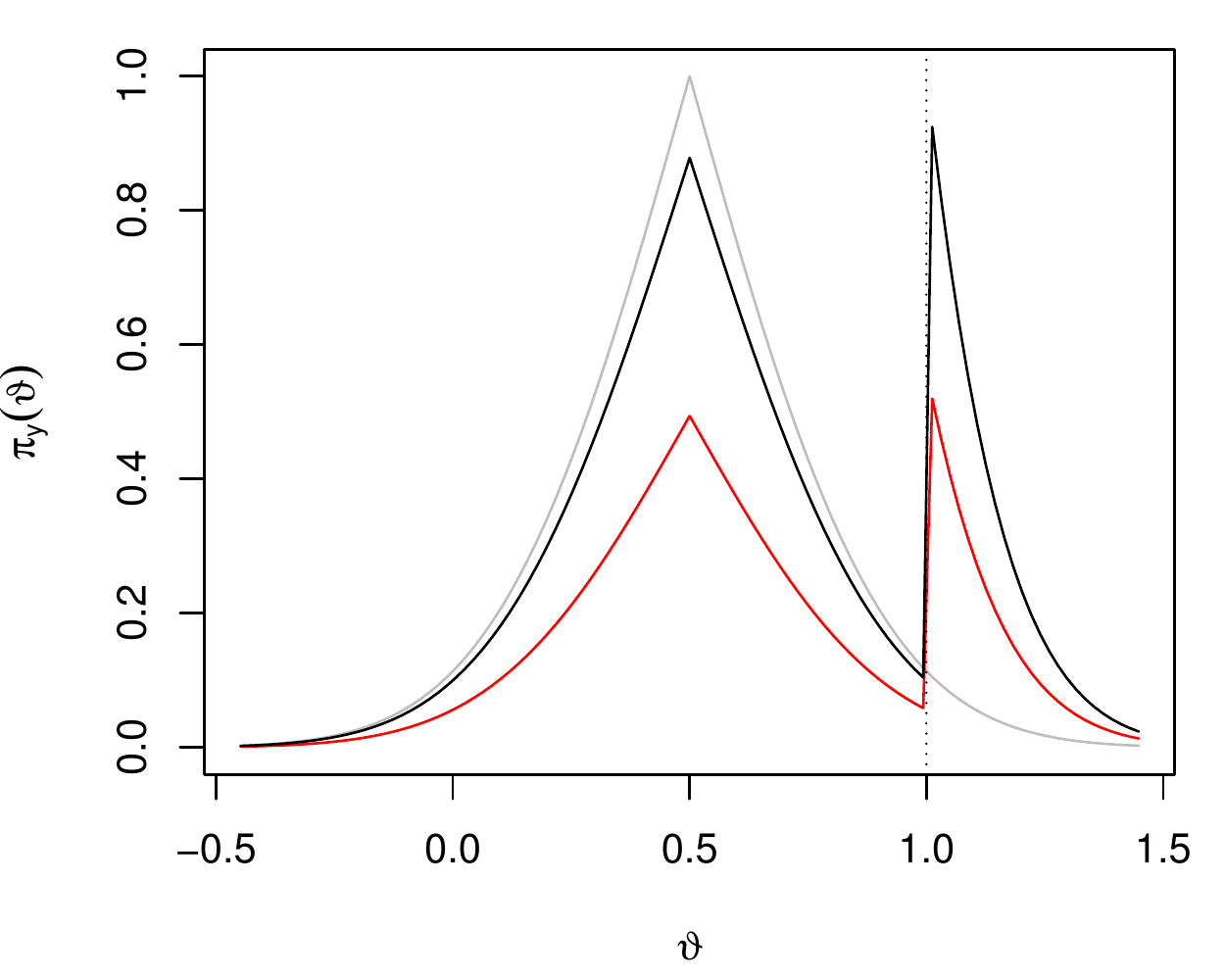}}}
\end{center}
\caption{Plausibility contour plots for the three different IMs discussed in Section~\ref{SS:running2}, for four different values of the data $y$. The gray line is based on the vacuous-prior IM; the black line is based on the consonance-preserving combination rule \eqref{eq:pic.example}; and the red line is based on Dempster's rule in \eqref{eq:pid.example}. In Panels~(c) and (d), that $\pic_{y,\credal} < 1$ is just numerical inaccuracy resulting from the discontinuity at the $a=1$ boundary.}
\label{fig:ab.example2}
\end{figure}

Although validity results for the prior-based IMs are currently lacking---the results in Propositions~\ref{prop:dempster}--\ref{prop:consonant} fall slightly short of the desired validity conclusion---the similarities to the strongly valid vacuous-prior IM in Figure~\ref{fig:ab.example2} suggest that they ought to be ``close to valid'' in some sense.  Next I briefly explore this with a simulation study.  I'll use the same setup as above, i.e., $n=10$, $[a,b] = [1,2]$, and $\beta=0.1$.  Also, $\Theta$ will be drawn from a distribution $\prior^\star \in \credal$ compatible with the information encoded in the prior, namely, 
\[ \prior^\star = \beta \, \delta_{\{0\}} + (1-\beta) \, \unif(a,b). \]
There's nothing special about these choices; I carried out similar simulations under different settings and the results were comparable to those presented here. Figure~\ref{fig:valid.1} shows (Monte Carlo estimates of) the distribution function 
\begin{equation}
\label{eq:cdf.1}
\alpha \mapsto \prob_{Y,\Theta}\{ \pi_{Y,\credal}(\Theta) \leq \alpha\}, \quad \alpha \in [0,1], 
\end{equation}
of the random variable $\pi_{Y,\credal}(\Theta)$ as a function of $(Y,\Theta)$ sampled from the joint distribution described above, for four different IM constructions: the vacuous-prior IM, the one based on information aggregation as in Section~\ref{SS:fusion}, and the ones based on Dempster's and the consonance-preserving rules above.  As the plot shows, the distribution function \eqref{eq:cdf.1} is one/below the diagonal line for all four IMs, which is consistent with the strong validity property.  This property was proved for both the vacuous-prior and information aggregation-based IMs, but not for the Dempster's and consonance-preserving rule-based IMs.  There are two other features of this plot that deserve to be mentioned.  First, the information aggregation-based IM has distribution function generally far less than that of the other IMs, which is consistent with the inefficiency pointed out in Section~\ref{SS:running}.  Second, the distribution functions for the Dempster's and consonance-preserving rule-based IMs are indistinguishable, which may not be unexpected based on the similarities observed in Figure~\ref{fig:ab.example2}.  This is interesting because, like the information aggregation-based IM, that based on Dempster's rule is not guaranteed consonant; apparently, based on the model for $(Y,\Theta)$, the data and prior ``tend to agree,'' so conflict is rare and these two non-consonant IMs are approximately consonant.  Finally, while there seems to be little difference between the vacuous-prior and consonance-preserved IMs in terms of (strong) validity, there is a difference in efficiency as suggested by Figure~\ref{fig:ab.example2}.  Indeed, the vacuous-prior IM's 95\% plausibility intervals have coverage probability and average length $0.95$ and $1.24$, respectively, whereas the IM that incorporates the partial prior information using the consonance-preserving rule has 95\% plausibility intervals with coverage $0.96$ and length $0.93$.  That (roughly) the same coverage is achieved with narrower intervals shows the gain in efficiency by incorporating the prior information.  (Reading off these same intervals from the Dempster's rule-based IM construction wouldn't be appropriate because it's not technically consonant.)

\begin{figure}[t]
\begin{center}
\scalebox{0.7}{\includegraphics{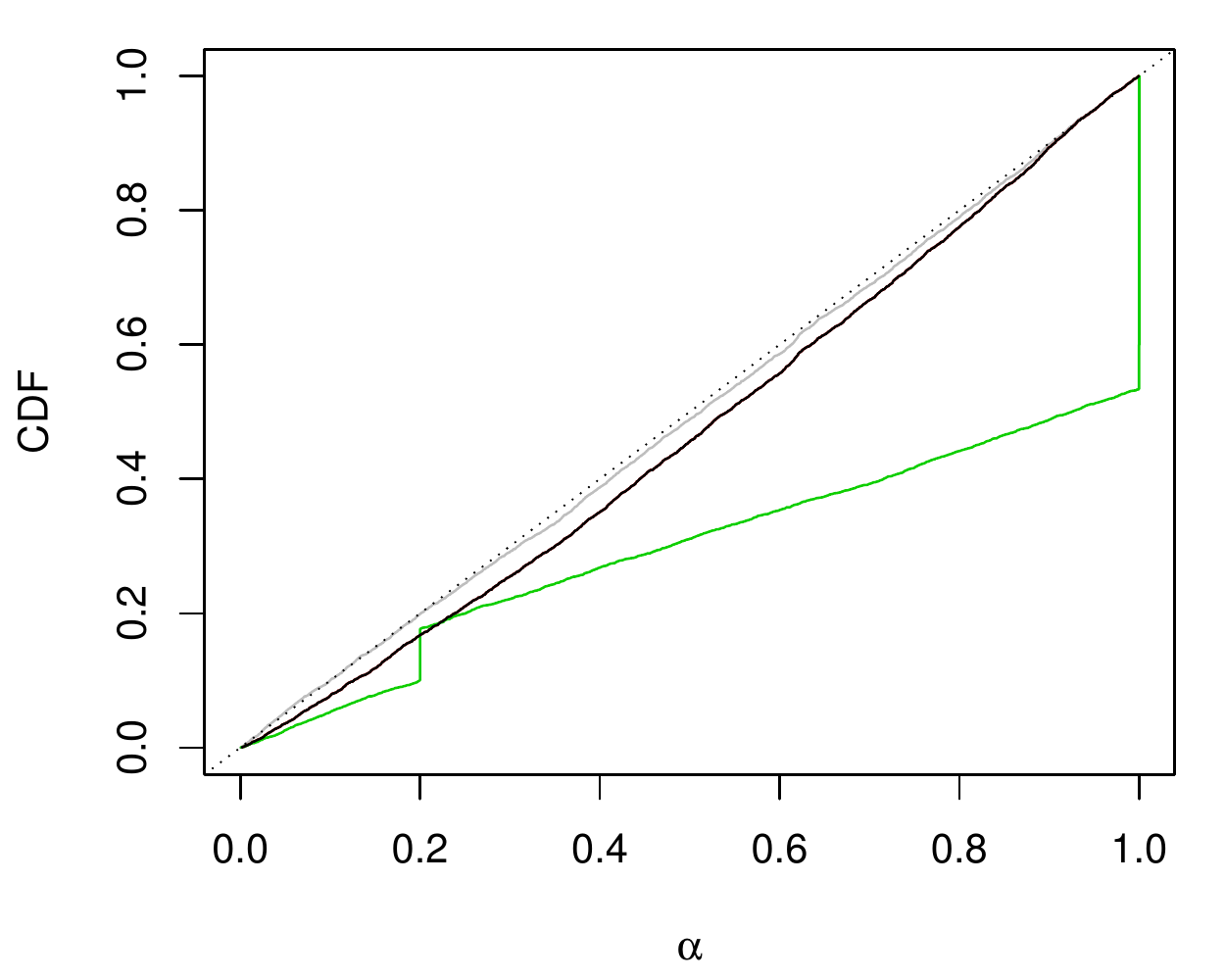}} 
\end{center}
\caption{Monte Carlo estimate of the distribution function \eqref{eq:cdf.1} for four different IMs: vacuous-prior IM (gray); Dempster's rule (red); consonance-preserving rule (black); information aggregation (green). On or below the diagonal line suggests strong validity.}
\label{fig:valid.1}
\end{figure}

\section{Validification}
\label{S:validify}

The strategies discussed in Section~\ref{S:other} seemed promising, so it's disappointing that the desired validity property could not be established for this.  But that's not the end of the story.  It turns out that one can take an IM that is potentially invalid and {\em validify} it, i.e., transform it into a new IM that's provably valid.  This is related to the probability--possibility transforms in \citet{dubois.etal.2004} and developed further in \citet[][Lemma~17]{hose.hanss.2021} and in Chapter~2.3.2.2 of \citet{hose2022thesis}.  The basic idea is that one can, from a given credal set and a partial order on $\TT$, construct a possibility measure whose associated credal set contains the given credal set.  This formulation is general but, in the applications I have in mind, I'm interested in a credal set of distributions for $(Y,\Theta)$, which is determined by the credal set $\credal$ of priors for $\Theta$.  Two key properties of validification are as follows. 
\begin{itemize}
\item I can start with a partial order defined by an IM that may or may not be valid, e.g., that based on the contour of the consonance-preserving combination rule described in Section~\ref{SS:consonant}, and the resulting possibility will be strongly valid in the sense of Definition~\ref{def:strong}.  Hence the name {\em validification}.  
\item If the partial order I start with is the plausibility contour of a strongly valid IM, then the validification procedure returns a new IM that's strongly valid and no less efficient than the one I start with.   
\end{itemize}

The specific construction is as follows.  Let $h_y(\vartheta)$ denote an appropriately measurable real-valued function on $\YY \times \TT$, and define the plausibility contour 
\begin{equation}
\label{eq:transform}
\pi_{y,\credal}(\vartheta) = \uprob_\credal\{ h_Y(\Theta) \leq h_y(\vartheta)\}, \quad \vartheta \in \TT. 
\end{equation}
In principle, no conditions on $h$ are required beyond measurability.  However, the left-hand side of \eqref{eq:transform} may not have the properties of a plausibility contour if $h$ doesn't as well.  Of course, the scale of $h$ isn't important; what's needed is that $\vartheta \mapsto h_y(\vartheta)$ attains its maximum value for each $y$, and that the maximum value is constant in $y$.  For this reason, I'll focus below primarily on cases where $h$ itself is a plausibility contour.  

\begin{prop}
\label{prop:validify}
Fix a mapping $(y,\vartheta) \mapsto h_y(\vartheta)$ as above and define the corresponding plausibility contour $\pi_{y,\credal}$ as in \eqref{eq:transform}.
\begin{enumerate}
\item The consonant IM defined by $\pi_{y,\credal}$ is strongly valid relative to $\credal$ in the sense of Definition~\ref{def:strong}. 
\item If $h_y$ is the plausibility contour of a consonant, strongly valid IM, then $\pi_{y,\credal}$ satisfies $\pi_{y,\credal} \leq h_y$; that is, the consonant IM corresponding to $\pi_{y,\credal}$ is not only strongly valid, but no less efficient than that based on $h_y$. 
\end{enumerate}
\end{prop}

\begin{proof}
To start, let ${\cal H}_\prior(\cdot)$ denote the distribution function of the random variable $h_Y(\Theta)$, as a function of $(Y,\Theta)$ when $\Theta \sim \prior$.  Then 
\[ \pi_{y,\credal}(\vartheta) = \sup_{\prior \in \credal} {\cal H}_\prior\bigl(h_y(\vartheta)\bigr) \geq {\cal H}_\prior\bigl(h_y(\vartheta)\bigr), \quad \text{for all $\prior$}, \]
and ${\cal H}_\prior\bigl( h_Y(\Theta) \bigr)$ is stochastically no smaller than $\unif(0,1)$ under $\prior$.  Therefore, 
\begin{align*}
\uprob_\credal\{ \pi_{Y,\credal}(\Theta) \leq \alpha\} & = \sup_{\prior \in \credal} \prob_{Y,\Theta | \prior}\{ \pi_{Y,\credal}(\Theta) \leq \alpha\} \\
& \leq \sup_{\prior \in \credal} \prob_{Y,\Theta | \prior}\{ {\cal H}_\prior\bigl( h_Y(\Theta) \bigr) \leq \alpha\} \\
& \leq \alpha,
\end{align*}
which proves the strong validity claim in Part~1. Next, it's easy to check that 
\begin{align*}
\pi_{y,\credal}(\vartheta) & = \uprob_\credal\{ h_Y(\Theta) \leq H_y(\vartheta)\} \\
& = \sup_{\prior \in \credal} \int \prob_{Y|\theta}\{h_Y(\theta) \leq h_y(\vartheta)\} \, \prior(d\theta) \\
& \leq \sup_{\prior \in \credal} \int h_y(\vartheta) \, \prior(d\theta) \\
& = h_y(\vartheta),
\end{align*}
where the inequality is by the assumed strong validity. This proves Part~2.
\end{proof}

There is an interesting comparison between Proposition~\ref{prop:validify}.2 and the famous Rao--Blackwell theorem in classical statistics.  The latter says if you start with an unbiased estimator, then you can Rao--Blackwellize it (conditional expectation given a sufficient statistic) to obtain a new unbiased estimator that's no less efficient in the sense that its variance is no larger than that of the given unbiased estimator.  In the present case, Proposition~\ref{prop:validify} says that if my partial ordering is determined by an $h_y$ that's already a consonant IM's plausibility contour, then \eqref{eq:transform} defines the plausibility contour for a new consonant IM that's valid no less efficient than the IM I started with.  

There are a number of ways that one could apply the above result in the present context.  First, I'd hoped that this would provide a general strategy for combining the vacuous-prior IM with the imprecise prior information in $\credal$.  That is, take $h_y = \pi_y$ to be the vacuous-prior IM and then construct a prior-dependent $\pi_{y,\credal}$ according to \eqref{eq:transform}.  By Proposition~\ref{prop:validify}.2, this is no less efficient than the vacuous-prior IM, so this idea initially seemed quite promising.  However, there's an interesting invariance property that prevents improvements of IMs that already have certain validity properties.  To illustrate this, suppose that the vacuous-prior IM is {\em exactly valid} in the sense that 
\[ \prob_{Y|\theta}\{ \pi_Y(\theta) \leq \alpha\} = \alpha \quad \text{for all $\alpha$ and all $\theta$}. \]
This is often the case; the vacuous-prior IM in the normal examples here are exactly valid, as it is for most of the examples in \citet{imbook}. If I apply the transformation \eqref{eq:transform} with $h_y = \pi_y$, then I get $\pi_{y,\credal} = \pi_y$. Hence the transform \eqref{eq:transform} is invariant on the collection of consonant IMs that are exactly valid and, hence, can't be used as a general rule for combing a vacuous-prior IM with a credal set to gain efficiency.  

Second, and even more exciting, is a potential resurrection of the IM based on the consonance-preserving combination rule described in Section~\ref{SS:consonant}.  I was unable to prove a general result about validity of this IM, but suppose now I take $h_y$ in the above construction to be the plausibility contour \eqref{eq:im.tnorm}, i.e., 
\begin{equation}
\label{eq:special.h}
h_y(\vartheta) = \frac{\pi_y(\vartheta) \star q(\vartheta)}{\sup_{t \in \TT} \{ \pi_y(t) \star q(t)\}}, \quad \vartheta \in \TT, 
\end{equation}
where $\pi_y$ is the vacuous-prior IM's contour, $q$ is the prior contour, and $\star$ is a t-norm.  The corresponding validified IM would be consonant, would closely resemble that constructed in Section~\ref{SS:consonant}, and---most importantly---would be valid, by Proposition~\ref{prop:validify}.1.  

\begin{cor}
\label{cor:validify.consonant}
The consonant IM defined by \eqref{eq:transform} with the $h$-function given by \eqref{eq:special.h} is a strongly valid version of the consonant IM defined in Section~\ref{SS:consonant}.  And if it turns out that the latter itself is strongly valid, then the former is no less efficient. 
\end{cor}

This validified version of the consonant IM developed in Section~\ref{SS:consonant}, as described above, seems quite promising.  The only downside I can see is computation.  That is, how can I actually compute the transformation \eqref{eq:transform} in cases where the prior information encoded in $\credal$ is non-trivial?  General solutions to this computational problem are beyond the scope of the present paper.  However, at least when $\credal$ is the credal set corresponding to a possibility measure with contour $q$ as discussed in Section~\ref{S:other}, things can be simplified.  In this case, the right-hand side of \eqref{eq:transform} is a Choquet integral, so I can use the formula in, e.g., Proposition~7.14 of \citet{lower.previsions.book}, for possibility measures:
\[ \pi_{y,\credal}(\vartheta) = \int_0^1 \Bigl[ \sup_{\theta: q(\theta) > \alpha} \prob_{Y|\theta}\{h_Y(\theta) \leq h_y(\vartheta)\} \Bigr] \, d\alpha. \]
I can carry out a naive Monte Carlo strategy to evaluate this integral for the simple running example of Sections~\ref{SS:running} and \ref{SS:running2}.  Figure~\ref{fig:ab.example3} shows plots of the plausibility contours for the vacuous-prior IM, $\pi_y$, the version $\pic_{y,\credal}$ in \eqref{eq:pic.example} based on the consonance-preserving combination rule in Section~\ref{SS:consonant}, and the validified version of $\pic_{y,\credal}$ with $h$ as in \eqref{eq:special.h}.  The validified version is evaluated based on Monte Carlo so please forgive the bit of wiggliness as a result.  There are several interesting observations in this plot.  First, as expected, $\pic_{y,\credal}$ and its validified version closely match except in a few places.  These differences are not the result of Monte Carlo error---we have reason to doubt that $\pic_{y,\credal}$ is strongly valid, so the validified version should be larger at least in some areas.  Second, despite being slightly larger than $\pic_{y,\credal}$ in certain places, its contours are still narrower than the vacuous-prior IM's, so apparently the desired gain in efficiency has been preserved.  Albeit brief and limited, this numerical illustration confirms my previous claim that this general strategy of validifying the consonant IM developed in Section~\ref{SS:consonant} is quite promising.  

\begin{figure}[t]
\begin{center}
\subfigure[$y=0.9$]{\scalebox{0.6}{\includegraphics{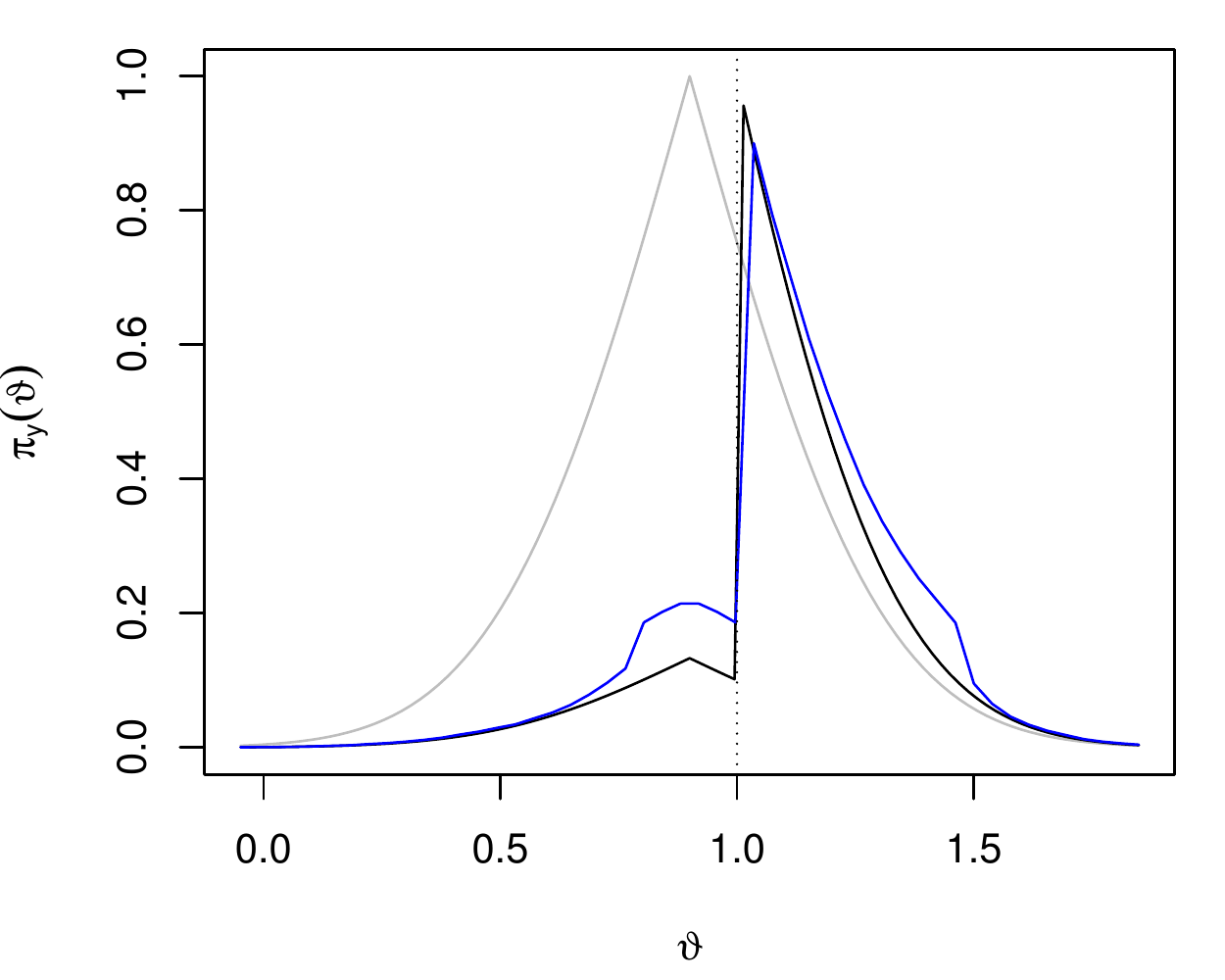}}}
\subfigure[$y=1.1$]{\scalebox{0.6}{\includegraphics{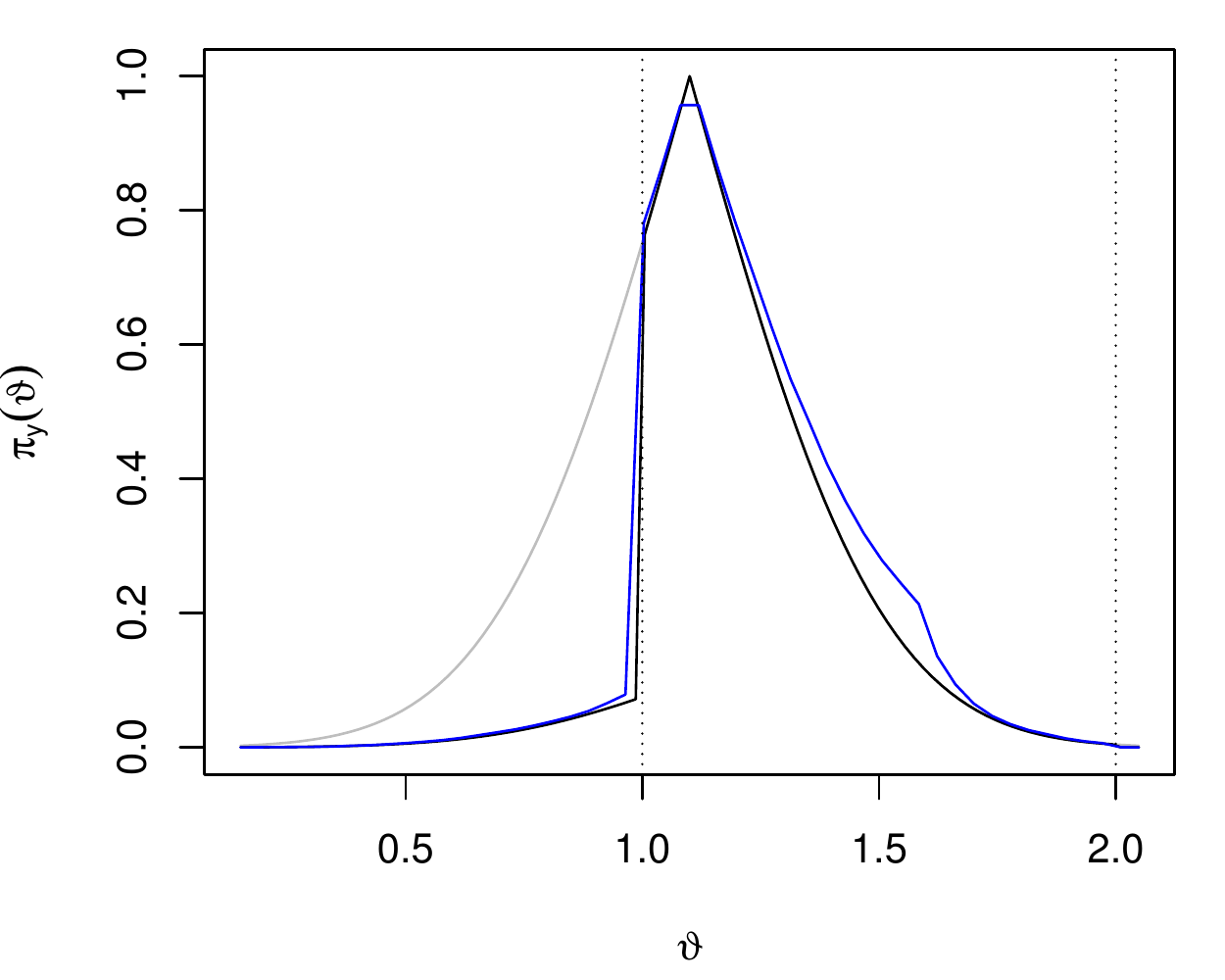}}}
\end{center}
\caption{Plausibility contours for three different IMs in the running example from Section~\ref{SS:running}. The gray line $\pi_y$ based on the vacuous-prior IM; the black line is $\pic_{y,\credal}$ based on the consonance-preserving combination rule \eqref{eq:pic.example}; and the blue line is the validified version of the black line. Both the black and blue curves attain the value 1; that they appear to fall short in the plot is due to numerical imprecision.}
\label{fig:ab.example3}
\end{figure}

There's another interesting perspective on the validification rule that connects to some notions that might be familiar to readers in the imprecise probability community.  In particular, the output of the validification rule is effectively the contour function of an outer consonant approximation to the upper probability $\uprob_\credal$.  I'll save the details about this characterization until Part~II of the series.

\section{Sparse normal mean problem}
\label{S:sparsity}

\subsection{Setup and motivation}

Let $Y \sim \nm_n(\theta, I_n)$ be an $n$-dimensional normal random vector with covariance function known and equal to the identity matrix but with mean vector $\theta \in \TT = \RR^n$ unknown.  As mentioned in Section~\ref{S:intro}, it is common that $\theta$ satisfies certain constraints, especially when $n$ is large.  These constraints impose an inherent low-dimensional structure on $\theta$.  One typical constraint is {\em sparsity}, which here I'll take to mean that most of the coordinates of $\theta$ are 0 and only a few are non-zero.  However, it remains unknown how many are 0, which ones are 0, and what values the non-zero coordinates take.  The goal, of course, is inference on $\theta$, but taking into account the low-dimensional constraints.  

My primary and original motivation for this investigation into the use of partial priors was that statements like ``$\theta$ has an inherent low-dimensional structure'' are imprecise, so quantifying these structural assumptions as imprecise probabilities or partial priors seems quite natural.  But there's another aspect of the developments above that's important in this general class of problems, one that I hadn't anticipated.  In the high-dimensional statistics literature, there are impossibility theorems \citep[e.g.,][]{li1989} that say, roughly, there are no confidence balls for the vector $\theta$ that both (a)~achieve the nominal coverage probability uniformly over $\TT$ and (b)~are not too large.  Here, by ``not too large'' I mean the radius is roughly equal to the minimax rate for estimation of $\theta$ in the class having the posited low-dimensional structure.  For the sparsity structure, this minimax rate is $\{o(n) \log n\}^{1/2}$ as $n \to \infty$ \citep[e.g.,][]{donohojohnstone1994b}.  Details aside, what this means is that, for high-dimensional problems with low-dimensional structure, there's a fundamental incompatibility between the vacuous-prior notion of validity and the efficiency I hope to achieve.  For example, the vacuous-prior IM for the normal means problem is (strongly) valid, so its plausibility regions achieve the nominal coverage probability uniformly over $\TT$.  However, those plausibility balls have radius of the order $n^{1/2}$, which is substantially larger than $\{o(n) \log n\}^{1/2}$, hence inefficient.  Since there are theorems saying it's impossible to close this gap, I need a weaker notion of validity that accommodates the available partial prior information.  It's still an open question whether the relative-to-partial-priors notion of validity presented in Definition~\ref{def:valid} can meet this need, but I'm optimistic.  

A complete solution to this problem is still in the works.  My very modest goal here in this short section is to show, first, how the sparsity structure can be encoded as a partial prior and, second, the potential for efficiency gains (compared to a naive vacuous-prior solution) using some of the ideas developed above.

\subsection{Sparsity prior IM constructions}

On $\TT = \RR^n$, define the ``$\ell_0$-norm'' as $\|\vartheta\|_0 = |\{i: \vartheta_i \neq 0\}|$, the number of entries in $\vartheta$ that are non-zero.  Believing that $\theta$ is sparse is equivalent to believing that $\|\theta\|_0$ is small.  My proposal here is to use a possibility measure to encode the partial prior belief that $\|\theta\|_0$ is small.  For this, define the prior plausibility contour 
\[ q(\vartheta) = 1 - F_{n,\varpi}(\|\vartheta\|_0 - 1), \]
where $F_{n,\omega}$ is the $\bin(n, \varpi)$ distribution function, with $\varpi \in (0,1)$ a hyperparameter that must be specified by the data analyst.  For example, perhaps the data analyst feels comfortable to say that only about 10\% of the entries in $\theta$ are non-zero, in which case he'd set $\varpi=0.1$.  I'm not suggesting this as a ``recommended''\footnote{The advantage of the framework developed in this paper is that users are free to construct their partial priors as they please, perhaps very different from my example here; the IM they get by following, say, the validification strategy described in Section~\ref{S:validify} is sure to be valid with respect to their partial prior specifications.} prior that someone ought to use for their application, it's just an example. Several key properties are encoded in the above-constructed $q$:
\begin{itemize}
\item the origin $\vartheta=0$ is perfectly plausible;
\vspace{-2mm}
\item $q(\vartheta)$ is decreasing in $\|\vartheta\|_0$; 
\vspace{-2mm}
\item and, $q$ is constant on sets with the same number of non-zero coordinates, i.e., $q$ is vacuous on the value of the non-zero entries. 
\end{itemize}

It's also straightforward to get a valid vacuous-prior IM.  Using effectively the same logic in the construction of the vacuous-prior IM in Section~\ref{SS:running}, we find that the vacuous-prior IM's plausibility contour is given by 
\[ \pi_y(\vartheta) = 1 - G_n(\|y - \vartheta\|^2), \]
where $G_n$ is a $\chisq(n)$ distribution function.  It's easy to verify that this is strongly valid in the sense of Definition~\ref{def:strong}.  

To get a partial prior-dependent IM for this sparse normal mean problem, one option is the construction in Section~\ref{SS:consonant}.  That is, the IM's plausibility contour is $\pi_{y,\credal}$ in \eqref{eq:im.tnorm} where the t-norm $\star$ is multiplication.  Fortunately, computation of the normalizing constant can be done almost entirely in closed-form for this problem.  Here's a brief explanation of how this calculation goes:
\begin{enumerate}
\item $q(\vartheta)$ is a constant value, say $q_k$, on sets where $\|\vartheta\|_0 = k$. 
\vspace{-2mm}
\item Let $\hat\theta^k$ denote the maximizer of the vacuous-prior IM contour $\pi_y(\vartheta)$ on the set $\|\vartheta\|_0 = k$.  First rank the $|y|$ values in descending order, i.e., $|y|_{i_1} > \cdots > |y|_{i_n}$.  Then set $\hat\theta_{i_j}^k = y_{i_j}$ for $j=1,\ldots,k$ and all the remaining $\hat\theta_i^k$'s equal to 0.
\vspace{-2mm}
\item Then $\sup_t\{ \pi_y(t) \, q(t)\} = \max_k \{\pi_y(\hat\theta^k) \, q_k\}$, which is easy to calculate.
\end{enumerate} 
Currently there is no fully satisfactory proof of validity for this  IM and, in particular, I have reasons (empirical and otherwise) to doubt that this IM is strongly valid.  It's possible to take this consonant but-probably-not-valid IM contour as input to the validification strategy described in Section~\ref{S:validify} above.  Computation of the validified IM is, as usual, non-trivial, but a naive Monte Carlo solution is within reach.  

Define this validified version as 
\[ \pi_{y,\credal}^\text{\sc v}(\vartheta) = \int_0^1 \sup_{\theta: q(\theta) > \alpha} \prob_{Y|\theta}\{ \pi_{Y,\credal}(\theta) \leq \pi_{y,\credal}(\vartheta)\} \, d\alpha, \]
where the right-hand side is the Choquet integral formula.  The challenge, of course, is the evaluation of the function in the integrand,
\[ \alpha \mapsto \sup_{\theta: q(\theta) > \alpha} \prob_{Y|\theta}\{ \pi_{Y,\credal}(\theta) \leq x\}, \quad \text{any fixed $x \in [0,1]$}. \]
Since $\pi_{Y,\credal}(\theta)$ can be easily evaluated using the three-step strategy above, the distribution function part of this computation isn't difficult.  The supremum is what creates problems.  The key observation, however, is that the random variable $\pi_{Y,\credal}(\theta)$, as a function of $Y \sim \prob_{Y|\theta}$, for $\theta$ such that $\|\theta\|_0 = k$, is stochastically the smallest when the $k$ non-zero entries of $\theta$ are very close to 0.  Moreover, it doesn't matter which of the entries are the small non-zero value, thanks to symmetry of the normal distribution.  These observations, of course, are specific to the present application, and wouldn't hold in general.  For the $n=2$ case considered below, the validified plausibility contour is evaluated as 
\begin{align*}
\pi_{y,\credal}^\text{\sc v}(\vartheta) & = (1-q_1) \prob_{Y|v_0} \{ \pi_{Y,\credal}(v_0) \leq \pi_{y,\credal}(\vartheta)\} \\
& \qquad + (q_1 - q_2) \prob_{Y|v_1}\{\pi_{Y,\credal}(v_1) \leq \pi_{y,\credal}(\vartheta)\} \\ 
& \qquad + q_2 \prob_{Y|v_2}\{\pi_{Y,\credal}(v_2) \leq \pi_{y,\credal}(\vartheta)\}, 
\end{align*}
where $v_0 = (0,0)$, $v_1 = (0, 0.01)$, and $v_2 = (0.01, 0.01)$, $q_k = q(v_k)$ for $k=1,2$, and the $\prob_{Y|\theta}$-probabilities are evaluated via simple Monte Carlo.

\subsection{Visualizing potential efficiency gain}

I said above that my goal in this section is a modest one.  All I want to do here is demonstrate visually the potential for efficiency gain by incorporating the partial prior described above.  For this illustration, I take $n=2$ and suppose that I've observed $y=(1,0.3)$.  In this case, since $y_2$ is relatively close to 0, the data is not incompatible with $\theta_2 = 0$.  When I incorporate the sparsity-encouraging partial prior, the resulting combined IM will push towards this simpler structure with $\theta_2=0$.  Figure~\ref{fig:sparse.contour} shows the contour plots of the plausibility contours for the vacuous-prior IM, $\pi_y$, the sparsity-encouraging prior IM, $\pi_{y,\credal}$, and its validified version $\pi_{y,\credal}^\text{\sc v}$.  In particular, the 90\% plausibility regions are highlighted in red.  The contour plot in Panels~(b) and (c) look a bit strange due to the structure in $q$.  What needs to be clarified is that the 90\% plausibility region includes line segments along the axes protruding from the ordinary-looking contours.  The key point is that the vacuous-prior IM's plausibility region is substantially larger than that of the sparsity-encouraging prior-based IMs, a direct effect of acknowledging the underlying low-dimensional structure.  It's also worth noting that the contour plots in Panels~(b) and (c) are almost identical, the only real difference is theoretical---the validified version is provably valid while the other isn't.  But it'd be worth exploring the proximity of $\pi_{y,\credal}$ to $\pi_{y,\credal}^\text{\sc v}$ in general since the former is much easier to compute than the latter.  

\begin{figure}
\begin{center}
\subfigure[vacuous-prior IM]{\scalebox{0.6}{\includegraphics{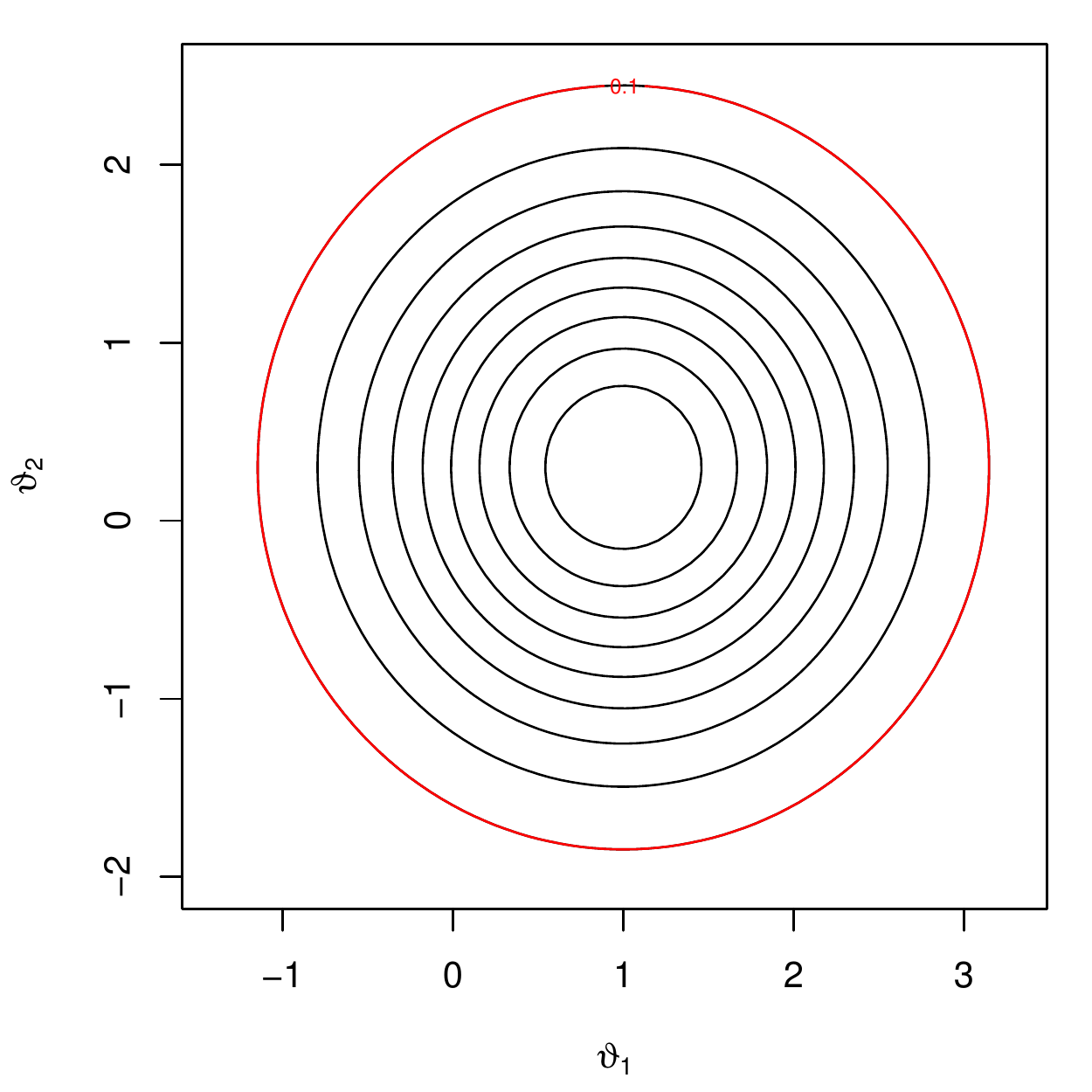}}}
\subfigure[sparsity prior consonant IM]{\scalebox{0.6}{\includegraphics{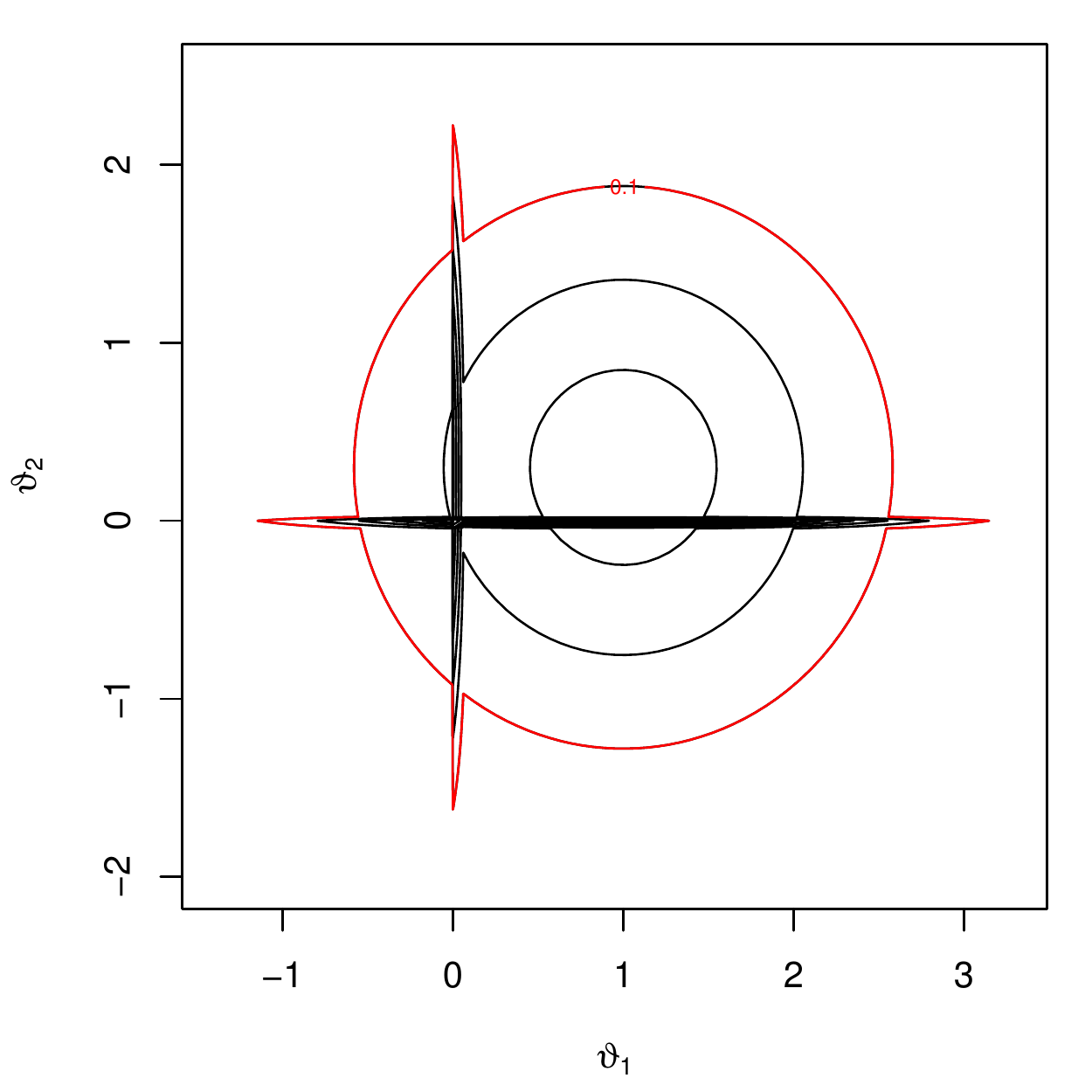}}}
\subfigure[validified sparsity prior consonant IM]{\scalebox{0.6}{\includegraphics{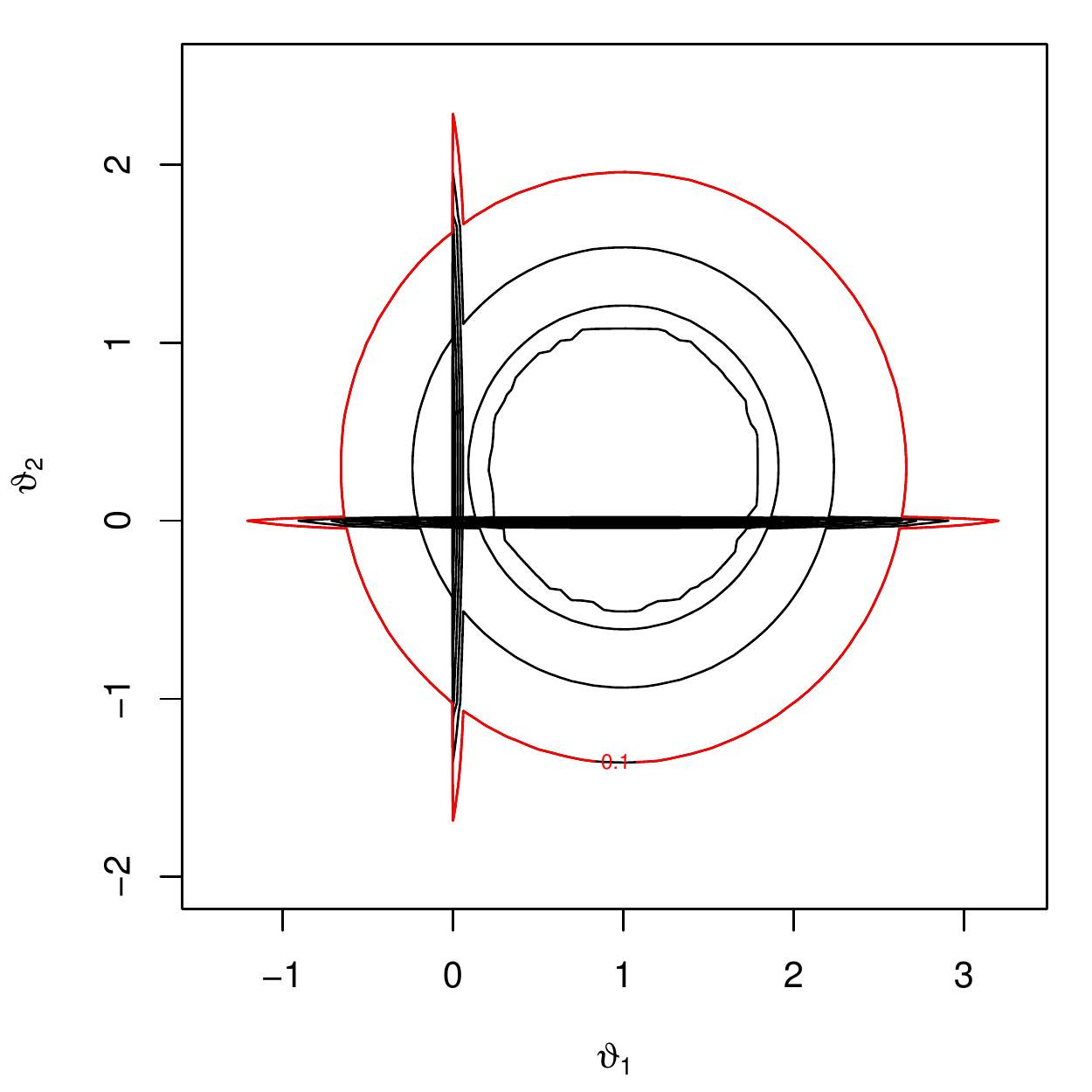}}}
\end{center}
\caption{Plots of the plausibility contour functions for $\theta=(\theta_1, \theta_2)$ based on the (a)~vacuous-prior IM, (b)~sparsity-prior consonant IM, and (c)~validified version of the IM in (b) with data $y=(1,0.3)$ and prior parameter $\varpi=0.5$. The corresponding 90\% plausibility regions are highlighted in red.}
\label{fig:sparse.contour}
\end{figure}




\section{Conclusion}
\label{S:discuss}

In this paper, I've considered a more general notion of validity for inferential models, one that covers the spectrum from no prior information, to partial prior information in the form of imprecise probabilities, to complete prior information in the form of a precise probability.  This is a substantial extension of the theory previously available in \citet{martin.nonadditive, imchar}.  I've also shown that achieving this more general notion of validity has a number of interesting and important statistical and behavioral consequences.  There are big-picture consequences thanks to the resolution of the Bayesian vs.~frequentist, two-theory problem in statistics.  But the definition and consequences are only meaningful if there's an IM construction that achieves it, and several different strategies are presented in Sections~\ref{S:achieving} and \ref{S:other}.  There are, unfortunately, two challenges.  First, the constructions for which validity is directly verifiable turn out to be unsatisfactory---they're generally less efficient than the basic, vacuous-prior IM that ignores the structural simplifications provided by the partial prior.  Second, for the constructions that tend to show a gain inefficiency compared to the vacuous-prior IM, only limited validity properties are currently available. Hope is not all lost, however, since the {\em validification} strategy presented in Section~\ref{S:validify} can be employed to take those efficient but not-currently-provably-valid and make them valid without sacrificing on efficiency.  At the end of the day, the partial prior-dependent IM that I'd currently recommend is that based on the consonance-preserving combination rule in Section~\ref{SS:consonant}, with the validification applied afterward for good measure.  But new and improved valid IM constructions have come to light.  In particular, Dominik Hose suggested to me a very general strategy for construction a (strongly) valid partial-prior-dependent IM: take the generator function $h_y$ in \eqref{eq:transform} to be  
\[ h_y(\vartheta) = L_y(\vartheta) \, q(\vartheta) / \textstyle\sup_{t \in \TT}\{ L_y(t) \, q(t)\}, \]
superficially Bayesian.  This turns out to be a powerful suggestion, with so much potential that it'll be the focus of Part~II of the series. 

There are still a number of interesting and important open question to be answered, and here I'll list a few.
\begin{itemize}
\item Strong validity for the two constructions presented in Section~\ref{S:other} may be out of reach, but ordinary validity seems possible.  If not, then it must be that the restriction \eqref{eq:dempster.restrict} of the candidate assertions is especially important or fundamental, so it'd be interesting to understand that better. 
\vspace{-2mm} 
\item The validification strategy is conceptually straightforward but computationally non-trivial.  I was able to handle this using a naive Monte Carlo strategy for the very special example in Section~\ref{SS:running}, but other examples would surely require more care.  If validification turns out to be important---that is, if the procedure in Section~\ref{SS:consonant} actually isn't valid---then general strategies for carrying out these computations would be of great practical importance.  
\vspace{-2mm} 
\item My recent work \citep{imdec} on decision-theoretic implications of validity focused exclusively on the vacuous-prior case.  Like in classical decision theory, there are limitations to what can be said when there's no prior information available.  For example, admissibility of ``optimal'' IM-based decisions is not guaranteed in the vacuous-prior case. However, I expect that the developments here in this paper on IMs with partial prior information can be leveraged to prove stronger decision-theoretic properties, e.g., I'm optimistic that results paralleling the classical admissibility of Bayes rules can be established in this IM framework. 
\vspace{-2mm} 
\item The result in Proposition~\ref{prop:dumb} provides a simple procedure for constructing tests with error rate control guarantees relative to the partial prior information.  This deserves a more thorough investigation, since, at the very least, it serves as a baseline, nominal level test against which other methods can be compared.  Furthermore, an interesting feature of this test is that it can discount an assertion that's supported by the data and not by the prior, but it can't mark up an assertion that's supported by the prior but not by the data.  As \citet{baker.truth} says ``Data doesn’t lie.~People do,''~so it's reasonable to not let the prior make an assertion that's incompatible with the data appear more plausible. 
\vspace{-2mm} 
\item The high-dimensional problems briefly discussed in Section~\ref{S:sparsity} make up an important class of applications of the new ideas here.  The consonant IM construction presented in Section~\ref{SS:consonant}, and its validified version, seems quite promising.  Indeed, at least superficially, multiplying the vacuous-prior IM plausibility function by the prior plausibility function closely resembles the Bayesian's likelihood-times-prior regularization.  Whether this particular construction is the best possible is, in my opinion, the most important of these open problems.  
\end{itemize}

The focus in this and many of my other recent papers on the topic has been on theory, but practical applications are important too.  While vacuous-prior IM solutions have been worked out for most of the standard textbook examples, there are important contributions there still to be made.  In particular, problems involving discrete data, e.g., contingency table data, are practically important and inherently more challenging.  To my knowledge, no satisfactory IM-style solution to the contingency table data problem is available.  That's one of the next tasks on my to-do list.  

I understand that the details I've presented here and elsewhere are unfamiliar, and that a lack of familiarity can be unsettling.  
Similarly, the two-theory problem and the negative impacts it has on statistics and science more generally are unsettling.  Rather than be discouraged by the difficulty of these problems or the unfamiliarity of the methods being described, I hope the reader will be encouraged by the opportunities these unsettling feelings create.  Let's ``settle the unsettling'' \citep{gong.meng.discuss}.

\section*{Acknowledgments}

This work is partially supported by the U.S.~National Science Foundation, grants DMS--1811802 and SES--2051225. 
Thanks to Michael Balch, Leonardo Cella, Harry Crane, Jasper De Bock, Dominik Hose (especially), Ruobin Gong, and Matthias Troffaes for helpful discussions, and to members of the {\em BELIEF 2021} program committee for their valuable feedback on a previous version of the manuscript.

\bibliographystyle{apalike}
\bibliography{/Users/rgmarti3/Dropbox/Research/mybib}

\end{document}